\documentclass[preprint,12pt]{elsarticle}
\makeatletter
\def\ps@pprintTitle{
  \let\@oddhead\@empty
  \let\@evenhead\@empty
  \let\@oddfoot\@empty
  \let\@evenfoot\@oddfoot
}

\makeatletter
\renewcommand{\theaffn}{
  \ifcase\c@affn 
    
  \or $\dagger$   
  \or $\ddagger$ 
  \fi
}
\makeatother

\usepackage[a4paper, left=0.8in, right=0.8in, top=0.8in, bottom=0.8in]{geometry}
\usepackage{amssymb}
\usepackage{paracol} 
\usepackage{amsthm}
\usepackage{amsmath}
\usepackage[numbers]{natbib}  
\usepackage{xcolor} 
\usepackage[linesnumbered,ruled,vlined]{algorithm2e}
\usepackage{algpseudocode}  
\usepackage{booktabs}
\usepackage{multirow}
\usepackage{array}
\usepackage{makecell}
\usepackage{enumitem}
\usepackage{mathrsfs}
\usepackage{graphicx}
\usepackage{subcaption}
\usepackage{tikz}
\usetikzlibrary{positioning, arrows.meta, shapes.geometric}

\newtheorem{theorem}{Theorem}[section]

\theoremstyle{definition}

\newtheorem{example}[theorem]{Example}
\newtheorem{remark}[theorem]{Remark}
\newtheoremstyle{noteStyle}
  {}{}{\normalfont}{}{\bfseries}{.}{ }{}
\theoremstyle{noteStyle}
\newtheorem{note}{Note}

\begin{document}

 \begin{frontmatter}

\title{
ENPINN: Energy-Norm-Guided Gradient-Enhanced PINNs for Generalized Transport Problems with Sharp Gradients}

\author[label1]{Subhendu Maity} 
\ead{subhendu_2421ma12@iitp.ac.in}

\author[label1]{Pratibhamoy Das}
\ead{pratibhamoy86math@gmail.com, pratibhamoy@iitp.ac.in}

\author[label1]{Arihant Patawari} 
\ead{arihant_2221ma09@iitp.ac.in}

\author[label2]{Ameya D. Jagtap} 
\ead{ajagtap@wpi.edu}

\affiliation[label1]{organization={Department of Mathematics, Indian Institute of Technology, Patna},
             city={Bihar},
             postcode={801103}, 
            country={India}}

\affiliation[label2]{organization={Aerospace Engineering Department, Worcester Polytechnic Institute, Worcester, MA 01609},{USA}}

\begin{abstract}
    {Physics-informed neural networks (PINNs) have emerged as a meshless alternative to conventional numerical methods for solving partial differential equations (PDEs). However, their limited ability to capture sharp gradients can lead to substantial errors when resolving boundary and interior layers. Here, we introduce an energy-norm-enhanced PINN (ENPINN) that incorporates gradient information and variational structure into the loss function to improve the resolution of layer-dominated solutions. We first examine two related formulations: weak-loss PINNs (WLPINNs), which incorporate test functions into the conventional PINN residual, and gradient-enhanced PINNs (gPINNs), which augment the loss with spatial derivatives of the PDE residual. By analyzing these formulations, we identify their limitations in resolving steep solution gradients and motivate the systematic construction of ENPINN. We establish theoretically how the energy-norm error depends on the ENPINN loss and show that a suitably modified residual-derivative term is essential for accurately capturing boundary layers. We further establish the existence of neural-network approximations with arbitrarily small energy error and derive corresponding derivative bounds, providing a theoretical foundation for the proposed framework. The performance of ENPINN is assessed through systematic comparisons with existing PINN variants for convection–diffusion–reaction problems exhibiting steep gradients. Numerical experiments include a combustion model, a coupled multiscale system, a two-dimensional Burgers equation with an interior layer, and a three-dimensional time-dependent problem. Across these examples, ENPINN consistently achieves low energy-norm errors, demonstrating its ability to resolve sharp gradients in convection- and reaction-dominated regimes. These results highlight the potential of combining variational test functions with gradient-based residual information in the development of reliable neural PDE solvers.

\vspace{0.2cm}
    }
\end{abstract}

\end{frontmatter}

\section{Introduction}\label{sec:introduction}
Scientific deep learning (SciDL) has emerged as a powerful framework for scientific computing, enabling data-driven and physics-constrained approaches across a broad range of problems in science and engineering. Applications span high-speed flows \cite{mao2020physics,jagtap2022physics,abbasi2025challenges}, turbulent, multiphase and combustion flows \cite{menon2026intelligent}, magnetohydrodynamic flows \cite{patawari2025modified}, Riemann problems \cite{peyvan2024riemannonets}, stiff chemical kinetics \cite{goswami2024learning}, fractured porous media \cite{abbasi2025history}, option pricing \cite{patawari2025traditional}, material identification \cite{shukla2021physics}, and nonlinear water waves \cite{jagtap2022deepD}, among many others. A key attraction of SciDL is its mesh-free formulation, which provides an alternative to conventional numerical discretizations for partial differential equations (PDEs). This flexibility is enabled, in part, by the universal approximation theorem (UAT) \cite{uat_cybenko_1989}, which establishes the capacity of neural networks to approximate broad classes of functions.

Within this broader SciDL framework, the incorporation of governing physical laws into neural-network training has given rise to physics-informed learning. Physics-informed neural networks (PINNs) \cite{pinn_2018_indroce} represent one of the most prominent formulations, in which the governing PDEs are embedded directly into the training objective and enforced at a set of collocation points. By combining physical constraints with neural-network representations, PINNs can approximate solutions without requiring labelled solution data and have been successfully applied to diverse classes of PDEs \cite{menon2025anant,menon2026scientific}. Their development has subsequently motivated a range of strategies to improve scalability, accuracy and trainability. Conservative PINNs (cPINNs) \cite{cpinn_intro}, for example, introduced domain decomposition for nonlinear conservation laws, whereas extended PINNs (XPINNs) \cite{Xpinn} proposed a generalized space–time domain-decomposition framework to broader classes of PDEs and enabled parallel training \cite{shukla2021parallel,hu2021extended}. Subsequent formulations have further incorporated temporal decomposition \cite{penwarden2023unified} and augmented representations \cite{hu2023augmented}. In parallel, adaptive activation functions have been developed to enhance the expressivity and trainability of PINNs \cite{jagtap2020adaptive,jagtap2020locally,jagtap2023important,jagtap2022deep}. The theoretical foundations of PINNs have also been progressively strengthened through analyses of approximation properties, convergence and error estimates \cite{de2024error,mhaskar2026approximation,shin2020convergence,hu2021extended}.

Beyond point-wise enforcement of the governing equations, SciDL has also explored physics-informed formulations based on their variational structure. Variational PINNs (VPINNs) \cite{vpinn_intro} incorporate the weak form of the governing PDEs directly into the training objective, providing a variational alternative to the strong-form residual minimization used in conventional PINNs. Building on this idea, hp-VPINNs \cite{hp-vpinn_intro} employ a subdomain Petrov–Galerkin formulation to extend the variational framework to one- and two-dimensional PDEs. The performance and applicability of PINNs and hp-VPINNs have been investigated extensively for linear steady-state convection–diffusion problems, including regimes characterized by sharp solution layers \cite{FHJ24}. Such studies are particularly relevant to the broader challenge of resolving multiscale and layer-dominated phenomena, where the choice of physical formulation and neural representation can strongly influence the accuracy and robustness of SciDL methods.

More recently, advances in neural-network architectures have broadened the SciDL landscape beyond conventional multilayer perceptrons. Kolmogorov–Arnold Networks (KANs) \cite{wang2025kolmogorov}, which replace fixed activation functions on network nodes with learnable univariate functions on edges, offer an alternative representation with potentially greater flexibility for capturing complex solution structures. Their integration with physics-informed learning has led to physics-informed KANs (PI-KANs) \cite{menon2026fekan,zhao2025pikan}, establishing another class of SciDL methods for solving PDEs. Early applications include bubble dynamics \cite{zhang2026bubbleokan} and high-dimensional problems \cite{menon2025anant}. Thus, PINNs, XPINNs, VPINNs and PI-KANs can be viewed within the same broader SciDL paradigm: they all leverage neural representations to construct mesh-free approximations while incorporating governing physical laws, but differ in how the physics is enforced and in the underlying network architecture. This progression—from point-wise and variational formulations to alternative neural representations, reflects a broader effort to develop more expressive, robust and adaptable SciDL frameworks for complex scientific systems.

Despite the remarkable success of PINNs, their ability to accurately resolve boundary layers is hindered by spectral bias \cite{maity2026ti,arihant2026nn}. This limitation has motivated extensive investigation into the representation and resolution of high-frequency and multiscale features within PINN frameworks \cite{jagtap2022deep,xu2019frequency}. One promising direction is a semi-analytic PINN formulation \cite{semianalytic_rectangulardomain}, in which a corrector function is introduced to analytically approximate the stiff component of the solution, thereby alleviating the associated stiffness.
Other approaches have sought to improve the training dynamics of PINNs. Variable-scaling PINNs (VS-PINNs) \cite{vspinn}, for example, employ variable scaling to enhance the resolution of boundary-layer features. Their effectiveness is further supported by Neural Tangent Kernel analysis, which provides a theoretical explanation for the improved training dynamics. Parameter-asymptotic PINNs (PAPINNs) \cite{papinn_onepar} offer another strategy for two-dimensional singularly perturbed problems (SPPs). Rather than relying on randomly initialized parameters, PAPINNs leverage the optimized parameters obtained for a larger perturbation parameter to initialize the network for successively smaller perturbation parameters. This continuation strategy facilitates smoother optimization and improves the resolution of solutions as the perturbation parameter decreases.
Despite these advances, existing approaches have predominantly focused on layer-forming problems governed by a single small perturbation parameter.

In contrast, this work considers the more challenging class of convection–diffusion–reaction problems involving multiple perturbation mechanisms. The interplay between these mechanisms determines whether the problem is convection- or reaction-dominated, often giving rise to sharp boundary layers \cite{dasmehrmann_1D_spp}. Accurately resolving such steep gradients remains a major challenge for PINN-based methods. Standard PINN formulations primarily minimize an $L^2$-type residual, yet a small $L^2$ error does not necessarily imply accurate resolution of the large gradients associated with boundary layers. Gradient-enhanced PINNs (gPINNs) \cite{yu2022gradient} address this limitation by incorporating gradient information into the loss function. Despite their empirical success, however, the theoretical basis for incorporating gradient information into the training objective remains largely unexplored.
Motivated by these observations, we develop a framework for resolving sharp gradients within boundary layers by employing a stronger energy norm and systematically combining gradient-based loss terms with suitable test functions. Beyond introducing a new loss formulation within the neural-network architecture, we establish a theoretical foundation for its effectiveness in problems exhibiting steep gradients. Recent work has established the existence of neural-network approximations for reaction–diffusion boundary-value problems \cite{opschoor2025neural}, while broader theoretical analyses of PINNs as PDE solvers have shown that their accuracy depends critically on the formulation of the loss function \cite{mishra2023estimates}. However, these analyses do not address problems characterized by steep gradients or boundary layers. In contrast, the present work quantifies the energy error of the proposed ENPINN approximation for PDEs with boundary layers, thereby providing a theoretical framework for assessing its accuracy in layer-dominated regimes.

The main contributions of this study are summarized as follows:
\begin{itemize}
\item We introduce ENPINN, an efficient framework that combines gradient-based information with a variational loss formulation to improve predictive accuracy in the energy norm while retaining the interpretability of PINNs.

\item We establish the theoretical role of test functions in controlling the spatial-derivative components of the residual loss and provide a mathematical analysis of two related PINN formulations. The first, termed WLPINN, incorporates a test function into the conventional PINN residual loss. This weighting improves the approximation relative to standard PINNs by scaling the residual near the boundary, thereby facilitating the enforcement of boundary conditions. However, WLPINN does not adequately capture derivative-dependent features of the solution, resulting in larger errors in the energy norm. The gPINN formulation, in contrast, augments the loss with spatial derivatives of the PDE residual. Our analysis shows that incorporating these derivative terms in the interior loss alone is insufficient to accurately resolve the steep gradients present in the underlying solution.
 
\item We strengthen the theoretical stability criterion for the energy-norm error, providing a rigorous foundation for the ENPINN framework. We first establish the existence of an ENPINN architecture capable of achieving arbitrarily small energy error and derive corresponding bounds on the derivatives of the neural-network approximation. These derivative estimates are then used to bound the energy error in terms of the ENPINN loss. Finally, we extend the ENPINN architecture to coupled multiscale problems, broadening its applicability to more complex systems.

\item We conduct a systematic comparison of VSPINN, WLPINN, gPINN, and ENPINN for convection–diffusion–reaction problems with steep gradients, using the energy norm as the primary error metric. The numerical experiments include a three-dimensional problem, a two-dimensional Burgers equation exhibiting an interior layer, and a coupled multiscale system. Across these test cases, the results demonstrate the effectiveness and broader applicability of the proposed ENPINN framework.

\end{itemize}

The paper is organized as follows. Section 2 formulates the generalized convection - diffusion – reaction problems arising in combustion flows and characterizes the corresponding solutions using perturbation theory. Sections 3 and 4 introduce the ENPINN framework and establish its existence and derivative bounds. Section 5 develops the energy-norm formulation for measuring the generalization error. Finally, Section 6 presents a series of numerical experiments that validate the theoretical results and demonstrate the effectiveness of the proposed approach. Section 7 summarizes the main findings of this study.

\section{Mathematical Formulation with Properties of Continuous Solution }\label{sec :Mathematical Formulation with Properties of Continuous Solution}
\subsection{Formulation of Governing Problem}\label{subsec: Formulation of Governing Problem}
The scope of this work is the application of generalized transport problems and its simulation. Mathematical formulation of the transport problems relies on the equations of heat transfer, chemical kinetics, and diffusion with the differential equations, where derivatives are defined in linear form and one of the unknown variables is defined in non-linear form \cite{pao2012nonlinear} as follows:
\begin{equation}\label{eq: gen_eq_heat_transfer}
        \frac{\partial}{\partial t}\left(\sum_{\mathfrak{i}}\mathfrak{A}_{\mathfrak{i}}\mathfrak{H}_{\mathfrak{i}}\right)=-\nabla. \mathfrak{q}+\mathfrak{q}_1,
\end{equation}
where $\mathfrak{A}_{\mathfrak{i}}$ is the concentration of the substances $\mathfrak{i}$, $\mathfrak{H}_{\mathfrak{i}}$ is their enthalpy, $\mathfrak{q}_1$ is the density of the heat sources, $\mathfrak{q}$ is the heat flux. The heat source term $\mathfrak{q}_1$ depends on the exponential form, and we denote it as $f(x,y,u)$. Using Fourier's law, $\mathfrak{q}$ is expressed as 
\begin{equation}\label{eq: val_q_in_heat_app}
        \mathfrak{q}\approx -D \nabla u+ \mathfrak{s}_p \rho \mathfrak{v}u,
\end{equation}
where $u$ is the varying temperature in a medium, $D$ is the coefficient of thermal conductivity, which could be taken as a very small quantity, $\mathfrak{s}_{p}$ is the mean mass heat capacity of the mixture at constant pressure, $\rho$ is the density, and $\mathfrak{v}$ is the flow velocity in the mixture, where we are assuming that the flow velocity of the mixture is very small. Then, the definition of the heat capacity gives
\begin{equation}\label{eq: val_lhs_of_gen_eq_heat_transfer}
    \frac{\partial }{\partial t}\left(\sum_{\mathfrak{i}}\mathfrak{A}_{\mathfrak{i}}\mathfrak{H}_{\mathfrak{i}}\right)=\mathfrak{s}_p \rho \frac{\partial u}{\partial t}.
\end{equation}
Substituting the values from equations \eqref{eq: val_q_in_heat_app} and \eqref{eq: val_lhs_of_gen_eq_heat_transfer} in \eqref{eq: gen_eq_heat_transfer}, we get
\begin{equation*}
    \mathfrak{s}_p \rho \frac{\partial u}{\partial t}= \nabla.\left(-D \nabla u+ \mathfrak{s}_p \rho \mathfrak{v}u\right) + f^*(x,y,t,u).
\end{equation*}  
{The above equation can be rewritten in the following form:}  
    \begin{equation}\label{eq: main_cont_prob_nonlin}
        \begin{cases}
            u_t-\varepsilon \Delta u+\mu {a^*}(x,y). \nabla u= {f^*(x,y,t,u)}, \quad (x,y,t) \in \Omega\equiv \Omega_s \times \Omega_t=(0,1)^2 \times (0,T],\\
            u(x,y,t)=0, \quad (x,y)\in \partial \Omega_s=\overline{\Omega}_s\backslash \Omega_s, ~ t\in [0,T], \\
            u(x,y,0)=\xi(x,y), \quad (x,y)\in \overline{\Omega}_s,
        \end{cases}
    \end{equation}
where $T \in (0,1]$, $\varepsilon$ depends on $D,\mathfrak{s}_p, \rho$ and $\mu$ on $\mathfrak{s}_p, \rho, \mathfrak{v}$, and $0< \varepsilon, \mu \ll 1$. The convection vector is given by ${a^*(x,y)=(a_1^*(x,y), a_2^*(x,y))}$ is a smooth function.
Using the mean value theorem, \eqref{eq: main_cont_prob_nonlin} can be formulated in a linear form
    \begin{equation}\label{eq: main_cont_prob_lin}
        \begin{cases}
            u_t-\varepsilon \Delta u+\mu {a}(x,y). \nabla u +b(x,y)u = f(x,y,t), \quad (x,y,t) \in \Omega=(0,1)^2 \times (0,T],\\
            u(x,y,t)=0, \quad (x,y)\in \partial \Omega_s=\overline{\Omega}_s\backslash \Omega_s, ~ t\in [0,T], \\
            u(x,y,0)=\xi(x,y), \quad (x,y)\in \overline{\Omega}_s,
        \end{cases}
    \end{equation}
where {$a=(a_1,a_2)$ and $ a_i \geq \alpha_0 >0$ for $\alpha=1,2$, $b\geq\beta_0>0$} and {$ {\partial f}/{\partial u} \leq 0$}. 
We define 
\begin{equation}\label{bound_a_b}
    \begin{aligned}
        \alpha=&\displaystyle \max_{(x,y)\in\overline{\Omega}_s}\{a_i(x,y),|a_{i_x}(x,y)|, |a_{i_y}(x,y)|\},\; i=1,2,\quad 
        \displaystyle \beta=\max_{(x,y)\in\overline{\Omega}_s}\{b(x,y),|b_x(x,y)|, |{b_y(x,y)|\}}.
    \end{aligned}
\end{equation} 
Furthermore, we define \(\zeta =\displaystyle \min_{(x,y)\in {\overline{\Omega}_s}} \left\{{b(x,y)}/{2a_1(x,y)},~{b(x,y)}/{2a_2(x,y)}\right\}\). {The underlying problem \eqref{eq: main_cont_prob_lin} exhibits boundary layer phenomena for small values of parameters $\varepsilon$, $\mu$.} The qualitative behavior and regularity properties of the solution depends on the ratio $\mu^2/\varepsilon$. 
Additionally, for $\mathcal{P}\in\mathbb{N}\cup \{0\}$, $\mathcal{Q}\in\mathbb{N}$ we define Sobolev space $W^{\mathcal{P},\mathcal{Q}}(\Omega)$ with the norm 
$$\|\chi\|_{W^{\mathcal{P},\mathcal{Q}}(\Omega)}=\left( \displaystyle \sum_{j=0}^{\mathcal{P}} \|\chi^{(j)}\|_{L^{\mathcal{Q}}(\Omega)}^{\mathcal{Q}}\right)^{\frac{1}{\mathcal{Q}}}, ~1 \leq \mathcal{Q} < \infty, \quad
\|\chi\|_{W^{\mathcal{P},\infty}(\Omega)}=\displaystyle \max_{i=0,1,...,\mathcal{P}} \left(\mathrm{ess ~~ sup}\left|\chi^{(i)}\right|\right).$$ 
The standard Sobolev space is denoted by $H^{\mathcal{P}}(\Omega)=W^{\mathcal{P},2}(\Omega),$ and in particular, $W^{0,0}(\Omega)\equiv L^2(\Omega).$ The constant $C$ denotes a generic positive constant independent of perturbation parameters like $\varepsilon, \mu$.

\subsection{{Nature of the boundary layers
}}\label{subsec:Properties of Exact Solutions}
For the convergence analysis of the proposed approach, we divide the spatial boundary $\partial\Omega_s$ into four parts as follows:
\begin{align*}
   &\Gamma_L=\{(0,y)\in \overline{\Omega}_s:0\leq y \leq 1\}, \quad \Gamma_R=\{(1,y)\in \overline{\Omega}_s:0\leq y \leq 1\}, \\
   &\Gamma_B=\{(x,0)\in \overline{\Omega}_s:0\leq x \leq 1\}, \quad \Gamma_T=\{(x,1)\in \overline{\Omega}_s:0\leq x \leq 1\}.
\end{align*}
We decompose $u$ into a regular component $r$, boundary layer components $s_{L}$, $s_{R}$, $s_{B}$, $s_T$ corresponding to the boundaries $\Gamma_L$, $\Gamma_R$, $\Gamma_B$, $\Gamma_T$ respectively, and corner layer functions $s_{LB}$, $s_{LT}$, $s_{RB}$, $s_{RT}$ related  to the corners $(0,0)$, $(0,1)$, $(1,0)$, $(1,1)$ of the spatial domain respectively. Therefore $$u=r+s_L+s_{R}+s_{B}+s_T+s_{LB}+s_{LT}+s_{RB}+s_{RT}.$$
{Depending on the compatibility conditions on the data, the derivative bounds for $\mu^2 \leq {\zeta \varepsilon}/{\alpha_0}$ of the solution satisfy (see, for details} \cite{o2011parameter})  
\begin{equation}\label{eq: exact_sol_derivative_bound_mu_less_eps}
    \begin{aligned}
        &\left|{\partial^{l_s+l_t}s_B}\right|\leq C \varepsilon^{-\frac{l_y}{2}} e^{-\sqrt{\frac{\alpha_0 \zeta}{\varepsilon}}y}, \hspace{0.85cm}
        \left|{\partial^{l_s+l_t}s_{RB}}\right|\leq C \varepsilon^{-\frac{l_s}{2}} \displaystyle \min \left\{ e^{-\sqrt{\frac{\alpha_0 \zeta}{\varepsilon}}(1-x)}, e^{-\sqrt{\frac{\alpha_0 \zeta}{\varepsilon}}y}\right\},\\
        &\left|{\partial^{l_s+l_t}s_T}\right|\leq C \varepsilon^{-\frac{l_y}{2}} e^{-\sqrt{\frac{\alpha_0 \zeta}{\varepsilon}}(1-y)}, ~~
        \left|{\partial^{l_s+l_t}s_{RT}}\right|\leq C \varepsilon^{-\frac{l_s}{2}} \displaystyle \min \left\{ e^{-\sqrt{\frac{\alpha_0 \zeta}{\varepsilon}}(1-x)}, e^{-\sqrt{\frac{\alpha_0 \zeta}{\varepsilon}}(1-y)}\right\}.\\
    \end{aligned}
\end{equation}
For the other case $\mu^2 \geq \frac{\zeta \varepsilon}{\alpha_0}$, the derivative bounds of the solution components are as follows:
\begin{equation}\label{eq: exact_sol_derivative_bound_mu_greater_eps}
    \begin{aligned}
        &\left|{\partial^{l_s+l_t}s_T}\right|\leq C \left(\frac{\varepsilon}{\mu}\right)^{-{l_y}} e^{-{\frac{\alpha_0 \mu}{2\varepsilon}}(1-y)},~ 
        \left|{\partial^{l_s+l_t}s_{RB}}\right|\leq C \displaystyle \min \left\{ \left(\frac{\varepsilon}{\mu}\right)^{-{l_x}} e^{-{\frac{\alpha_0 \mu}{2\varepsilon}}(1-x)}, \mu^{-{l_y}} e^{-{\frac{\zeta}{2\mu}}y}\right\},\\
        &\left|{\partial^{l_s+l_t}s_B}\right|\leq C \mu^{-{l_y}} e^{-{\frac{\zeta}{2\mu}}y}, ~
        \left|{\partial^{l_s+l_t}s_{RT}}\right|\leq C \displaystyle \min \left\{ \left(\frac{\varepsilon}{\mu}\right)^{-{l_x}} e^{-{\frac{\alpha_0 \mu}{2\varepsilon}}(1-x)}, \left(\frac{\varepsilon}{\mu}\right)^{-{l_y}} e^{-{\frac{\alpha_0 \mu}{2\varepsilon}}(1-y)}\right\}.\\
    \end{aligned}
\end{equation}
Here $l_s=l_x+l_y$, $0\leq l_s+2l_t\leq 4$. {In both cases, $\left|{\partial^{l_s+l_t}r}\right|\leq C$}. {Similarly, derivative bounds for the decomposition components $s_L$, $s_R$, $s_{LT}$ and $s_{LB}$ can be derived in an analogous manner.}

\section{Formulation of ENPINN}\label{sec:Formulation of ENPINN}
\subsection{{Mathematical Architecture of Neural Networks}}\label{subsec: Mathematical Architecture of Neural Networks}
Let $H \in \mathbb{N}$ denote the depth of the NN, $T$ be the size of the NN, and let $d_{h} \in \mathbb{N}$ be the number of neurons in the $ h-$th layer, for $h=0,1,..., H$. The trainable network parameters are collected in $\mathcal{P}=\{(\mathcal{W}^{h},{b}^h): h = 1,...., H\}$, where $\mathcal{W}^{h}\in\mathbb{R}^{d_{h}\times d_{h-1}}$ and ${b}^h\in\mathbb{R}^{d_{h}}$ denote the weight matrix and bias vector of the $h-$th layer, respectively. 
Let ${x} \in \Omega \subseteq \mathbb{R}^{d_0}$ be the input and let $u: \Omega \rightarrow \mathbb{R}$ be the target output, which is interpreted as a function belonging to a suitable Hilbert or Sobolev space equipped with an appropriate norm $\|.\|_{\star}$ or seminorm $|.|_{\star}$ such as the $L^2, H^1$ or $L^{\infty}$ norm. 
A mathematical framework of the feedforward neural network is defined as follows:
\begin{align*}
    {x}^{0}&={x},\\
    {x}^{h}&=\sigma\left({\mathcal{W}}_p^{h}{x}^{h-1}+{b}^{h}\right),\quad \mbox{for } h\in\{1,....,H-1\},\\
    {x}^{H}&={\mathcal{W}}_p^{H}{x}^{H-1},\\
    u_{N,p}(x)&=x^H.
\end{align*}
Here, $\sigma: \mathbb{R} \rightarrow \mathbb{R}$ denotes a nonlinear activation function chosen based on the nature of the target output, and {$p\in \mathcal{P}$} denotes the parameter.

{\subsection{Training Points}\label{subsec: Training Points}
The training set $\mathfrak{T} \subseteq \overline{\Omega}_s  \times [0, T]$  consists of training points, often called collocation points, and will be used to enforce the residual, boundary, and initial conditions. The training set $\mathfrak{T}$ is decomposed into the following components:
\begin{enumerate}
    \item Interior training points are defined by $\mathfrak{T}_{int}$=\{$\mathfrak{R}_{\mathfrak{c}}\}_{\mathfrak{c}=1}^{\mathcal{C}_{int}}$, where each point $\mathfrak{R}_{\mathfrak{c}}=(\mathfrak{C}_x,\mathfrak{C}_y,\mathfrak{C}_t)_{\mathfrak{c}}$ $ \in \Omega_s \times \Omega_t$ is used to minimize the PDE residual.
    \item The spatial boundary training set is defined as $\mathfrak{T}_{spb}$=\{$\mathfrak{S}_{\mathfrak{c}}\}_{\mathfrak{c}=1}^{\mathcal{C}_{spb}}$ with $\mathfrak{S}_{\mathfrak{c}}=(\mathfrak{C}_x,\mathfrak{C}_y,\mathfrak{C}_t)_{\mathfrak{c}} \in \partial\Omega_s \times \overline{\Omega}_t$. These points are employed to enforce the spatial boundary conditions.
    \item Initial condition training points are given by $\mathfrak{T}_{ini}$=\{$\mathfrak{I}_{\mathfrak{c}}\}_{\mathfrak{c}=1}^{\mathcal{C}_{ini}}$ where $\mathfrak{I}_{\mathfrak{c}}=(\mathfrak{C}_x,\mathfrak{C}_y)_{\mathfrak{c}} \in \overline{\Omega}_s$ corresponds to points at the initial time level $t=0$.
\end{enumerate}
  Consequently, the complete training set is expressed as, $\mathfrak{T}= \mathfrak{T}_{int} \cup \mathfrak{T}_{spb} \cup \mathfrak{T}_{ini} $.}

\subsection{Modified Residuals}\label{subsec:Modified Residuals}
Residual terms play a central role in the formulation of PINNs and their variants for approximating solutions of PDEs. In the present work, we define the residuals and their derivatives to construct the ENPINN algorithm for \eqref{eq: main_cont_prob_lin} and corresponding \eqref{eq: main_cont_prob_nonlin}. Let $u_{N,p} $
denotes the NN approximation having parameter $p \in \mathcal{P}$ with activation function Tanh. Then the residuals of the loss function are defined as follows:
    \begin{enumerate}
        \item \textbf{Interior Residual:} The residual at interior points of $\Omega$ is given by
        \begin{equation}\label{interior residual}
            \mathcal{N}_{int,p}(x,y,t)=\mathcal{N}_{res,p}(x,y,t)\times \mathbb{V}(x,y), \quad \mbox{for} ~ (x,y)\in \Omega_s,~ t\in[0,T],
        \end{equation}
        where
        \begin{align*}
            \mathcal{N}_{res,p}(x,y,t)&=\left(\frac{\partial}{\partial t} \left(u_{N,p}\right) -\varepsilon \Delta u_{N,p} + \mu {a}(x,y). \nabla u_{N,p} + b(x,y)u_{N,p}-f(x,y,t)\right),
        \end{align*}
        {$\mathbb{V}\in H_0^1(\Omega)$ for \eqref{eq: main_cont_prob_lin}}.
        The main motivation for introducing the factor $\mathbb V(x,y)$ is that it takes very small values near the boundary, in regions where boundary layers are expected. Hence, the residual in the boundary layers is scaled with a very small number. This approach is in the same spirit as the one used in the stabilization parameter optimization for convection-dominated problems in \cite{JKS11}, where the residual is calculated by neglecting all boundary cells. 
        \item {\textbf{Derivative Interior Residual:} The derivative residuals contributions associated with the interior domain are defined by incorporating the spatial derivatives of the interior residual function. Specifically,
        \begin{equation*}\label{Derivative_x_interior_residual}
            \mathcal{N}_{(int,p)_x}(x,y,t)=\mathcal{N}_{res_x}(x,y,t)\times \mathbb{V}(x,y)+ \mathcal{N}_{res}(x,y,t)\times \mathbb{V}_x(x,y), \, \mbox{for} ~ (x,y,t)\in \Omega,
        \end{equation*}
        \begin{equation*}\label{Derivative_y_interior_residual}
            \mathcal{N}_{(int,p)_y}(x,y,t)=\mathcal{N}_{res_y}(x,y,t)\times \mathbb{V}(x,y)+ \mathcal{N}_{res}(x,y,t)\times \mathbb{V}_y(x,y), \, \mbox{for} ~ (x,y,t)\in \Omega,
        \end{equation*}}
        where $\mathcal{N}_{(int,p)_x},~ \mathcal{N}_{(int,p)_y}$ are the derivatives of $\mathcal{N}_{(int,p)}$ with respect to $x$ and $y$ respectively.
        
        \item \textbf{Spatial Boundary Residual:} The residual corresponding to the spatial boundary condition on $\partial \Omega_s$ is expressed as
        \begin{equation}\label{spatial bounday residual}
            \mathcal{N}_{spb,p}(x,y,t)=u_{N,p}(x,y,t), \quad \mbox{for} ~ (x,y)\in \partial\Omega_s,~ t\in[0,T].
        \end{equation}
        \item \textbf{Initial condition residual:} The residual corresponding to the initial condition at $t=0$ is represented by
        \begin{equation}\label{Initial condition residual}
            \mathcal{N}_{ini,p}(x,y)=u_{N,p}(x,y,0)-\xi(x,y), \quad \mbox{for} ~ (x,y)\in \overline{\Omega}_s.
        \end{equation}
    \end{enumerate}

\subsection{Quadrature Rules}\label{subsec: Quadrature rules}
Let $\Omega^* \subseteq \mathbb{R}^d,\, d=2,3$ and $x^* \in \Omega^*$.
We adopt a general quadrature rule for a function $\chi: \Omega^* \rightarrow $ $\mathbb{R}$ depending on the regularity. The quadrature approximation is given by
\begin{equation}\label{eq: quadrature err approx}
    \begin{aligned}
        \int_{\Omega^*} \chi (x^*) \, d\Omega^* &\approx \displaystyle \sum_{i=1}^{\mathcal{C}^*} w_i \chi (x_i^*),
    \end{aligned}
\end{equation}
where $\{(x_i^*)\}_{i=1}^{\mathcal{C}^*} \subseteq \Omega^*, $ are the quadrature nodes, and $w_i \in \mathbb{R}^+$ is the corresponding suitable quadrature weights. It is important to note that the weights used in ENPINN are not necessarily identical to quadrature weights, since the former may depend on perturbation parameters such as $\varepsilon$ or $\mu$. In general, we assume the quadrature errors satisfy the bounds
\begin{equation*}
    \begin{aligned}
        \left|\int_{\Omega^*} \chi (x^*) \, d\Omega^* - \displaystyle \sum_{i=1}^{\mathcal{C}^*} w_i \chi (x_i^*)\right| &\leq \mathcal{K}_{quer} {(\mathcal{C}^*)}^{-\alpha_{quer}},
    \end{aligned}
\end{equation*}
where $\mathcal{K}_{quer}$ denotes a positive constant that depends on the number of points that are used to approximate the integrand, $\alpha_{quer} >0$ denotes the convergence rate. {Analogously, for time-dependent problems, we shall consider $\Omega^* \subseteq \mathbb{R}^2 \cup \mathbb{R}^+$ where $t^*$ is considered as an additional coordinate in $\Omega^*$.}
For instance, Monte Carlo quadrature approximates the integral by treating the quadrature points as independent and identically distributed, and the quadrature error is bounded by $C \mathcal{C}^{-\alpha_{quer}}$, where $C$ is a constant independent of inverse power of {$\varepsilon, \mu$}. 
Let $\mathscr{X}: \mathbb{R}^2 \times \mathbb{R}^+ \rightarrow \mathbb{R}$ be the random variable, $\Omega_{qu}$ be the compact {subset of} $\mathbb{R}^2 \times \mathbb{R}^+$ and $\mathscr{F}$ be a continuous function on $\Omega_{qu}$. From \cite{haber1966modified}, we assume that the random variable is $\mathscr{X}= |\Omega_{qu}| \mathscr{F}(x,y,t) $. 
Then the expectation of $\mathscr{X}$ is defined as follows
\begin{equation*}
    E(\mathscr{X}) = \dfrac{\int_{\Omega_{qu}}|\Omega_{qu}| \mathscr{F}\, d\Omega_{qu}}{|\Omega_{qu}|} = \int_{\Omega_{qu}} \mathscr{F}(x,y,t)\, d\Omega_{qu} := \mathscr{I} \mbox{(say)}.
\end{equation*}
Assume $\mathscr{I_A}$ is the approximation of $\mathscr{I}$ of the form \eqref{eq: quadrature err approx}. Define
\begin{equation*}
    \varepsilon_{quad}^{\mathscr{I}} = \mathscr{I}- \mathscr{I}_{\mathscr{A}},
\end{equation*}
where $\varepsilon_{quad}^{\mathscr{I}}$ denotes the error. The central limit theorem guarantees that
\begin{equation}\label{eq: quadrature error monte carlo expression}
    \varepsilon_{quad}^{\mathscr{I}} = |\Omega_{qu}| \mathfrak{s} \mathcal{C}^{-\alpha_{quer}},
\end{equation}
where $\mathfrak{s}^2 = E[\mathscr{X}^2]- [E(\mathscr{X})]^2$.

\subsection{Loss Function}\label{subsec: Loss Function}
{To train ENPINN, {we require a boundary fitting loss function that predicts the layer effectively.} The loss function used in the proposed framework is defined as
\begin{equation}\label{eq: enpinn_loss}
    \begin{aligned}
        \mathcal{L}_{en}^C(p) = &\lambda_{int}\int_{\Omega_s \times [0,T]} \mathcal{N}_{int,p}(x,y,t)\, dx \,dy\, dt + \lambda_{int} \int_{\Omega_s \times [0,T]} \mathcal{N}_{(int,p)_x}(x,y,t)\, dx \,dy\, dt \\
        &+ \lambda_{int}\int_{\Omega_s \times [0,T]} \mathcal{N}_{(int,p)y}(x,y,t)\, dx \,dy\, dt + \lambda_{spb} \int_{\partial\Omega_s \times [0,T]} \mathcal{N}_{spb,p}(x,y,t)\, dx \,dy\, dt\\
        &+ \lambda_{ini} \int_{\bar{\Omega}_s} \mathcal{N}_{ini,p}(x,y)\, dx \,dy,
    \end{aligned}
\end{equation}
where $\lambda_{int}, \lambda_{spb}, \lambda_{ini}$ are the penalty coefficients associated with the corresponding loss functions.
Now, using the training points and the quadrature rule defined in subsection \ref{subsec: Quadrature rules}, we employ the following loss function in the ENPINN algorithm at the discrete collocation points,
\begin{equation}\label{loss_function}
    \begin{aligned}
        \mathcal{L}_{en}^D(p)=& \displaystyle {\lambda_{ini}}\sum_{n=1}^{\mathcal{C}_{ini}}w_n^{ini}|\mathcal{N}_{ini,p}(x_s^n)|^2+ {\lambda_{spb}}\sum_{n=1}^{\mathcal{C}_{spb}}w_n^{spb}|\mathcal{N}_{spb,p}(x_{st}^n)|^2
        +{\lambda_{int}}\sum_{n=1}^{\mathcal{C}_{int}}w_n^{int}|\mathcal{N}_{int,p}(x_{st}^n)|^2\\
        &+ {\lambda_{int}} \sum_{n=1}^{\mathcal{C}_{int}}w_n^{int^x}|\mathcal{N}_{(int,p)_x}(x_{st}^n)|^2 + {\lambda_{int}}\sum_{n=1}^{\mathcal{C}_{int}}w_n^{int^y}|\mathcal{N}_{(int,p)_y}(x_{st}^n)|^2,
    \end{aligned}
\end{equation}}
where $x_s^n$ is a point in $(x,y)$ plane and $x_{st}^n$ is a point in $(x,y,t)$ space.
Here, all the residual terms are defined in Subsection \ref{subsec:Modified Residuals}. The terminologies $w_{n}^{ini}$,$w_{n}^{spb}$, and $w_{n}^{int}$ denote the quadrature weights corresponding to the initial condition, spatial boundary, and interior training points, respectively. Similarly, $w_{n}^{int_x}$ and $w_{n}^{int_y}$ are quadrature weights associated with the interior training points corresponding to the derivatives of the interior residual in the $x$ and $y$ directions.

    The loss function for WLPINN is defined by excluding the residual derivative loss components, and it is as follows:
    \begin{equation}\label{eq: WLPINN_loss}
        \begin{aligned}
            \mathcal{L}_{wl}^C(p) = &\lambda_{int}\int_{\Omega_s \times [0,T]} \mathcal{N}_{int,p}(x,y,t)\, dx \,dy\, dt + \lambda_{spb} \int_{\partial\Omega_s \times [0,T]} \mathcal{N}_{spb,p}(x,y,t)\, dx \,dy\, dt \\
            &\quad+ \lambda_{ini} \int_{\bar{\Omega}_s} \mathcal{N}_{ini,p}(x,y)\, dx \,dy,
        \end{aligned}
    \end{equation}
    Furthermore, the loss function for {gPINN} is defined by excluding the test function from the residual derivative terms, and it is given by:
    \begin{equation}\label{eq: WLDPINN_loss}
    \begin{aligned}
        \mathcal{L}_{wld}^C(p) = &\lambda_{int}\int_{\Omega_s \times [0,T]} \mathcal{N}_{int,p}(x,y,t)\, dx \,dy\, dt + \lambda_{int} \int_{\Omega_s \times [0,T]} \mathcal{N}_{(res,p)_x}(x,y,t)\, dx \,dy\, dt \\
        &+ \lambda_{int}\int_{\Omega_s \times [0,T]} \mathcal{N}_{(res,p)y}(x,y,t)\, dx \,dy\, dt + \lambda_{spb} \int_{\partial\Omega_s \times [0,T]} \mathcal{N}_{spb,p}(x,y,t)\, dx \,dy\, dt\\
        &+ \lambda_{ini} \int_{\bar{\Omega}_s} \mathcal{N}_{ini,p}(x,y)\, dx \,dy.
    \end{aligned}
\end{equation}
\begin{note}
        The derivative residuals $\mathcal{N}_{(int,p)_x}$, $\mathcal{N}_{(int,p)_y}$ of $\mathcal{N}_{int,p}$ play a vital role in establishing the convergence of ENPINN in the energy norm. {The incorporation of the derivative of weighted residual is crucial for capturing large gradients near the boundary.} This assertion is substantiated in the numerical section through a comparison of ENPINN with WLPINN (without residual derivative) and {gPINN} (without the test function in the residual derivative).
\end{note}

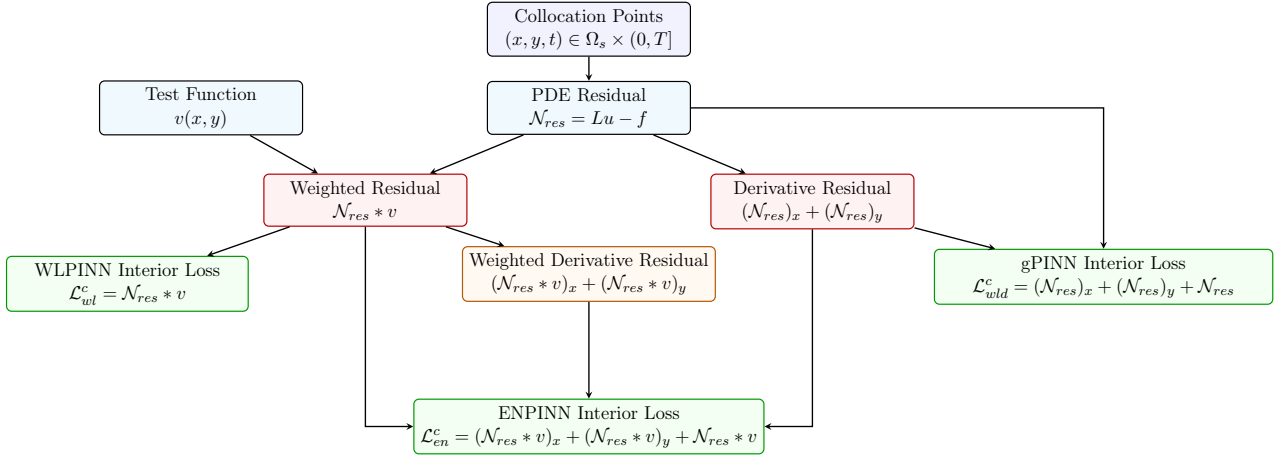
\begin{figure}
    \centering
    \resizebox{\textwidth}{!}{
    \tikzset{
        box/.style={
            rectangle,
            rounded corners=3pt,
            draw=black,
            minimum width=4.2cm,
            minimum height=0.5cm,
            align=center,
            font=\small,
            line width=0.7pt
        },
        input/.style={box, fill=blue!5},
        process/.style={box, fill=cyan!5},
        residual/.style={box, fill=red!5, draw=red!70!black},
        loss/.style={box, fill=green!5, draw=green!60!black},
        arrow/.style={->, thick, >=stealth}
    }

    \begin{tikzpicture}[node distance=1.6cm and 2.2cm, transform shape]

        \node[input] (colloc)
        {Collocation Points\\
        $(x,y,t)\in \Omega_s \times (0,T]$};

        \node[process, below= 0.5 cm of colloc] (pde_res)
        {PDE Residual\\
        $\mathcal{N}_{res} = Lu-f$};

        \node[process, left=3.8cm of pde_res] (test)
        {Test Function\\
        $v(x,y)$};

        \node[residual, below left=0.8cm and 0.4cm of pde_res] (weighted_res)
        {Weighted Residual\\
        $\mathcal{N}_{res}*v$};

        \node[residual, below right=0.8cm and 0.4cm of pde_res] (deriv_res)
        {Derivative Residual\\
        $(\mathcal{N}_{res})_x+(\mathcal{N}_{res})_y$};

        \node[residual, below=2.3cm of pde_res, fill=orange!5, draw=orange!70!black]
        (weighted_deriv_res)
        {Weighted Derivative Residual\\
        $(\mathcal{N}_{res}*v)_x+(\mathcal{N}_{res}*v)_y$};

        \node[loss, below=2cm of weighted_deriv_res, minimum width=7cm]
        (ENPINN_int_loss)
        {ENPINN Interior Loss\\
        $\mathcal{L}_{en}^c
        =
        (\mathcal{N}_{res}*v)_x+(\mathcal{N}_{res}*v)_y+\mathcal{N}_{res}*v$};

        \node[loss, below left=0.6cm and 0.3cm of weighted_res, minimum width=5cm]
        (WLPINN_int_loss)
        {WLPINN Interior Loss\\
        $\mathcal{L}_{wl}^c
        =\mathcal{N}_{res}*v$};

        \node[loss, below right=0.4cm and 0.4cm of deriv_res, minimum width=7cm]
        (WLDPINN_int_loss)
        {gPINN Interior Loss\\
        $\mathcal{L}_{wld}^c
        =
        (\mathcal{N}_{res})_x+(\mathcal{N}_{res})_y+\mathcal{N}_{res}
        $};

        \draw[arrow] (colloc) -- (pde_res);
        
        \draw[arrow] (pde_res) -- (weighted_res);
        \draw[arrow] (pde_res) -- (deriv_res);

        \draw[arrow] (pde_res) -| (WLDPINN_int_loss);
        \draw[arrow] (deriv_res) -- (WLDPINN_int_loss);

        \draw[arrow] (test) -- (weighted_res);

        \draw[arrow] (weighted_res) -- (weighted_deriv_res);
        \draw[arrow] (weighted_res) -- (WLPINN_int_loss);

        \draw[arrow] (weighted_res) |- (ENPINN_int_loss);
        \draw[arrow] (deriv_res) |- (ENPINN_int_loss);
        \draw[arrow] (weighted_deriv_res) -- (ENPINN_int_loss);
    \end{tikzpicture}}
    \caption{\scriptsize{{Schematic illustration of the pathway leading to ENPINN interior loss. A boundary-fitting test function $v$ with small scaling of the loss function near the boundary is employed. The key idea of the ENPINN loss is to emphasize not only the residual itself but also its spatial derivatives.}}}
    \label{fig: pathway to enpinn}
\end{figure}

\subsection{Algorithm for ENPINN}\label{subsec: Algorithm for ENPINN}
The main objective of the ENPINN is to find the optimal network parameter $p^*$ such that the traditional loss, including the residual derivative, is minimized, {}{i.e.,} find $p^* \in \mathcal{P}$ defined as $p^*= \displaystyle\arg\min_{p\in \mathcal{P}}\mathcal{L}_{en}^D$.
{A schematic development of the ENPINN loss corresponding to \eqref{eq: main_cont_prob_lin} is provided in Figure \ref{fig: pathway to enpinn}.
Here $u_{N,p^*}(x,y,t;p^*)$ denotes the optimal NN approximation obtained using the ENPINN framework. The loss function $\mathcal{L}_{en}^D(x,y,t;p)$ of the ENPINN comprises the interior, spatial boundary, and initial residual, along with derivative terms of the interior residual. Notably, the interior residual is multiplied by a function $\mathbb{V} \in H^1(\Omega).$ To enhance computational efficiency, an adaptive learning rate is employed. The Adam optimizer is utilized to update the ENPINN parameters. The algorithm {}{update the learning parameters} until either the ENPINN reaches the maximum number of epochs or the loss function falls below a prescribed tolerance.}

\section{Bounds of ENPINN}\label{sec: Existence of ENPINN}
{}{The purpose of this section is to show the derivative bounds of ENPINN approximation for \eqref{eq: main_cont_prob_lin} analogous to \eqref{eq: exact_sol_derivative_bound_mu_less_eps}-\eqref{eq: exact_sol_derivative_bound_mu_greater_eps}.}
Such estimate bounds are fundamental in the derivation of the $L^2$ and energy error established in the subsequent section. For the sake of simplicity, we derive the derivative bounds for an NN corresponding to the following boundary value problem having boundary layers:
\begin{equation}\label{eq: ode_in_existence}
    \begin{aligned}
        -\varepsilon u''(x) &+ a \mu u'(x) + bu=0, \quad x \in(0,1)\\
        &u(0)=\mathbb{U}_0, ~ u(1)=\mathbb{U}_1,
    \end{aligned}
\end{equation}
where $a>0$, $b>0$ are constants.
\subsection{Existence of ENPINN in multi parameter boundary value  problems}\label{subsec:Existence of ENPINN}
{In the following proofs, we use several neural networks defined below.} {The exponential function $\exp{(-)}$ and identity functions are approximated by the neural networks $u_{\mathfrak{J}, N}^{\exp}$ and $u_{h-1,\mathfrak{K}, T, N}^{id}$ with tolerances $\mathfrak{J}$ and $\mathfrak{K}$ respectively (\cite{opschoor2025neural}, Lemma $7.2$ and Lemma $7.9$). {}{Both neural network constructions use the Tanh activation function}. Here $N$ denotes for neural network representation, and the parameters $T$, $h$ are defined in Subsection \ref{subsec: Mathematical Architecture of Neural Networks}.}

\begin{theorem}\label{th: existence_enpinn}
    {Let $\overline{\Omega} = [0,1]$. For a fixed depth $h$ and every $\mathfrak{B}>0$, there exists $\hat{p}\in \mathcal{P}$ such that $\mathcal{L}_{en}^C(\hat{p})< \mathfrak{B}$, where the corresponding neural network $u_{N,\hat{p}}$ approximates the solution of \eqref{eq: ode_in_existence}.}
    \begin{proof}
        See in \ref{sec: appen_Bounds of ENPINN}.
    \end{proof}
\end{theorem}
\begin{remark}
    The parameter $p^*$ and the parameter $\hat{p}$ obtained from Theorem \ref{th: existence_enpinn} are generally not the same. The difference between $\mathcal{L}_{en}^C(\hat{p})$  and $\mathcal{L}_{en}^C({p}^*)$ is often termed the optimization error in neural network prediction \cite{de2024error}. As $\mathfrak{B}$ decreases, the value of $\mathcal{L}_{en}^C(\hat{p})$ becomes closer to $\mathcal{L}_{en}^C(p^*)$. Consequently, the optimization error depends implicitly on $\mathfrak{B}$; a smaller value of $\mathfrak{B}$ yields a lower optimization error. 
\end{remark}

\subsection{Derivative Bounds of ENPINN based Approximation}\label{subsec: Bounds of Neural Network Solution}
In the estimation of the generalization error for ENPINN, we require the derivative bounds for both the exact solution and the ENPINN approximation. In the following theorem, we demonstrate that the derivative bounds of ENPINN closely resemble those of the exact solution with respect to the neural network learning parameter $\hat{p}$ used in establishing the existence of ENPINN.
\begin{theorem}\label{th: deriv_bounds_enpinn}
    Let $u_{N,\hat{p}}^L$ be the approximation of the layer decomposition $s_L$ of $u$ {for $\mu^2 \leq ({4b}\varepsilon/{a^2})$}. Then
    \begin{equation*}
        |\partial ^{l_x} u_{N,\hat{p}}^L| \leq C \varepsilon ^{-\frac{l_x}{2}} \exp{\left(-\sqrt{\frac{b}{\varepsilon}}x\right)}. 
    \end{equation*}
    \begin{proof}
        Details are given in \ref{sec: appen_Bounds of ENPINN}.
    \end{proof}
\end{theorem}

\begin{remark}\label{remark: deriv_bounds_enpinn}
    {}{The derivative bounds in Theorem \ref{th: deriv_bounds_enpinn} are derived for \eqref{eq: ode_in_existence} under the condition $\mu^2 \leq ({4b}\varepsilon/{a^2})$. Similar estimates can be obtained for the case $\mu^2 \geq ({4b}\varepsilon/{a^2})$. Furthermore, these results can be extended to higher dimensional problems by employing Alternating Direction Implicit (ADI) splitting method whose derivative bounds satisfy estimates analogous to the exact solution bounds given in \eqref{eq: exact_sol_derivative_bound_mu_less_eps} and \eqref{eq: exact_sol_derivative_bound_mu_greater_eps}.}
\end{remark}

\section{Error Estimation and its Convergence}\label{sec: Error Estimation and Convergence Analysis}
In this section, we derive the generalization error in the energy norm in terms of the training error. After training the ENPINN using the loss function in \eqref{loss_function}, we obtain the optimal parameter set {}{${p}^*$}. Evaluating the final loss at {}{${p}^*$} yields the training error. Hence, the training error $\mathcal{E}_{Tr}$ can be directly computed from the value of the loss function. First, we define the training errors based on $L^2$ norm $\mathcal{E}_{Tr}^{L^2}$ and energy norm $\mathcal{E}_{Tr}^{en}$ as
\begin{equation}\label{eq: training error l2}
    \begin{aligned}
        \left(\mathcal{E}_{Tr}^{L^2}\right)^2=& \displaystyle \underbrace{\lambda_{ini}\sum_{n=1}^{\mathcal{C}_{ini}}w_n^{ini}|\mathcal{N}_{ini,{{p}^*}}(x_s^n)|^2}_{\mathcal{E}_{ini}^2}+ \underbrace{\lambda_{spb}\sum_{n=1}^{\mathcal{C}_{spb}}w_n^{spb}|\mathcal{N}_{spb,{{p}^*}}(x_{st}^n)|^2}_{\mathcal{E}_{spb}^2} +\underbrace{\lambda_{int}\sum_{n=1}^{\mathcal{C}_{int}}w_n^{int}|\mathcal{N}_{int,{{p}^*}}(x_{st}^n)|^2}_{\mathcal{E}_{int}^2},
    \end{aligned}
\end{equation}
\begin{equation}\label{eq: training error energy}
    \begin{aligned}
        \left(\mathcal{E}_{Tr}^{en}\right)^2=        \left(\mathcal{E}_{Tr}^{L^2}\right)^2+ \underbrace{\lambda_{int}\sum_{n=1}^{\mathcal{C}_{int}}w_n^{int^x}|\mathcal{N}_{(int,{{p}^*})_x}(x_{st}^n)|^2}_{(\mathcal{E}_{int}^x)^2}+ \underbrace{\lambda_{int}\sum_{n=1}^{\mathcal{C}_{int}}w_n^{int^y}|\mathcal{N}_{(int,{{p}^*})_y}(x_{st}^n)|^2}_{(\mathcal{E}_{int}^y)^2}.
    \end{aligned}
\end{equation}
\subsection{Energy Norm}
The energy norm is defined as follows:
\begin{equation}\label{energy norm}
    \|v\|_E^2=\varepsilon |v|^2_{H^1(\Omega)}+\|v\|^2_{L^2(\Omega)} \quad \forall v \in H^1(\Omega).
\end{equation}
Here, we first derive the generalization error in the $L^2$ norm and subsequently extend the result for the energy norm. To initiate this analysis, we define the generalization error in $L^2$ norm as:
\begin{equation}\label{eq: l2 error}
    \mathcal{E}^{L^2}_G:=\left(\int_0^T \int_{\Omega_s}|u(x,y,t)-u_{N,{{p}^*}}(x,y,t)|^2\, dxdy dt\right)^{\frac{1}{2}}.
\end{equation}

\subsection{Error Estimate in $L^2$ norm}
We begin by deriving an estimate for the generalization error of the NN model in the $L^2$ norm, based on its approximation to the solution of \eqref{eq: main_cont_prob_lin}.

\begin{theorem}\label{th:gronwall_genearlization_err_l^2_norm}
    Let {$u \in {H^2(\Omega) \cap C^0(\overline{\Omega})}$} denote the exact solution of \eqref{eq: main_cont_prob_lin} for {$\mu^2 \geq \zeta \varepsilon/\alpha_0$}. Let $u_{N,p^*}$ be the approximation of u obtained using the 
    loss function $\mathcal{L}_{N}^C (p)$  in \eqref{loss_function}. 
    Then the generalization error satisfies the following inequality
    \begin{equation}\label{eq: gronwall_genearlization_err_l^2_norm}
        \|\mathfrak{e}\|_{{L^2(\Omega_s)}}^2= \|u-u_{N,{{p}^*}}\|_{{L^2(\Omega_s)}}^2 \leq e^{\kappa_{L_2}^l}\times \mathcal{G}_{L_2}^l,
    \end{equation}
    where $\kappa_{L_2}={(2\alpha + 1)T}$, $\mathcal{G}_{L_2} = \|\mathcal{N}_{ini,{{p}^*}}\|_{L^2(\Omega_s)}^2 + T^{1/2} \|\mathcal{N}_{spb,{{p}^*}}\|_{L^2(\Omega)}  \mathcal{K}_{\partial \Omega_s}^{\nabla \mathfrak{e}} + 2\alpha \|\mathcal{N}_{spb,{{p}^*}}\|^2_{L^2(\Omega)} + \|\mathcal{N}_{int,{{p}^*}}\|_{L^2(\Omega)}^2$. The positive constant $\mathcal{K}_{\partial\Omega_s}^{\nabla \mathfrak{e}}$ is independent of $\varepsilon$ and depends on the continuous solution $u$, the NN approximation $u_{N,{{p}^*}}$, $i.e.,$ 
    $$\mathcal{K}_{\partial\Omega_s}^{\nabla \mathfrak{e}}\equiv \mathcal{K}_{\partial\Omega_s}^{\nabla \mathfrak{e}}(u,u_{N,{{p}^*}}).$$
    \begin{proof}
        The details of the proof are given in \ref{sec: appen_Error Estimation}.
    \end{proof}
\end{theorem}

\begin{remark}\label{remark: err_less_mu_less_eps_than_mu_greater_eps_l2}
    Theorem \ref{th:gronwall_genearlization_err_l^2_norm} is valid for both $\mu^2 \leq \zeta \varepsilon/\alpha_0$ and $\mu^2 \geq \zeta \varepsilon/\alpha_0$.  For the case  $\mu^2 \geq \zeta \varepsilon/\alpha_0$, the exponential term contains the factor $(2\alpha+1)$, and an additional spatial boundary loss term appears. In contrast, for $\mu^2 \leq \zeta \varepsilon/\alpha_0$, such a rearrangement of the convection term is unnecessary. Consequently, the exponential weight contains only $(\alpha+1)$, and no additional spatial boundary loss component arises.
\end{remark}

\begin{theorem}\label{th:generalization_err_l^2_norm}
    Let $u_{N,{{p}^*}}$ be the approximated solution of \eqref{eq: main_cont_prob_lin} obtained by Theorem \ref{th:gronwall_genearlization_err_l^2_norm}.  
    Then the generalization error measured in $L^2$ norm is given by
    \begin{equation}\label{generalization_err_l^2_norm}
        \begin{aligned}
            \|u-u_{N,{{p}^*}}\|_{L^2(\Omega)}^2
            &\leq \mathcal{K}_{L_G}^2\left(\varepsilon_{ini}^2+\mathcal{K}_{ini}\mathcal{C}_{ini}^{-O_{ini}}+ \mathcal{K}_{\partial\Omega_s}^{\nabla \mathfrak{e}, s_0}{T^{\frac{1}{2}}}\varepsilon_{spb}+ \mathcal{K}_{\partial \Omega_s}^{\nabla \mathfrak{e},s_0}{T^{\frac{1}{2}}}\mathcal{K}_{spb}^{\frac{1}{2}}\mathcal{C}_{spb}^{-\frac{O_{spb}}{2}}\right)\\
            & \quad + \mathcal{K}_{L_G}^2\left(\varepsilon_{int}^2+ \mathcal{K}_{int}\mathcal{C}_{int}^{-O_{int}}+2\alpha\varepsilon_{spb}^2+ 2\alpha\mathcal{K}_{spb}\mathcal{C}_{spb}^{-O_{spb}}\right)
        \end{aligned}
    \end{equation}
    where $\mathcal{K}_{L_G}=\sqrt{Te^{(2\alpha + {1})T}}$, and $\mathcal{K}_{\partial\Omega_s}^{\nabla \mathfrak{e}}\equiv \mathcal{K}_{\partial\Omega_s}^{\nabla \mathfrak{e}}(u,u_{N,{}{{p}^*}})$ is {the quadrature approximation $\mathcal{K}_{\partial\Omega_s}^{\nabla \mathfrak{e}, l}$ or $\mathcal{K}_{\partial\Omega_s}^{\nabla \mathfrak{e}, g}$}. {The constant $s_0$ are independent of $\varepsilon,\,\mu$ but depend on the ratio ${\mu^2}/{\varepsilon}$}. Here $\varepsilon_{ini},\varepsilon_{spb},\varepsilon_{int}$ denote the training errors of the initial condition, spatial boundary, interior domain, respectively, and $\mathcal{K}_{ini}\mathcal{C}_{ini}^{-O_{ini}}$, $\mathcal{K}_{spb}\mathcal{C}_{spb}^{-O_{spb}}$, $\mathcal{K}_{int}\mathcal{C}_{int}^{-O_{int}}$ are the corresponding quadrature error, respectively.
    \begin{proof}
        The proof follows directly from Theorem \ref{th:gronwall_genearlization_err_l^2_norm} for both the cases $\mu^2 \leq {\zeta\varepsilon}/{\alpha_0}$ and $\mu^2 \geq {\zeta\varepsilon}/{\alpha_0}$, obtained by integrating over $\overline{T}\in [0, T]$. 
    \end{proof}
\end{theorem}
{
    \begin{note}
        Based on \eqref{eq: quadrature error monte carlo expression}, we can get the quadrature error for $\mathcal{N}_{int}$ as follows 
        \begin{equation*}
            \mathcal{K}_{int}\mathcal{C}_{int}^{-O_{int}} = |\Omega| \mathfrak{s} \mathcal{C}^{-O_{int}},
        \end{equation*}
        where $\mathfrak{s}$ is independent of the perturbation parameters. Consequently, $\mathcal{K}_{spb}$, $\mathcal{K}_{ini}$ can be computed.
     \end{note}}

\begin{remark}
    A careful inspection of \eqref{generalization_err_l^2_norm} reveals that the generalization error remains small under the following circumstances:\\
    $i)$ The training errors are sufficiently small. 
    Their magnitude can also be evaluated {a posteriori}.\\
    $ii)$ The quadrature errors are negligible. Their size depends on the number of training points and on the quadrature constants. In particular, for a sufficiently large number of training points, one can ensure that $\mathcal{K}_{quer} \mathcal{C}^{-\alpha_{quer}}\ll1.$\\
\end{remark}

\subsection{Error Estimate in Energy Norm}
Using the generalization error estimate in the $L^2$ norm, we now derive a corresponding error bound in the energy norm, which is defined as follows:
\begin{equation}\label{eq: energy error}
    \mathcal{E}^{en}_G:=\mathcal{E}^{L^2}_G+ \varepsilon \left(\int_0^T \int_{\Omega_s}|\nabla (u-u_{N,{}{{p}^*}})(x,y,t)|^2\, dxdy dt\right)^{\frac{1}{2}}.
\end{equation}

\begin{theorem}\label{th:gronwall_genearlization_err_energy_norm}
     Let $u\in {H^3(\Omega) \cap C^0(\overline{\Omega})}$ satisfies \eqref{eq: main_cont_prob_lin} {for $\mu^2 \geq \zeta \varepsilon/\alpha_0$} and $u_{N,{{p}^*}}$ be ENPINN approximation of u obtained through ENPINN. Then the energy norm-based error will satisfy the following inequality:
    \begin{equation}\label{gronwall_genearlization_err_energy_norm}
        \|\mathfrak{e}\|_{L^2(\Omega)}^2+ \varepsilon\|\nabla \mathfrak{e}\|_{L^2(\Omega)}^2
        \leq \mathcal{G}_{L_2} e^{\kappa_{L_2}}+\mathcal{G}_{en} e^{\kappa_{en}},
    \end{equation}
     where $T_1 \leq T, \text{and }T \in (0,1]$, $$\kappa_{en}= \left(\frac{ 6  \alpha+3\beta +2}{T_1}\right)T, ~~ \mathcal{G}_{en}= \|\mathcal{N}_{int_x,{{p}^*}}^{\varepsilon}(x,y,t)\|^2_{L^2(\Omega)}+ \|\mathcal{N}_{int_y,{{p}^*}}^{\varepsilon}(x,y,t)\|^2_{L^2(\Omega)}+ 2 \beta\mathcal{E}_{L_2}^{G}$$ and the quantities $\mathcal{G}_{L_2}, e^{\kappa_{L_2}}$ are defined in Theorem \ref{th:gronwall_genearlization_err_l^2_norm}.
\end{theorem}

{\begin{remark} 
    As mentioned in Remark \ref{remark: err_less_mu_less_eps_than_mu_greater_eps_l2}, the rearrangement of the convection term causes additional error for the regime $\mu^2 \geq \zeta \varepsilon/\alpha_0$. As a result, the exponential factor in the error estimate contains an extra contribution compared to the case $\mu^2\leq \zeta \varepsilon/\alpha_0$.
\end{remark}}

{\begin{remark}
    We obtain the desired inequality for $\varepsilon \| \nabla \mathfrak{e}\|_{L^2}$. The term $\|\nabla \mathfrak{e}\|_{L^2(\Omega)}$ introduces an additional first-order derivative of $\mathcal{N}_{int}$ with respect to the spatial derivatives. The minimization of the residual derivatives, including the residual and the given conditions, provides an error bound under the energy norm and also yields sharper gradients than the $L^2$ norm-based loss function {}{because it reduces the sharpness error between the ENPINN approximation and the underlying solution.}    
\end{remark}}

{
\begin{remark}
   The key advantage of ENPINN over WLPINN is its ability to reduce derivative residuals throughout training.
   Consequently, gPINN encounters significant challenges in resolving the layer phenomena, as both the residual $\mathcal{N}_{res,p}$ and its derivatives depend on inverse powers of $\varepsilon$ and $\mu$. As the parameter values decrease, the interior loss increases. ENPINN overcomes this dependence on the inverse powers of the perturbation parameters by including a suitable test function in the loss formulation. This incorporation plays a key role in mitigating the growth of residuals and their derivatives, which is the primary reason for ENPINN's superior performance compared with gPINN. 
\end{remark}
}

\begin{theorem}\label{th: gen_err_energy_norm}
     Let $u$ and $u_{N,{{p}^*}}$ be defined in Theorem \ref{th:gronwall_genearlization_err_energy_norm}. Then the generalization error estimate in the previously defined energy norm \eqref{energy norm} satisfies
     \begin{equation}\label{genearlization_err_energy_norm}
            \mathcal{E}^{en}_G=\|u-u_{N,{}{{p}^*}}\|^2_E
            \leq \mathcal{K}_{en_G}^2 \times\left(\mathcal{E}_{int^{\varepsilon}_x}^2 +\mathcal{E}_{int^{\varepsilon}_y}^2+ \mathcal{K}_{int_x}^{\varepsilon}\mathcal{C}_{int}^{-O_{int_x}^{\varepsilon}}+ \mathcal{K}_{int_y}^{\varepsilon}\mathcal{C}_{int}^{-O_{int_y}^{\varepsilon}}+ 2 \beta\mathcal{E}_{L_2}^{G}\right)+ \mathcal{E}_{L_2}^{G},
    \end{equation}
     where $\mathcal{K}_{en_G}=\sqrt{T e^{\left(\frac{ 6 \alpha+3\beta +2}{T_1}\right)T}}$, $T_1 \leq T$, $\mathcal{E}_{L_2}^{G}$ is the generalization error in $L^2$ norm, and $\mathcal{E}_{int_x^\varepsilon}$, $\mathcal{E}_{int_y^\varepsilon}$  related to the training error of interior domain associated with the derivatives of the interior residual, and $\mathcal{K}_{int_x}^\varepsilon\mathcal{C}_{int}^{-O_{int_x}^\varepsilon}$, $\mathcal{K}_{int_y}^\varepsilon\mathcal{C}_{int}^{-O_{int_y}^\varepsilon}$ are the corresponding quadrature error, respectively.  
     \begin{proof}
         The proof follows directly from Theorem \ref{th:gronwall_genearlization_err_energy_norm}, with the additional incorporation of the quadrature rule in Subsection \ref{subsec: Quadrature rules}.
     \end{proof}
\end{theorem}

\begin{remark}
    An inspection of \eqref{genearlization_err_energy_norm} reveals that the error in the energy norm depends not only on the training error but also on quadrature error and the generalization error in the $L^2$ norm. This substantiates our observation that the energy norm gives more error than the $L^2$ norm. However, the inclusion of derivative terms in the loss function enhances the sharpness of the layer in the approximation.
\end{remark}

\begin{remark}
    It is also evident from the proof that the generalization error is inherently tied to the parameter $\varepsilon$ appearing in the energy norm, as well as to the specific conditions imposed on $\varepsilon$ and $\mu$. Altering these conditions will give rise to different bounds that appear in Theorems \ref{th:gronwall_genearlization_err_energy_norm} and \ref{th: gen_err_energy_norm}. However, in each case, the generalization error is bounded above by these estimates.
\end{remark}

\subsection{Theoretical Extension of ENPINN to Coupled System of Problems}\label{subsec: Theoretical Extension of ENPINN to Systems of Equations}
Let $u_s = (u_1, u_2)^T$, $a_1 = \operatorname{diag}(a_{11}, a_{12})$, $a_2 = \operatorname{diag}(a_{21}, a_{22})$, $b = (b_{ij})_{\{i,j\}_1^2}$,
    $D_{\varepsilon} = \operatorname{diag}(\varepsilon, \mu)$, $f_s = (f_1, f_2)^T$. Then the parabolic singularly perturbed convection diffusion systems are given by,
    \begin{equation}\label{eq: system_equation}
        \begin{cases}
            \frac{\partial u_s}{\partial t}-D_\varepsilon \Delta u_s+ {a_{1}}(x,y)\frac{\partial u_s}{\partial x} + {a_2}(x,y)\frac{\partial u_s}{\partial y} + b(x,y) u_s= {f_s(x,y,t)}, \quad (x,y,t) \in \Omega,\\
            u_s(x,y,t)=g_s(x,y,t), \quad (x,y)\in \partial \Omega_s=\overline{\Omega}_s\backslash \Omega_s, ~ t\in [0,T], \\
            u_s(x,y,0)=\xi_s(x,y), \quad (x,y)\in \overline{\Omega}_s,
        \end{cases}
    \end{equation}
    where $g_s = (g_1, g_2)$, $\xi_s = (\xi_1, \xi_2)$, $a_{rp} \geq \alpha > 0$, for $r,p =1,2$ and the reaction matrix satisfies $b_{rr} \geq \beta_0 \geq 0$, $\beta_2 \leq b_{rp} \leq \beta_1 \leq 0$, if $r \neq p$, $r,p = 1,2$, $\beta_3 = - \beta_2 \geq 0$ and $\displaystyle\sum_{p=1}^2 b_{rp} \geq 0$, $r,p =1,2$. These conditions assure the uniqueness of the solutions \cite{kumar2025uniformly, linss2009numerical}.
    The ENPINN loss to approximate the solutions of \eqref{eq: system_equation} is defined as 
    \begin{equation}\label{eq: enpinn_loss_sys}
        \begin{aligned}
            \mathcal{L}_{en}^{C,s}(p) = 
            &\, \lambda_{int}^1\int_{\Omega_s \times [0,T]} \mathcal{N}_{int,p}^1 + \lambda_{int}^2\int_{\Omega_s \times [0,T]} \mathcal{N}_{int,p}^2 + \lambda_{int}^1 \int_{\Omega_s \times [0,T]} \mathcal{N}_{(int,p)_x}^1   \\
            &+\lambda_{int}^2 \int_{\Omega_s \times [0,T]} \mathcal{N}_{(int,p)_x}^2 + \lambda_{int}^1\int_{\Omega_s \times [0,T]} \mathcal{N}_{(int,p)y}^1
            + \lambda_{int}^2\int_{\Omega_s \times [0,T]} \mathcal{N}_{(int,p)y}^2 \\ &+ \lambda_{spb}^1 \int_{\partial\Omega_s \times [0,T]} \mathcal{N}_{spb,p}^1 
            + \lambda_{spb}^2 \int_{\partial\Omega_s \times [0,T]} \mathcal{N}_{spb,p}^2 + \lambda_{ini}^1\int_{\bar{\Omega}_s} \mathcal{N}_{ini,p}^1 \\
            &+ \lambda_{ini}^2 \int_{\bar{\Omega}_s} \mathcal{N}_{ini,p}^2,
        \end{aligned}
    \end{equation}
   where $\mathcal{N}_{int,p}^i$, $\mathcal{N}_{spb,p}^i$, $\mathcal{N}_{ini,p}^i$, $\mathcal{N}_{(int,p)_x}^i$, $\mathcal{N}_{(int,p)_y}^i$ are interior, boundary, initial, and gradient loss components for $i$-th components respectively, $i=1,2$. The dependence of the total error on the ENPINN loss function defined in \eqref{eq: enpinn_loss_sys} is established in the following theorems.

    \begin{theorem}\label{th: sys_gronwall_genearlization_err_l^2_norm}
        Let {$u_s \in {H^2(\Omega) \cap C^0(\overline{\Omega})}$} denote the exact solution of \eqref{eq: system_equation}. Let $u_{s_N,p^*}$ be the approximation of u obtained using the 
        loss function $\mathcal{L}_{N}^{C,s} (p)$  in \eqref{eq: enpinn_loss_sys}. 
        Then the generalization error satisfies the following inequality
        \begin{equation}\label{gronwall_genearlization_err_l^2_norm}
            \|\mathfrak{e}_1\|_{{L^2(\Omega)}}^2 + \|\mathfrak{e}_2\|_{{L^2(\Omega)}}^2= \|u_1-u_{1_N,{{p}^*}}\|_{{L^2(\Omega)}}^2 + \|u_2-u_{2_N,{{p}^*}}\|_{{L^2(\Omega)}}^2 \leq e^{\kappa_{L_2}^s}\times \mathcal{G}_{L_2}^s,
        \end{equation}
        where $\kappa_{L_2}^s={(2\beta_3 + 1)T}$, $\mathcal{G}_{L_2}^s = \|\mathcal{N}_{ini,{p}^*}^1\|_{L^2(\Omega_s)}^2 + \|\mathcal{N}_{ini,{p}^*}^2\|_{L^2(\Omega_s)}^2 + T^{1/2} \|\mathcal{N}_{spb,{p}^*}^1\|_{L^2(\Omega)}  \mathcal{K}_{\partial \Omega_s}^{\nabla \mathfrak{e}_1} + T^{1/2} \|\mathcal{N}_{spb,{p}^*}^2\|_{L^2(\Omega)}  \mathcal{K}_{\partial \Omega_s}^{\nabla \mathfrak{e}_2} + \|\mathcal{N}_{int,{p}^*}^1\|_{L^2(\Omega)}^2 + \|\mathcal{N}_{int,{p}^*}^2\|_{L^2(\Omega)}^2$. The positive constant $\mathcal{K}_{\partial\Omega_s}^{\nabla \mathfrak{e}_i}$ is independent of $\varepsilon$ and depends on the continuous solution $u_i$, the NN approximation $u_{i_N,{{p}^*}}$, $i.e.,$ 
        $$\mathcal{K}_{\partial\Omega_s}^{\nabla \mathfrak{e}_i}\equiv \mathcal{K}_{\partial\Omega_s}^{\nabla \mathfrak{e}_i}(u_i,u_{i_N,{{p}^*}}), \quad i=1,2.$$
        \begin{proof}
            Refer to \ref{sec: appen_Stability for System}.
        \end{proof}
    \end{theorem}

   In addition to the error estimation in $L_2$ norm, the following theorem establishes the corresponding approximation error in energy norm.

    \begin{theorem}\label{th: sys_gronwall_genearlization_err_energy_norm}
         Let {$u_s\in {H^3(\Omega) \cap C^0(\overline{\Omega})}$} satisfies \eqref{eq: system_equation}  and $u_{s_N,{{p}^*}}$ be the ENPINN approximation of $u_s$. Then the energy norm-based error will satisfy the following inequality:
        \begin{equation*}\label{sys_gronwall_genearlization_err_energy_norm}
            \|\mathfrak{e}_1\|_{E}^2+  \|\mathfrak{e}_2\|_{E}^2
            \leq \mathcal{G}_{L_2}^s e^{\kappa_{L_2}^s}+\mathcal{G}_{en}^s e^{\kappa_{en}^s},
        \end{equation*}
         where $T_1 \leq T, \text{and }T \in (0,1]$,  $\kappa_{en}^s= \left(\frac{2\beta_3 +1}{T_1}\right)T$, and
         \begin{equation*}
             \begin{aligned}
                \mathcal{G}_{en}^s
                 =&\, \|\mathcal{N}_{int_x,{p}^*}^{\varepsilon,1}(x,y,t)\|^2_{L^2(\Omega)}
                 + \|\mathcal{N}_{int_y,{p}^*}^{\varepsilon,1}(x,y,t)\|^2_{L^2(\Omega)}
                 + \|\mathcal{N}_{int_x,{p}^*}^{\varepsilon,2}(x,y,t)\|^2_{L^2(\Omega)}\\
                 &+ \|\mathcal{N}_{int_y,{p}^*}^{\varepsilon,2}(x,y,t)\|^2_{L^2(\Omega)}
             \end{aligned}
         \end{equation*}
         and the quantities $\mathcal{G}_{L_2}^s, e^{\kappa_{L_2}^s}$ are defined in Theorem \ref{th: sys_gronwall_genearlization_err_l^2_norm}.
         \begin{proof}
             See details in \ref{sec: appen_Stability for System}.
         \end{proof}
     \end{theorem}

\section{Numerical Experiments}\label{sec:Numerical Experiments}
In this section, we examine several numerical examples to demonstrate the effectiveness of the proposed ENPINN loss function for two- and three-dimensional parabolic convection-diffusion-reaction problems, as well as an application to combustion flow and the Burger equation with an interior layer. {}{To demonstrate the applicability of ENPINN, no a priori knowledge regarding the layer location and layer width has been incorporated.} We compare three types of loss functions to capture the sharpness of the boundary layer, termed as WLPINN \eqref{eq: WLPINN_loss}, gPINN \eqref{eq: WLDPINN_loss}, and ENPINN \eqref{eq: enpinn_loss}.
We evaluate the accuracy of the predicted solutions using the error measures defined in \eqref{eq: l2 error} and \eqref{eq: energy error}. For all numerical experiments, we employ the dual-nested Tanh \cite{wang2025adaptive} activation function across all problems. 
To validate the theoretical claim, we examine the behavior of the generalization errors as the numbers of spatial and temporal partitions, namely $N_x$, $N_y$, and $N_t$, increase.
First, we begin our experiments with a two-parameter problem having sharp gradients with variable coefficients, for which the exact solution is known.

\begin{example}\label{ex: num_ex_var_2d}
    {As a test problem, consider the time-dependent convection diffusion reaction model in 2D as defined in \eqref{eq: main_cont_prob_lin}, where $a_1= (3-x)$, $a_2= (3-y)$ and $\varepsilon= 10^{-3}$, $\mu=10^{-4}$.}
    The source function $f$ is chosen in such a way that the exact solution is 
    \begin{align*}
        u(x,y)=0.25 &\times \left(1+\dfrac{\sin{8x}}{2}\right) \times \left(1-\exp(-t)\right)\times \left(1-\exp\left({- \lambda_1x}{\mathfrak{W}_1}\right)\right)\times \left(1-\exp\left({-y}\mathfrak{W}_2\right)\right)\\
        &\times \left(1-\exp\left({- \lambda_2(1-x) \mathfrak{W}_1}\right)\right) \times \left(1-\exp\left({-(1-y)}\mathfrak{W}_2\right)\right),
    \end{align*}
    where $\lambda_{1,2}=\left(\sqrt{1+\frac{16 \varepsilon}{\mu^2}}\right)\mp 1, \mathfrak{W}_1=\frac{\mu}{2 \varepsilon}, \mathfrak{W}_2=\frac{1}{\sqrt{\varepsilon}}.$
The exact heatplot for the solution of Example \ref{ex: num_ex_var_2d} at $t=1$ is highlighted in the Figure \ref{fig:exact_heat_t1.0_N25_var}. In this problem, training consists of $10k$ epochs of Adam followed by $10k$ epochs of L-BFGS.
The predicted solution of ENPINN is computed over $ N_x = N_y = 25$, $N_t=15$, which is illustrated in the Figure \ref{fig:pred_h_t1.0_N25_w_var}, and pointwise error between the exact and ENPINN prediction is shown in Figure \ref{fig:err_h_t1.0_N25_w_var}. It can be observed that the maximum pointwise error is approximately $2\times10^{-2}$. 

\begin{figure}[h]
    \centering
    \begin{subfigure}{0.3\textwidth}
        \centering
        \includegraphics[width=4.7cm, height=3.57cm]{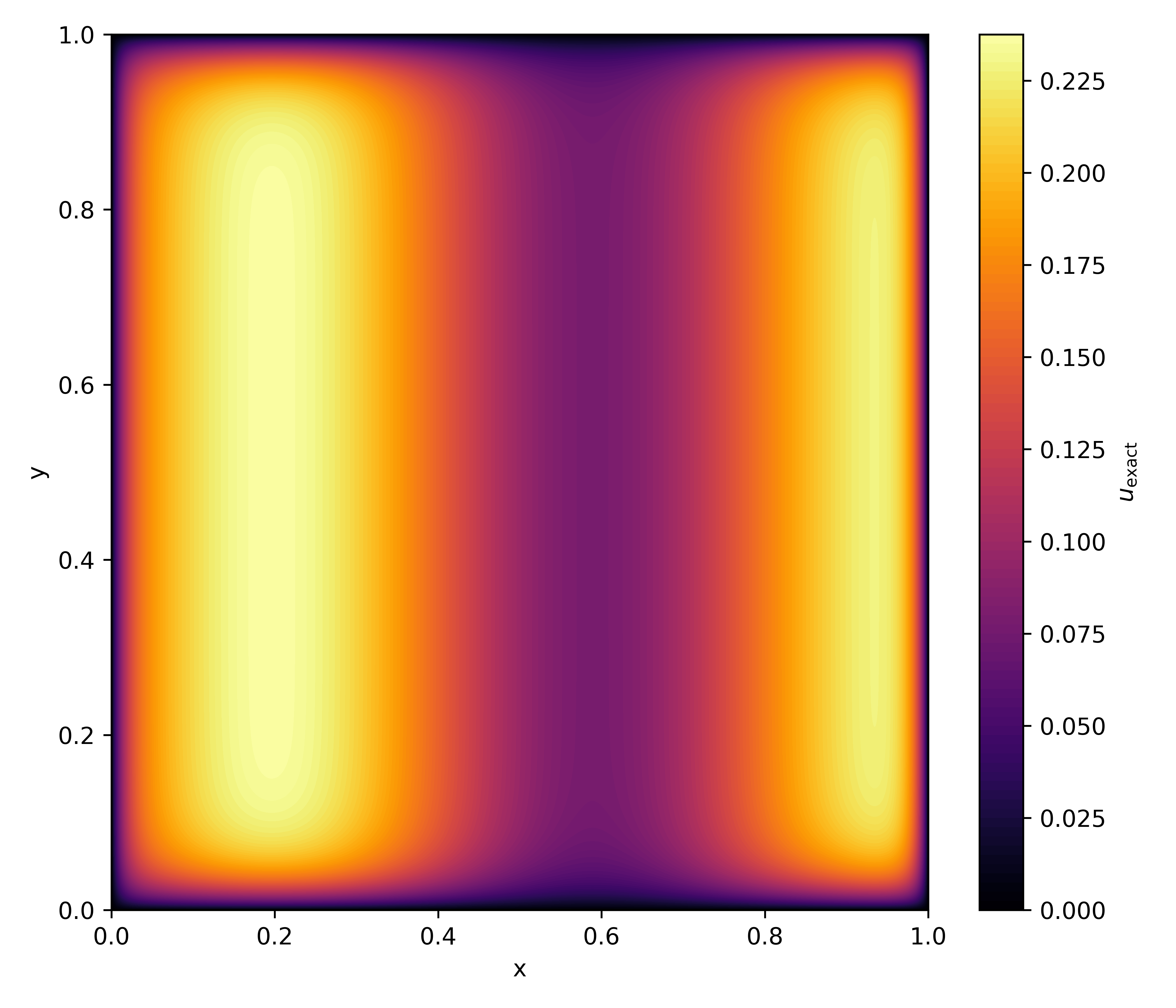}
    \caption{Exact heatmap at $t=1$}
    \label{fig:exact_heat_t1.0_N25_var}
    \end{subfigure}
    \begin{subfigure}{0.3\textwidth}
        \centering
        \includegraphics[width=5cm, height=4cm]{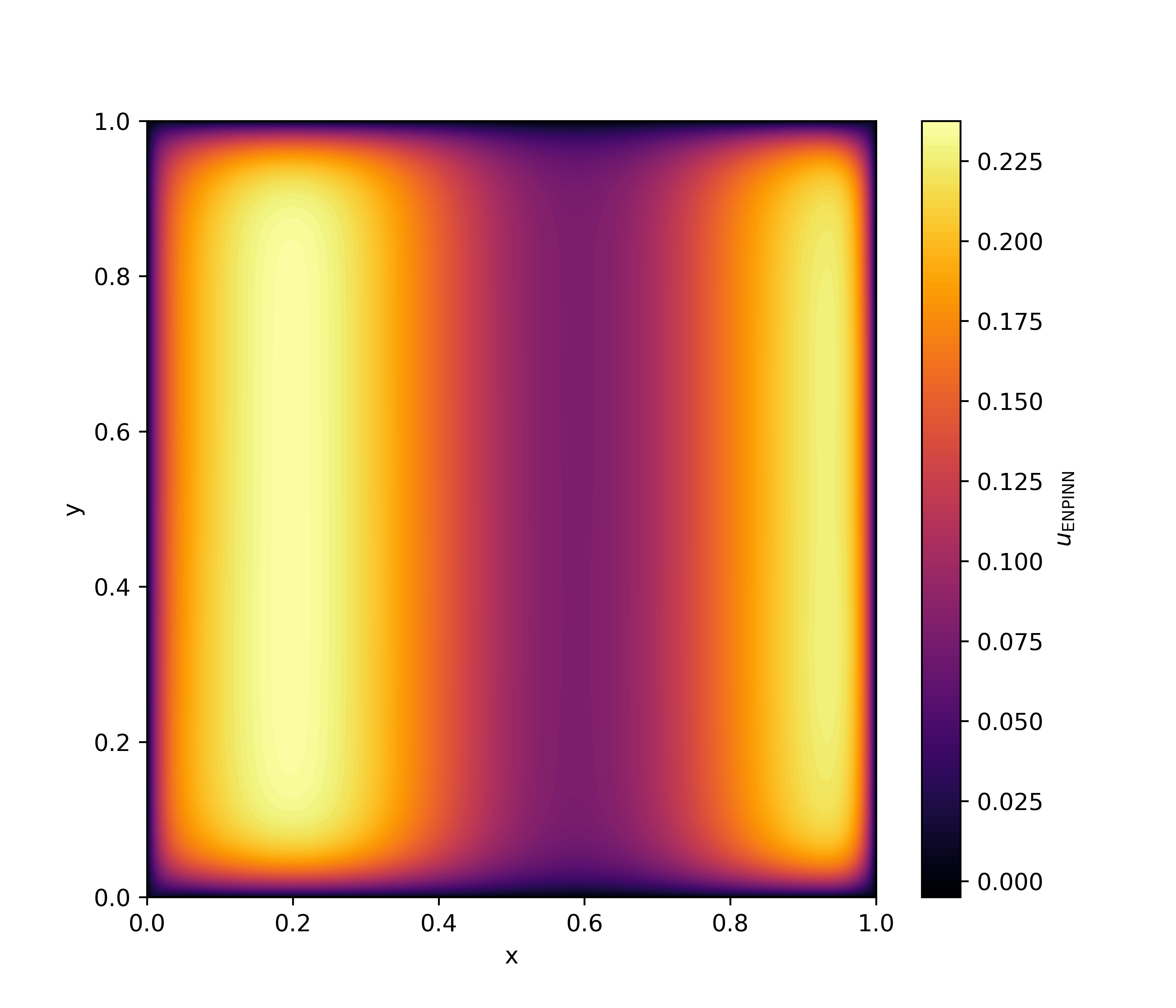}
    \caption{ENPINN prediction at $t=1$}
    \label{fig:pred_h_t1.0_N25_w_var}
    \end{subfigure}
    \begin{subfigure}{0.3\textwidth}
        \centering
        \includegraphics[width=5cm, height=4cm]{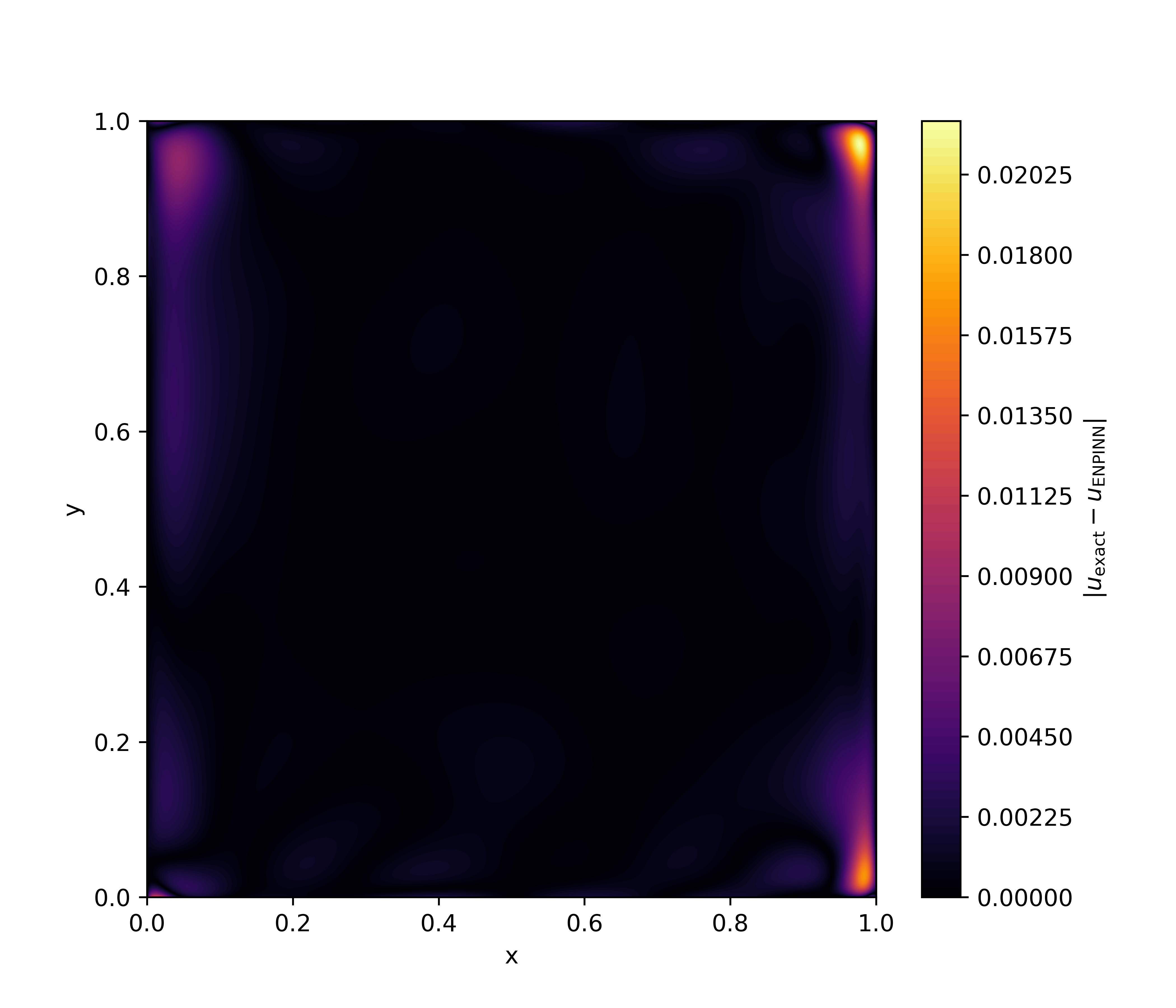}
    \caption{Pointwise error at $t=1$}
    \label{fig:err_h_t1.0_N25_w_var}
    \end{subfigure}
    \caption{Exact heatmap, ENPINN prediction, and pointwise error at $t=1$ with $N_x=N_y=25, N_t=15$ for Example \ref{ex: num_ex_var_2d}}
\end{figure}
 
Furthermore, we compare the $L^2$ and energy error for the ENPINN algorithm as the number of uniform partitions increases. Consequently, we find that both $L^2$ and the energy error decrease as the number of partitions increases, which supports our theoretical claim. Hence, the ENPINN effectively captures the sharp gradients and layer structure, as evidenced by the energy error and $L^2$ error presented in Table \ref{table: comparison_num1_experiment}.

\begin{table}[ht]
    \centering
    \caption{Comparison between Exact solution and ENPINN predicted solution for Example \ref{ex: num_ex_var_2d}.}
    \label{table: comparison_num1_experiment}
    \renewcommand{\arraystretch}{1.5}
    \begin{tabular}{c c c}
        \hline
        \multirow{2}{*}{\textbf{Partition}} 
        & \multicolumn{2}{c}{\makecell{\textbf{ENPINN}}} \\ \cline{2-3}
        & {$\mathcal{E}^{L^2}_G$} & {$\mathcal{E}^{en}_G$} \\ 
        \hline
        
        $N_x=5$, $N_y=5$, $N_t=5$  
        & $3.9\times 10^{-2} \pm 1.2 \times 10^{-2}$  & $1.2\times 10^{-1} \pm 1.2 \times 10^{-2}$  \\ 
        \hline
        
        $N_x=15$, $N_y=15$, $N_t=10$  
        & $3.4\times 10^{-4} \pm 6.8 \times 10^{-5}$ & $7.9 \times 10^{-3} \pm 2.4 \times 10^{-3}$ \\ 
        \hline

        $N_x=25$, $N_y=25$, $N_t=15$  
        & $2.8\times 10^{-4} \pm 1.7 \times 10^{-4}$ & $2.1\times 10^{-3} \pm 9.1\times 10^{-4}$  \\
        \hline
    \end{tabular}
\end{table}

To show the effectiveness of the ENPINN, we compare ENPINN with other PINN variants like VSPINN \cite{vspinn}, WLPINN, and gPINN \cite{yu2022gradient} over the same hyperparameters. Figures \ref{fig:Ep_L2_Err_Comp} and \ref{fig:Ep_Ener_Err_Comp_var} depict the progression of energy and $L^2$ errors over the training epochs. To assess neural network variability, each PINN variant is trained with multiple random seeds. The solid line represents the mean errors ($L_2$ or energy) over the training epochs, while the shaded region represents the corresponding standard deviation.
The results demonstrate that ENPINN achieves significantly lower energy and $L^2$ errors compared to VSPINN. Furthermore, WLPINN's predictions indicate that excluding the residual derivative in ENPINN results in lower accuracy.

In addition, the performance of gPINN suggests that incorporating the derivative of the test function is crucial for capturing the layer.

\begin{figure}[ht]
  \centering
  
  \begin{subfigure}[b]{0.48\textwidth}
    \includegraphics[width=\textwidth]{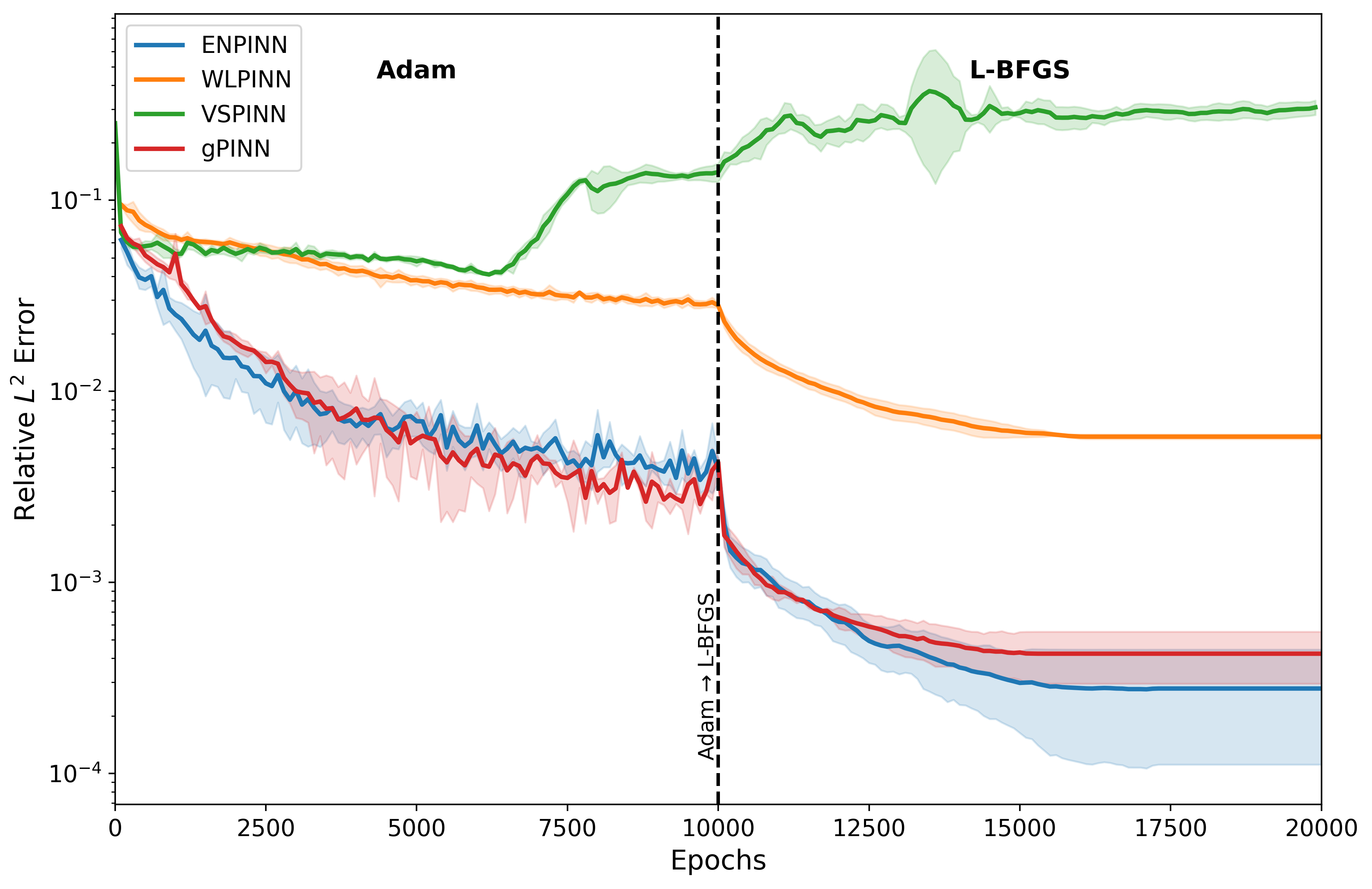}
    \caption{Comparison of $L^2$ error versus epochs}
    \label{fig:Ep_L2_Err_Comp}
  \end{subfigure}
  \begin{subfigure}[b]{0.48\textwidth}
    \includegraphics[width=\textwidth]{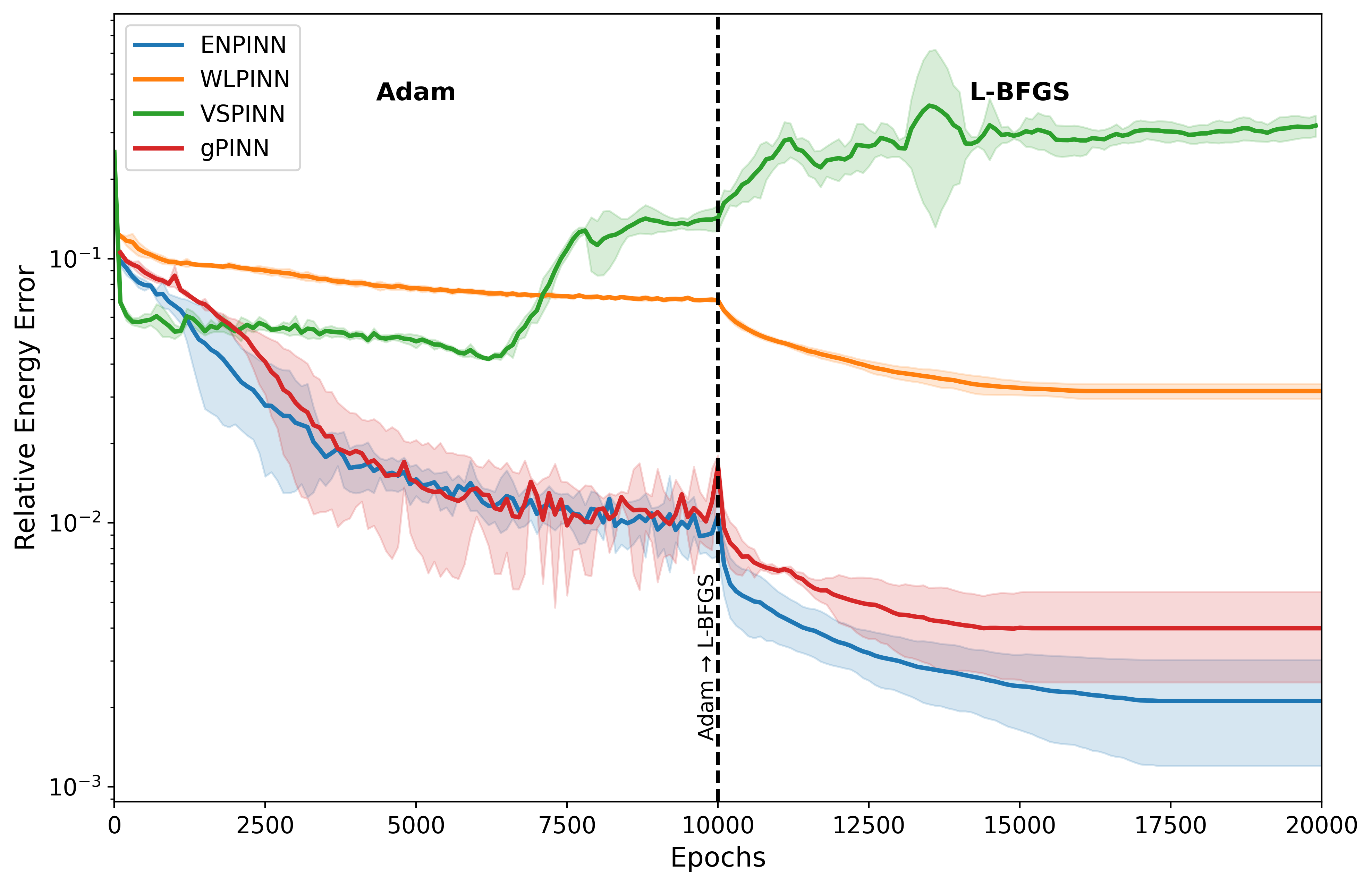}
    \caption{Comparison of energy error versus epochs}
    \label{fig:Ep_Ener_Err_Comp_var}
  \end{subfigure}
  \caption{Comparison of ENPINN with the different PINN variants for Example \ref{ex: num_ex_var_2d}}
\end{figure}

\end{example}

\begin{example}\label{ex: numerical_example_app}{}{(Combustion Flow)}
We now consider a nonlinear thermal explosion model arising in combustion theory with a nonlinear source function. In particular, we take {}{$f(x,y,u)=\sigma e ^{\gamma u}$, $ a^* = (1,1)$ in \eqref{eq: main_cont_prob_nonlin} with $\sigma>0$, $\gamma>0$ as defined in \cite{pao2012nonlinear}}. Here, we demonstrate ENPINN's superiority for thermal explosion. The numerical reference solution is presented in Figure \ref{fig: ex_h_t_0.5_app}. It is computed using the implicit Backward Difference Formula (BDF) solver in SciPy. The second-order spatial derivatives are computed using a central finite-difference scheme for the diffusion term, and an upwind scheme is used to compute the first-order derivatives. The corresponding ENPINN prediction for $\epsilon = 10^{-3}$, $\mu = 10^{-4}$
is presented in Figure \ref{fig: pred_h_t_0.5_app}. For the ENPINN implementation, we use
$ N_x=N_y=25$ and $ N_t=15$ as the number of uniform partitions, respectively. The network is trained using $5000$ epochs with the Adam optimizer, followed by $5000$ epochs with L-BFGS. 
The prediction near the boundary is prominent, indicating that ENPINN accurately captures the boundary conditions. Moreover, even a slight deviation from the boundary leads to greater variation in the solution profile, reflecting the presence of sharp boundary layers. 

\begin{figure}[ht]
    \centering
    
    \begin{subfigure}[b]{0.45\textwidth}
        \centering
        \includegraphics[width=0.8\linewidth]{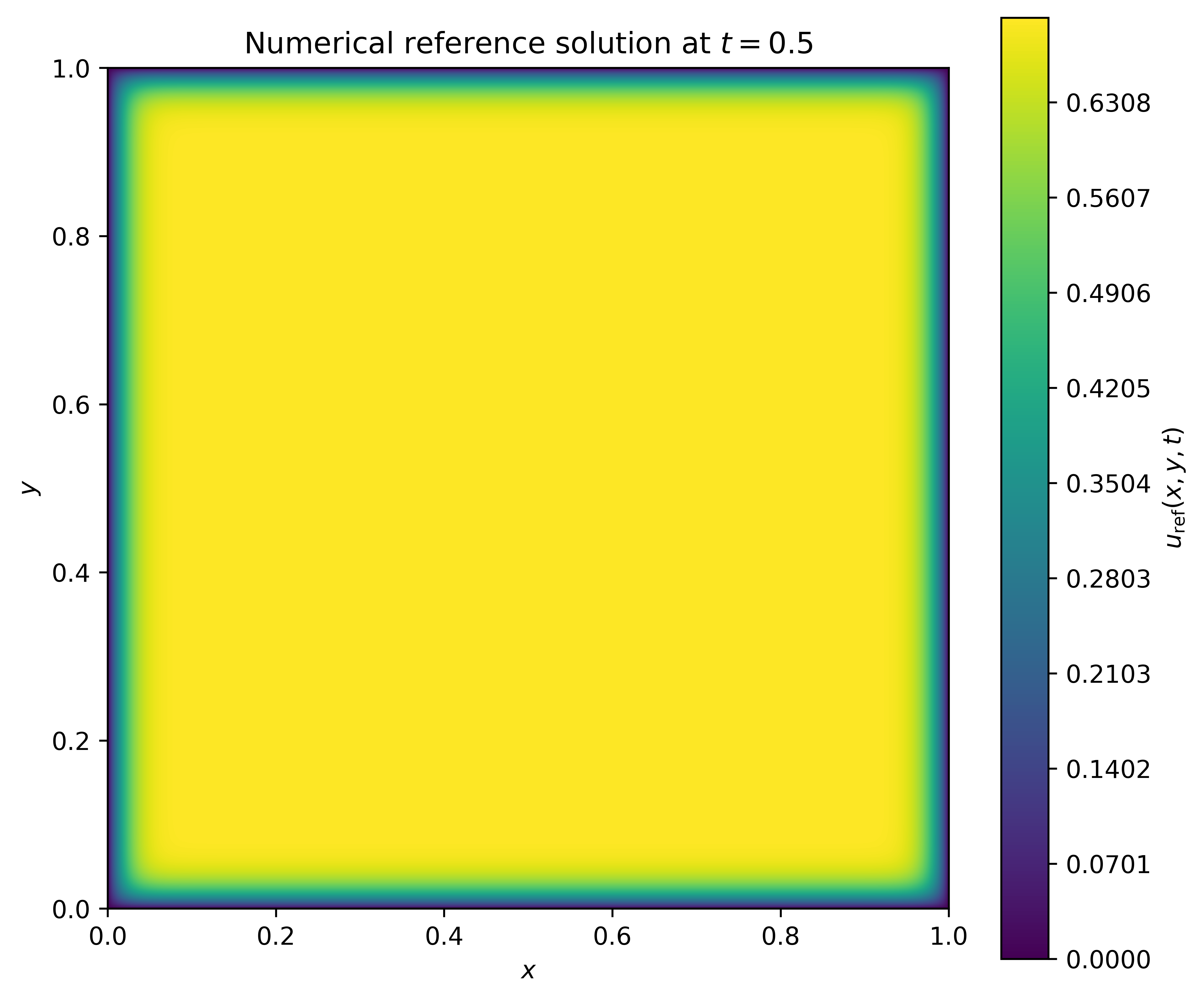}
        \caption{Reference Solution}
        \label{fig: ex_h_t_0.5_app}
    \end{subfigure}
    \hfill
    \begin{subfigure}[b]{0.45\textwidth}
        \centering
        \includegraphics[width=0.8\linewidth]{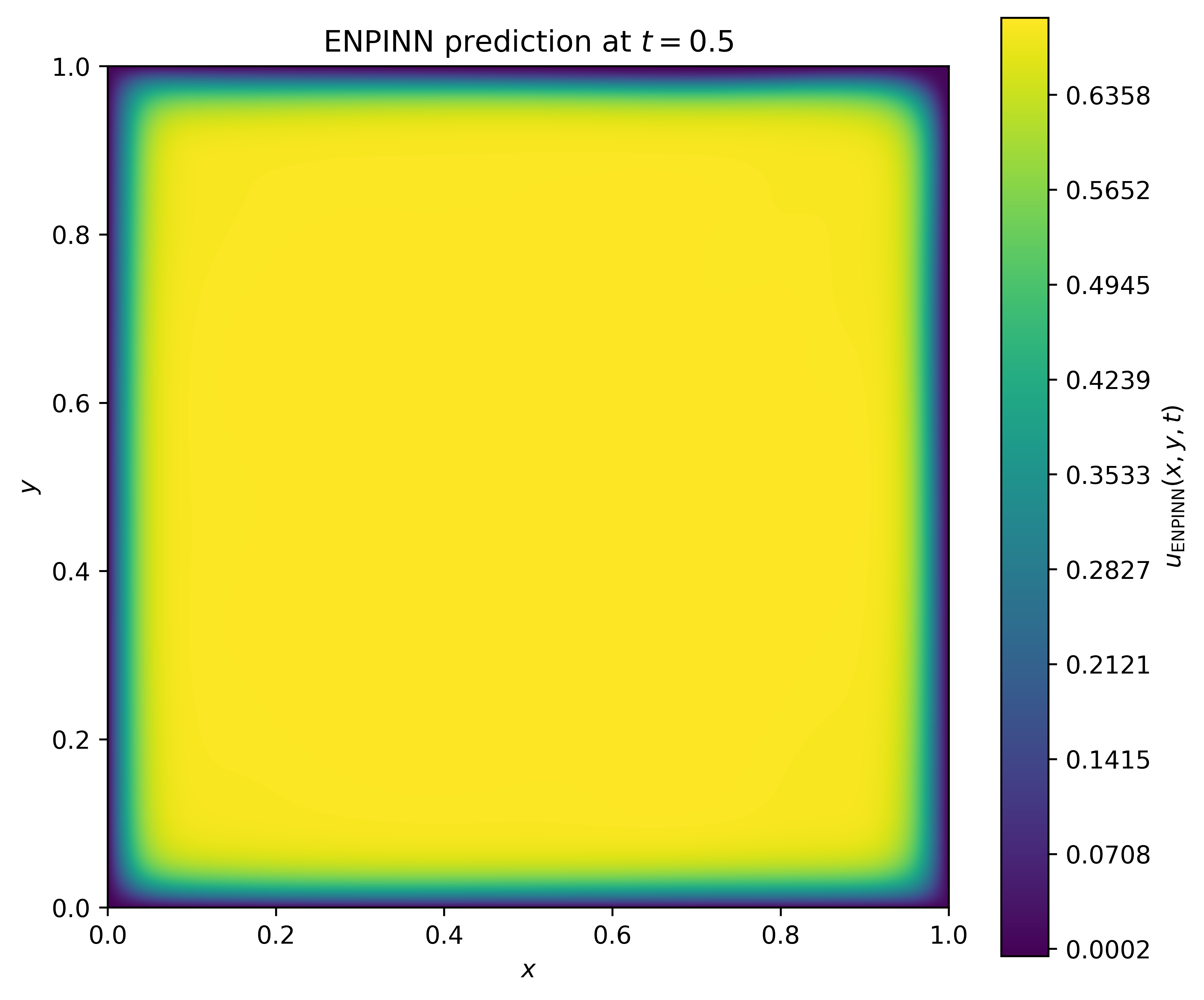}
        \caption{ENPINN Prediction}
        \label{fig: pred_h_t_0.5_app}
    \end{subfigure}
    \caption{Reference solution and ENPINN prediction at $t=0.5$ with $N_x=N_y=25$ and $N_t=15$ for Example~\ref{ex: numerical_example_app}.}
    \label{fig:reference_enpinn_app}
\end{figure}

We provide the total error in $L_2$ and energy norm with different partitions for Example \ref{ex: numerical_example_app} in Table \ref{table: comparison_num2_experiment}.

\begin{table}[ht]
    \centering
    \caption{Comparison of ENPINN predicted solution for Example \ref{ex: numerical_example_app}.}
    \label{table: comparison_num2_experiment}
    \renewcommand{\arraystretch}{1.5}
    \begin{tabular}{c c c}
        \hline
        \multirow{2}{*}{\textbf{Partition}} 
        & \multicolumn{2}{c}{\makecell{\textbf{ENPINN}}} \\ \cline{2-3}
        & {$\mathcal{E}^{L^2}_G$} & {$\mathcal{E}^{en}_G$} \\ 
        \hline
        
        $N_x=5$, $N_y=5$, $N_t=5$  
        & $2.9\times 10^{-1} \pm 1.0 \times 10^{-2}$  & $3.5\times 10^{-1} \pm 1.1 \times 10^{-2}$  \\ 
        \hline
        
        $N_x=15$, $N_y=15$, $N_t=10$  
        & $4.1\times 10^{-2} \pm 6.1 \times 10^{-3}$ & $7.7 \times 10^{-2} \pm 1.1 \times 10^{-2}$ \\ 
        \hline

        $N_x=25$, $N_y=25$, $N_t=15$  
        & $2.2\times 10^{-2} \pm 3.9 \times 10^{-3}$ & $4.3\times 10^{-2} \pm 8.4\times 10^{-4}$  \\
        \hline
    \end{tabular}
\end{table}
\end{example}

\begin{example}\label{ex: num_3d_exam}
    Performance of the proposed ENPINN on the 3D two-parametric PDEs:

    \begin{equation}\label{eq: numerical_example_3d}
            \begin{cases}
                u_t-\varepsilon \Delta u+{\mu} (3-x,3-y,3-z). \nabla u+u=f, \quad (x,y,z,t) \in (0,1)^3 \times (0,1],\\
                u(x,y,z,t)=0, \quad (x,y,z)\in (0,1)^3, ~ t\in [0,1], \\
                u(x,y,z,0)=0, \quad (x,y,z)\in (0,1)^3.
            \end{cases}
        \end{equation}
        The source function $f$ is chosen such that the solution admits the following exact solution
        \begin{align*}
            u(x,y)=\frac{1}{4}& \times \left(1-\exp(-t)\right)\times \left(1-\exp\left({- \lambda_1x}{\mathfrak{W}_1}\right)\right)\times \left(1-\exp\left({-y}\mathfrak{W}_2\right)\right) \times \left(1-\exp\left({-z}\mathfrak{W}_2\right)\right)\\
            \times& \left(1-\exp\left({- \lambda_2(1-x) \mathfrak{W}_1}\right)\right) \times \left(1-\exp\left({-(1-y)}\mathfrak{W}_2\right)\right) \times \left(1-\exp\left({-(1-z)}\mathfrak{W}_2\right)\right),
        \end{align*}
        where $\varepsilon = 10^{-5}$, $\mu = 10^{-2}$, and $\mathfrak{W}_{1,2}$, $\lambda_{1,2}$ are defined as in Example \ref{ex: num_ex_var_2d}.

    We compare the evolution of the $L^2$ and energy errors across epochs for ENPINN, gPINN, and WLPINN. Note that this example corresponds to the regime $\mu^2 \geq \epsilon$. In this case, gPINN performs poorly compared to WLPINN and ENPINN, primarily due to sharper gradients. For consistency with Example \ref{ex: num_ex_var_2d}, we have trained the model with $10k$ epochs of Adam followed by $10k$ epochs of L-BFGS. Figure \ref{fig:Ep_L2_Err_Comp_3d} depicts that WLPINN achieves better relative $L_2$ accuracy than ENPINN during most of the training process. However, the relative $L_2$ error of WLPINN begins to plateau after $16k$ iterations, whereas the $L_2$ error associated with ENPINN continues to decrease up to $20k$ iterations. This indicates the further improvement of ENPINN during the later stage of L-BFGS optimization. But a more improved behavior in the energy error of ENPINN is observed in Figure \ref{fig:Ep_Ener_Err_Comp_3d}, where ENPINN achieves better accuracy than both WLPINN and gPINN.

    \begin{figure}[h]
      \centering
      
      \begin{subfigure}[b]{0.48\textwidth}
        \includegraphics[width=\textwidth]{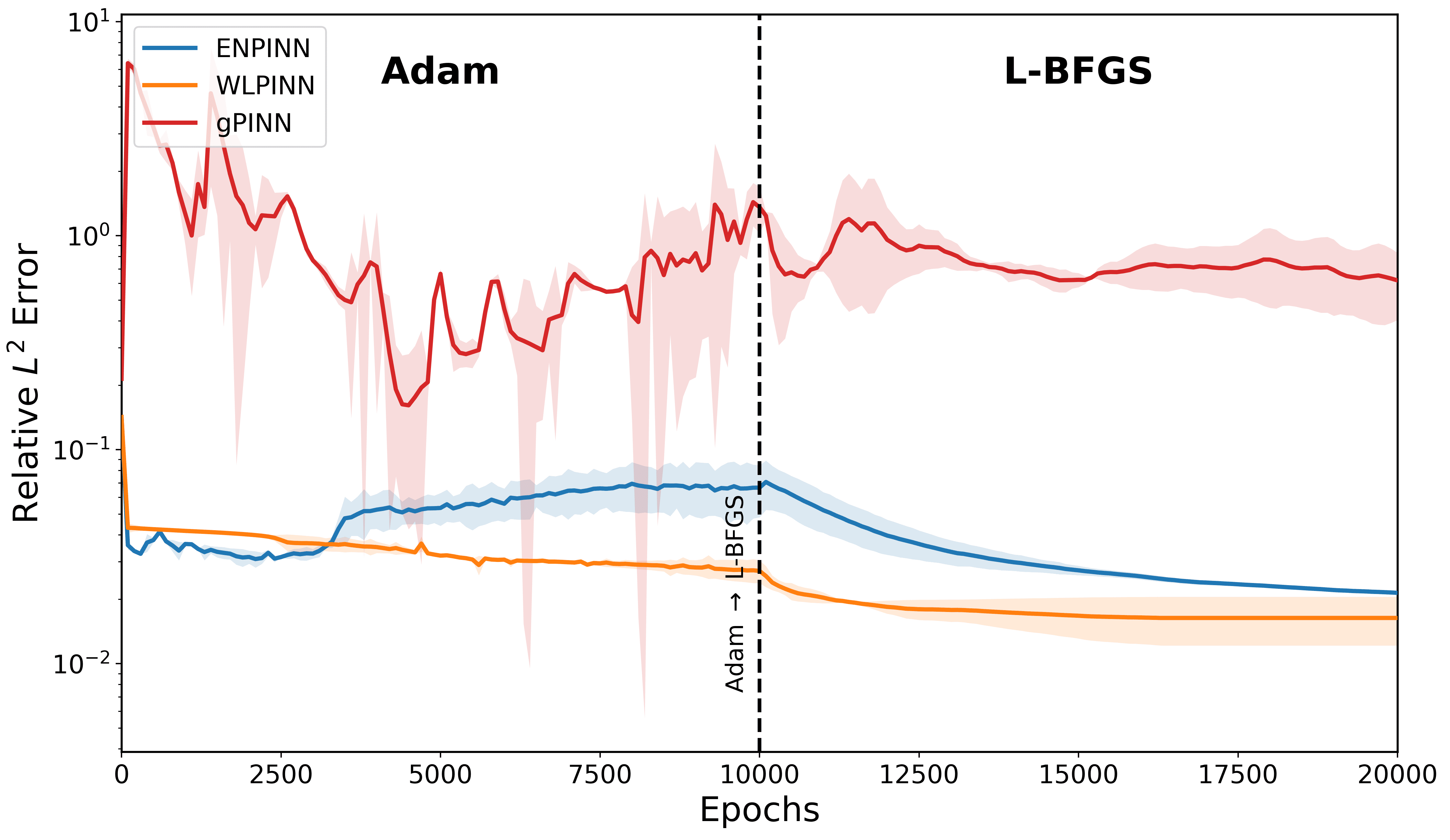}
        \caption{$L^2$ error}
        \label{fig:Ep_L2_Err_Comp_3d}
      \end{subfigure}
      \begin{subfigure}[b]{0.48\textwidth}
        \includegraphics[width=\textwidth]{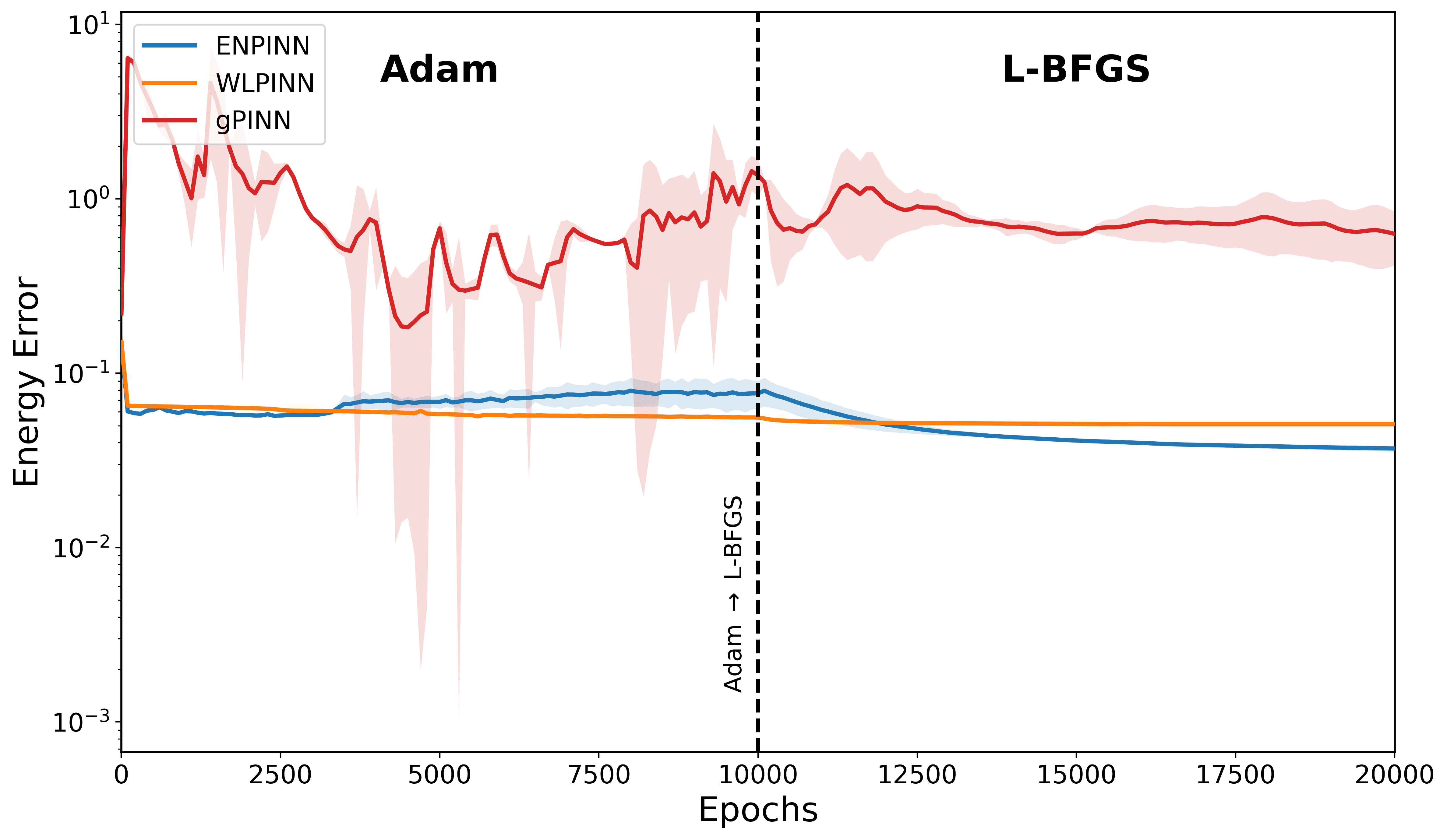}
        \caption{Energy error}
        \label{fig:Ep_Ener_Err_Comp_3d}
      \end{subfigure}
      \caption{Comparison of $L^2$ and energy error over epochs with different PINN variants at final time for Example \ref{ex: num_3d_exam}}.
    \end{figure}
\end{example}

\begin{example}\label{ex: burger_eq}
    Performance of the proposed ENPINN on the 2D Burger equation:
    \begin{equation}\label{eq: burger_pde_equation}
        \begin{cases}
                u_t-\varepsilon \Delta u+u (u_x+u_y)=0, \quad (x,y,t) \in (-1,1)^2 \times (0,1],\\
                u(x,y,t)=0, \quad (x,y)\in (-1,1)^2, ~ t\in [0,1], \\
                u(x,y,0)=-\sin{\pi x}\sin{\pi y}, \quad (x,y)\in (-1,1)^2,
            \end{cases}
    \end{equation}
    where $\varepsilon=0.01/\pi$.
    
    Figure \ref{fig: burger_ex_cont_xt_fix_y} presents the reference solution for a fixed value $y=0.5$ as time evolves. It reveals the existence of an interior layer in the solution of \eqref{eq: burger_pde_equation} near $x=0$. The temporal evolution along the $x$ direction for the fixed value $y=0.5$ is depicted in Figure \ref{fig: burger_pred_cont_xt_fix_y}, where ENPINN effectively captures the interior layer near $x=0$. The error between the ENPINN predicted solution and the reference solution in the energy norm is $8.8 \times 10^{-2} \pm 3.21 \times 10^{-3}$.
    
   \begin{figure}[h]
      \centering
      
      \begin{subfigure}[b]{0.45\textwidth}
        \includegraphics[width=\textwidth]{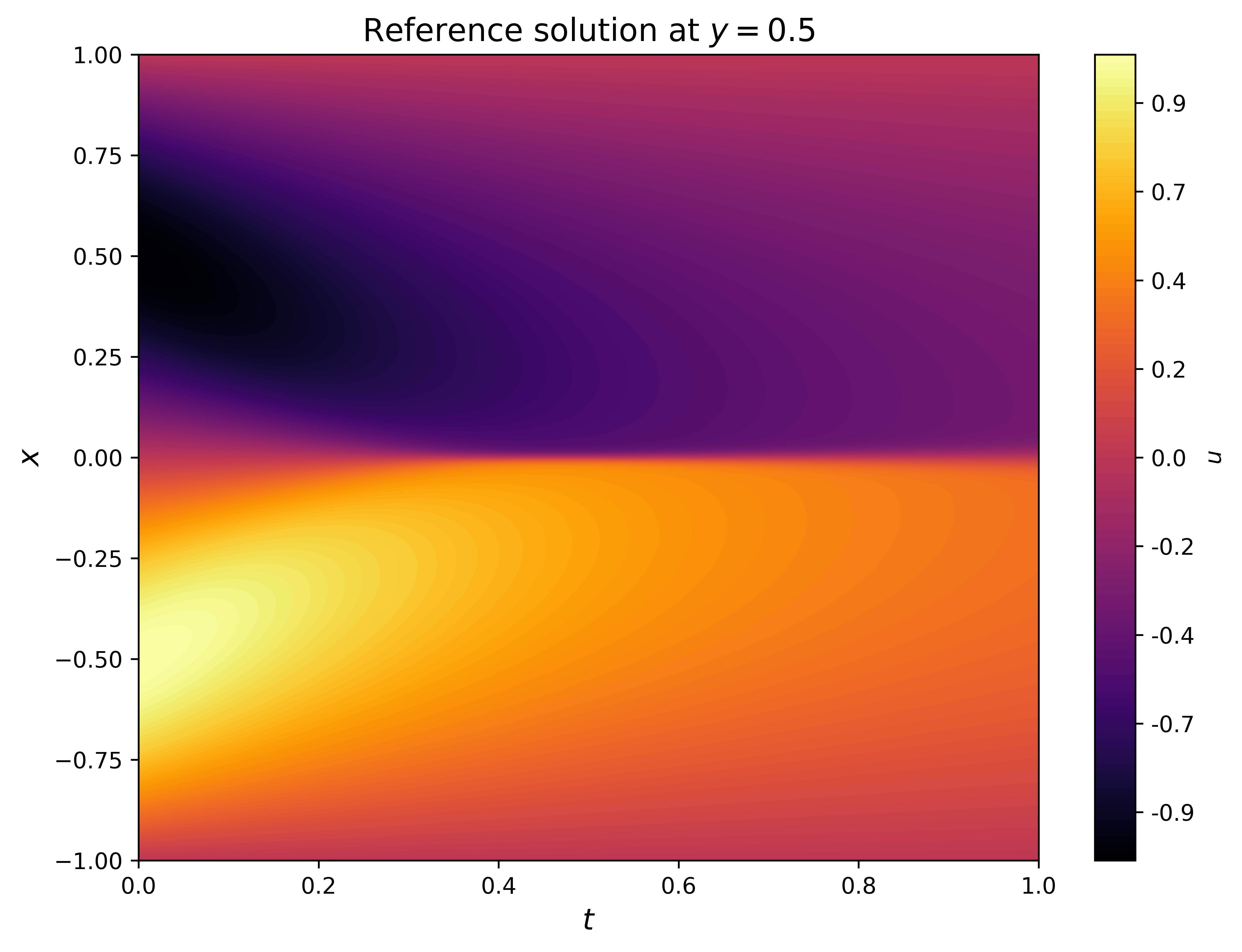}
        \caption{Reference Solution at $y=0.5$}
        \label{fig: burger_ex_cont_xt_fix_y}
      \end{subfigure}
      \begin{subfigure}[b]{0.45\textwidth}
        \includegraphics[width=\textwidth]{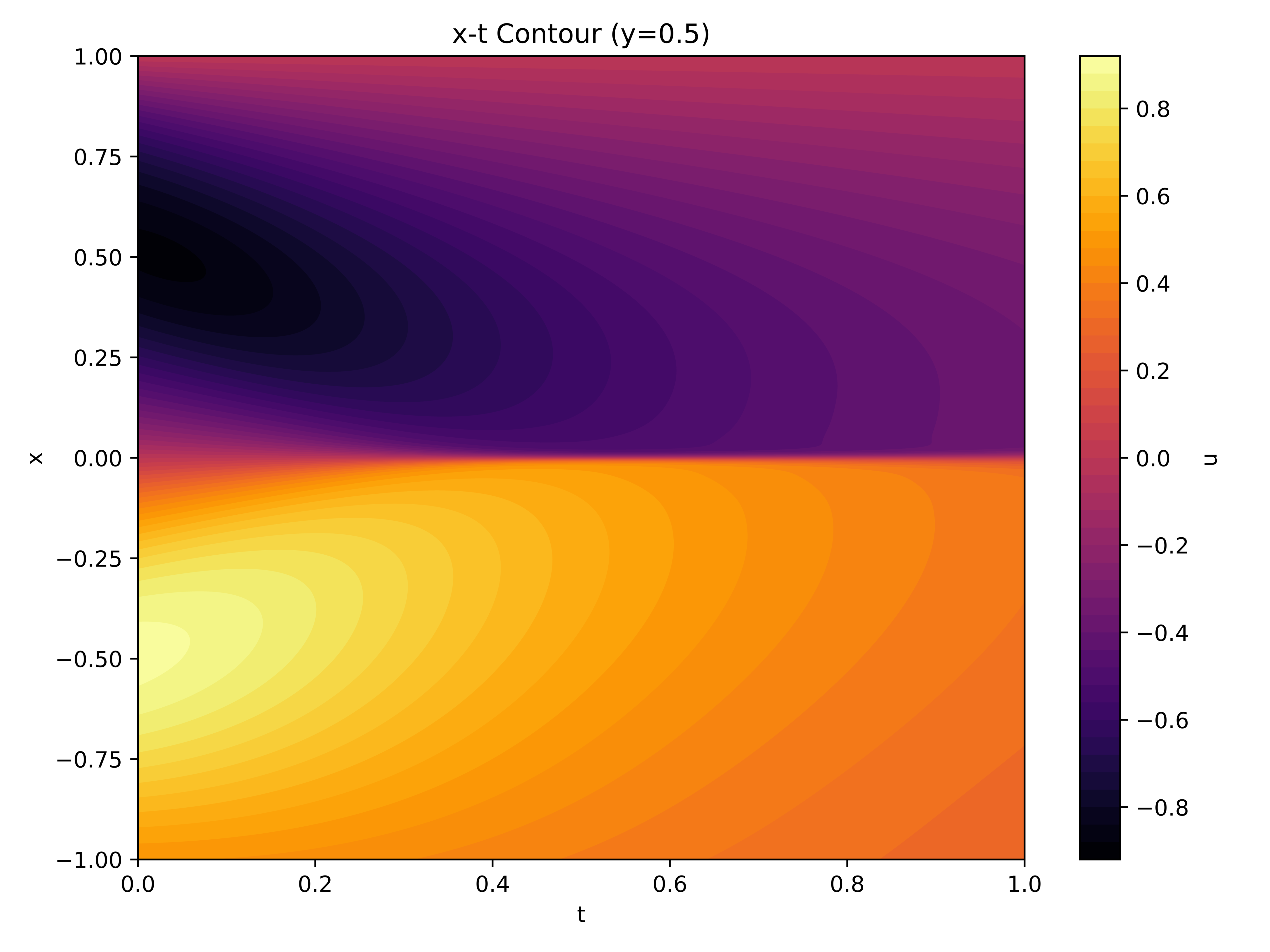}
        \caption{Prediction at $y=0.5$}
        \label{fig: burger_pred_cont_xt_fix_y}
      \end{subfigure}
      \caption{ENPINN prediction with a fixed $x$ and $y$ value for Example \ref{ex: burger_eq}.}
    \end{figure}
\end{example}

    \begin{example}\label{ex: sys_ex}
        As a test problem, consider the time-dependent convection diffusion system defined in \eqref{eq: system_equation}, where 
        \begin{align*}
            &a_1 = \operatorname{diag}(1,1), \quad a_2 = \operatorname{diag}(1,1),\\
           & b = 
            \begin{pmatrix}
                10 \quad -10s_q  \\
                -20s_q \quad 20
            \end{pmatrix}, 
            \quad s_q = (2^{16})x^4 y^4 (1-x)^4 (1-y)^4,\\
           & D_{\varepsilon} = \operatorname{diag}(10^{-4},10^{-2}),\\
           & f = ((1 - \exp{(-5t)})(x+y) + 5xy,(1 - \exp{(-10t)})(x+y) + 10xy)^T,\\
           & g_s = ((1-\exp{(-5t)})xy,(1-\exp{(-10t)})xy)^T, \quad \xi_s = (0,0)^T.
        \end{align*}
        
        We use uniform partitions with $N_x = 18 = N_y$, and $N_t = 16$. The network is trained for 15000 epochs using the Adam optimizer, followed by an additional 15000 epochs of L-BFGS optimization.
        The ENPINN prediction of the first component $u_1$, at $ t=1.0$, is presented in Figure \ref{fig: u1_pred}. The observation indicates the development of a boundary layer near $x=1.0$. Similarly, the prediction of the second component $u_2$ at $t=1.0$ shown in Figure \ref{fig: u2_pred}, also indicates a boundary layer near $x=1.0$.
        \begin{figure}[ht]
            \centering
            \begin{subfigure}{0.48\linewidth}
                \centering
                \includegraphics[width=0.8\linewidth]{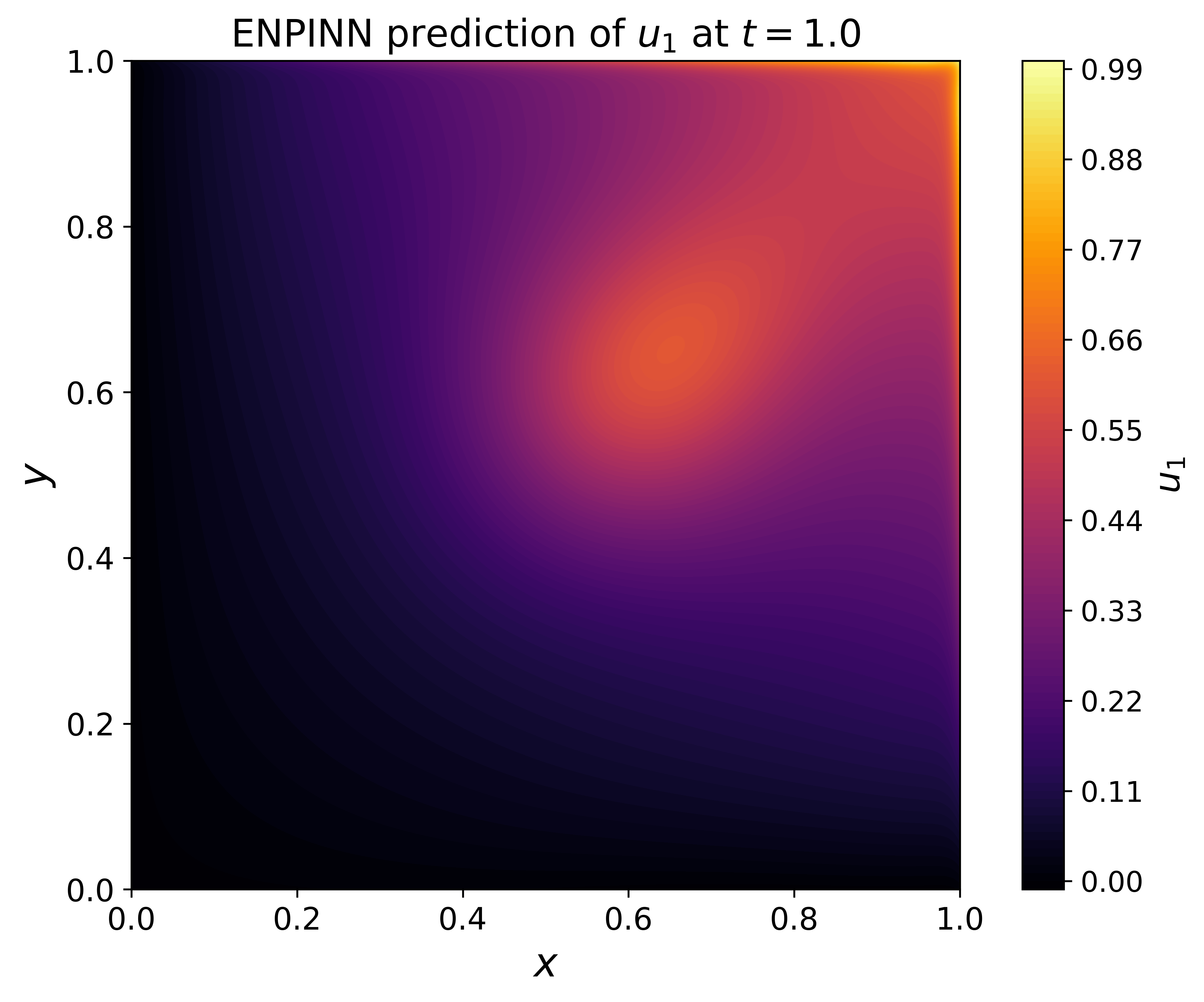}
                \caption{Prediction of the first component}
                \label{fig: u1_pred}
            \end{subfigure}
            \hfill
            \begin{subfigure}{0.48\linewidth}
                \centering
                \includegraphics[width=0.8\linewidth]{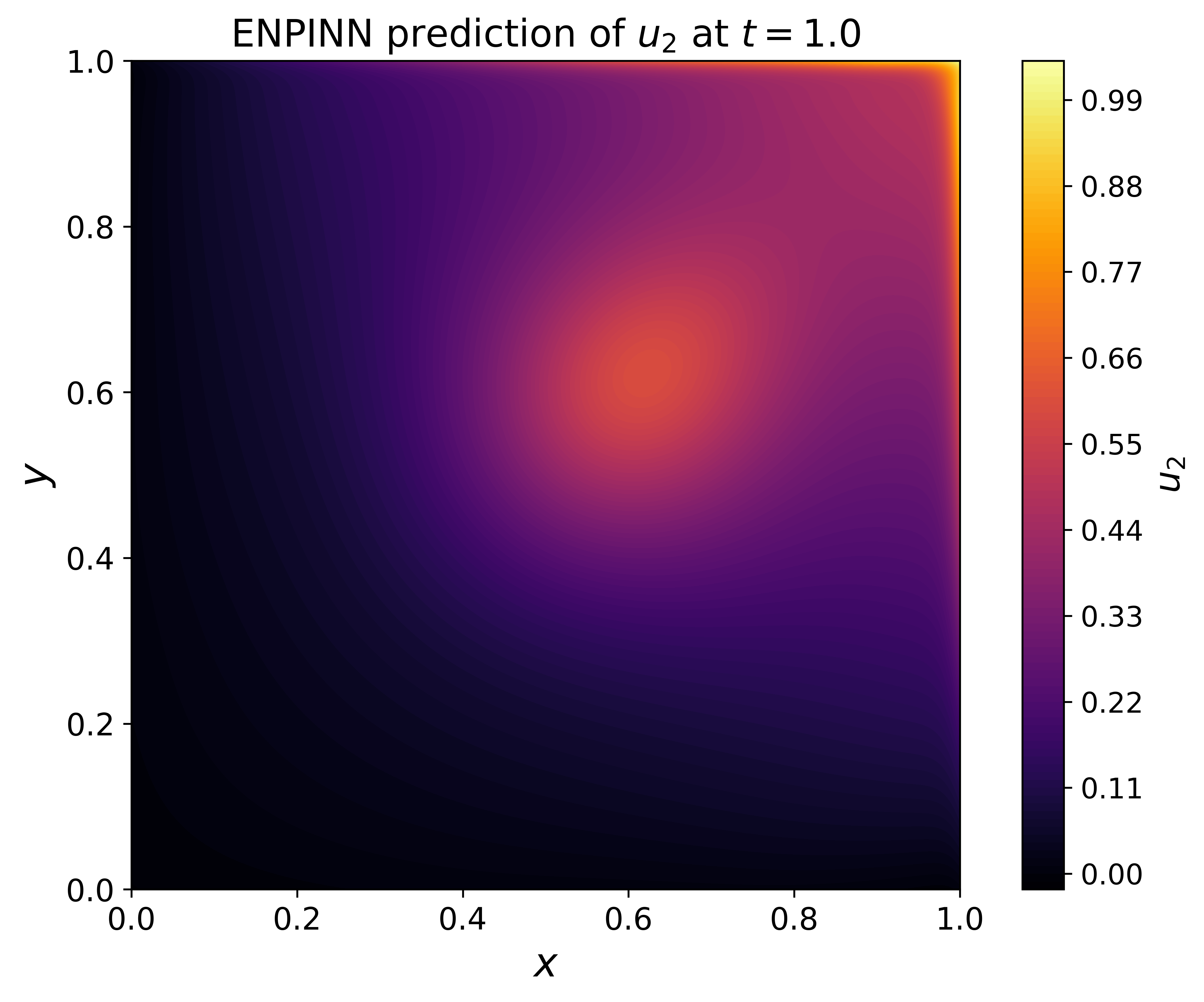}
                \caption{Prediction of the second component.}
                \label{fig: u2_pred}
            \end{subfigure}
        
            \caption{ENPINN Predicted solutions for the two components of the system \ref{ex: sys_ex}.}
            \label{fig: pred_system}
        \end{figure}
         To evaluate the accuracy of ENPINN, we first computed a reference solution using Scipy. The pointwise errors between the reference and predicted solutions are illustrated in Figures \ref{fig: u1_err} and \ref{fig: u2_err}. The resulting $L_2$ error is $9.3 \times 10^{-3} \pm 9.6 \times 10^{-4}$, while the corresponding energy error is $6.2\times 10^{-2} \pm 1.21 \times 10^{-3}$.
        \begin{figure}[ht]
            \centering
            \begin{subfigure}{0.48\linewidth}
                \centering
                \includegraphics[width=0.8\linewidth]{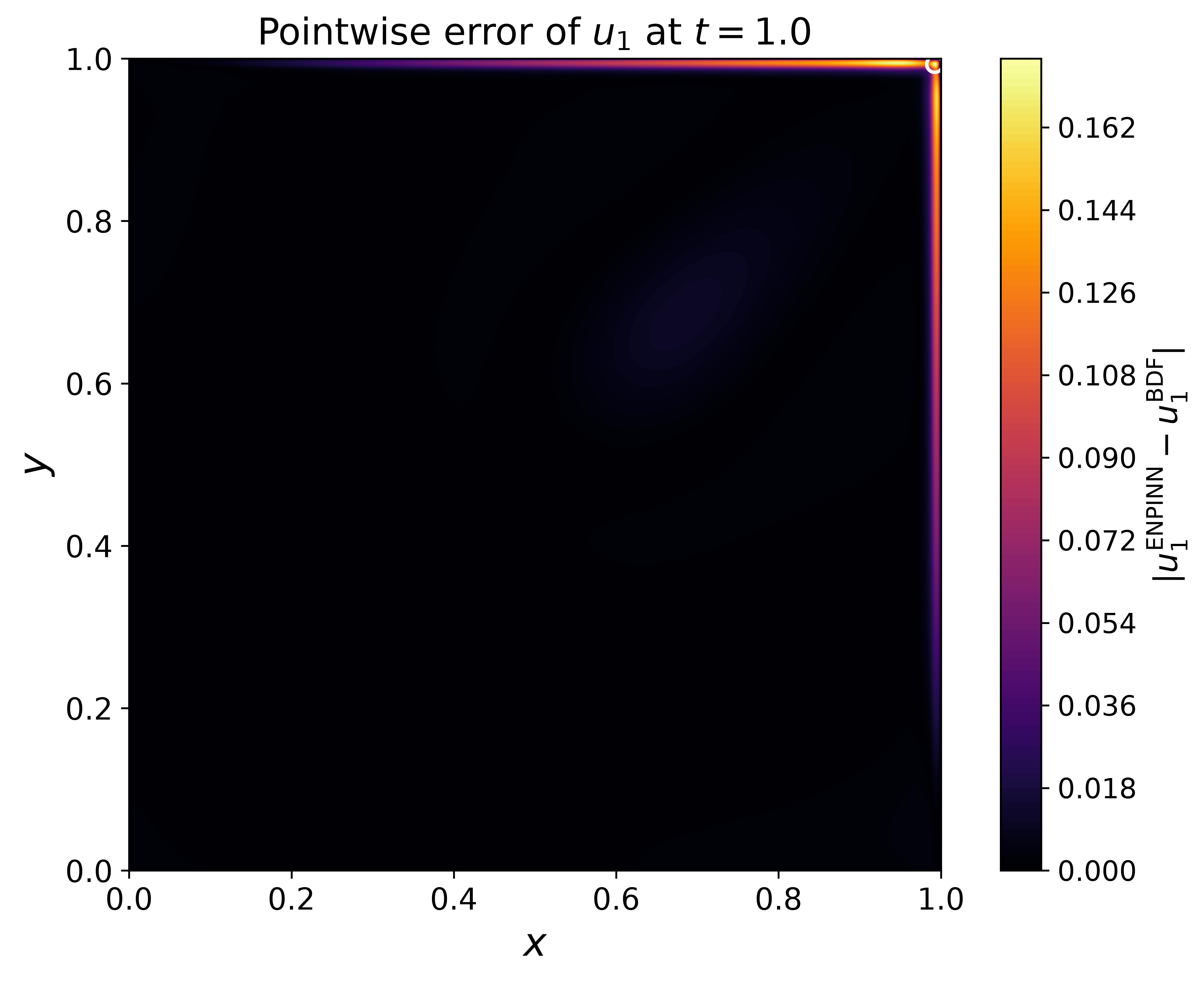}
                \caption{Pointwise error of the first component.}
                \label{fig: u1_err}
            \end{subfigure}
            \hfill
            \begin{subfigure}{0.48\linewidth}
                \centering
                \includegraphics[width=0.8\linewidth]{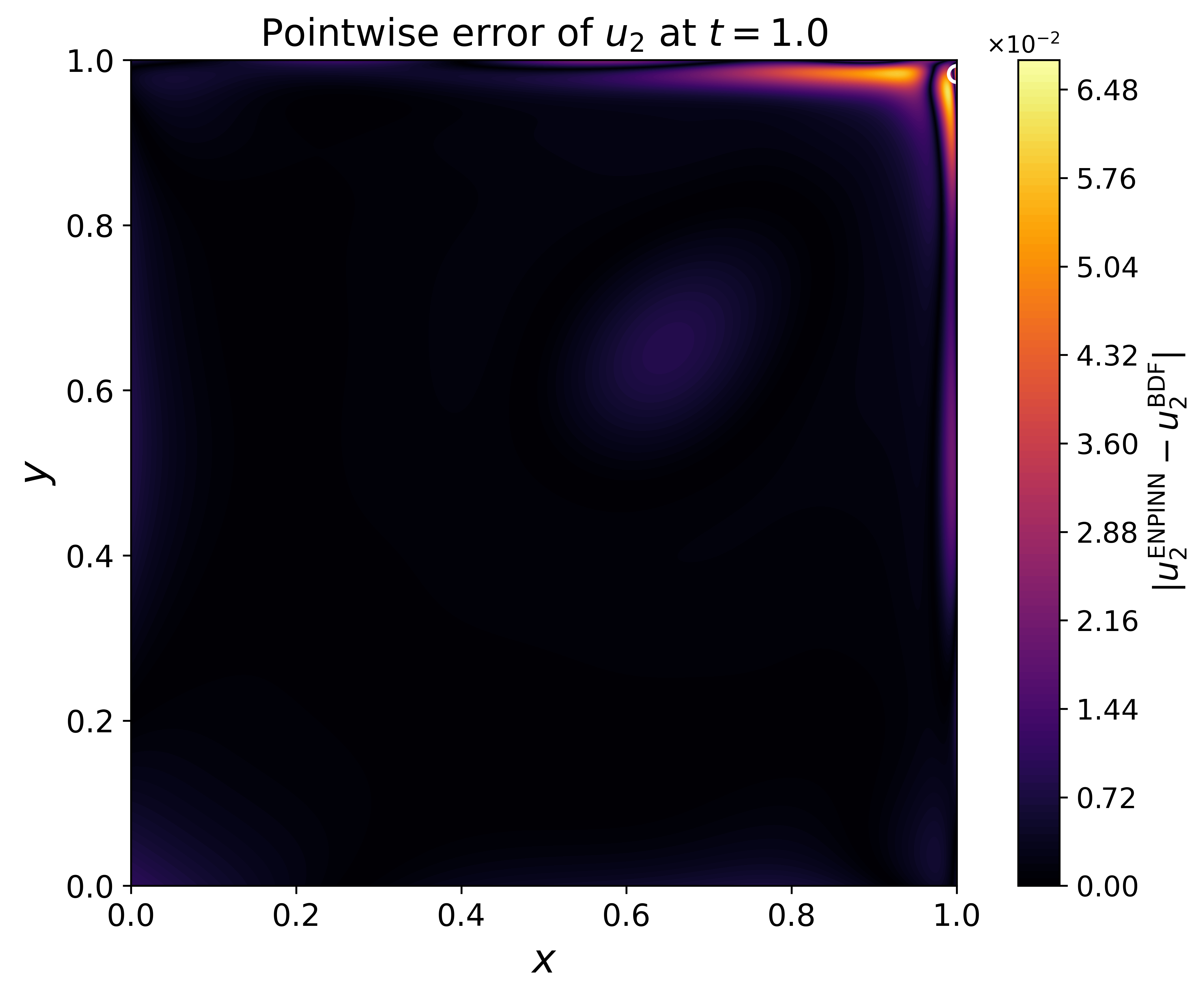}
                \caption{Pointwise error of the second component.}
                \label{fig: u2_err}
            \end{subfigure}
        
            \caption{Pointwise errors for the two components of the system \ref{ex: sys_ex} from the reference solutions.}
            \label{fig: ex_system}
        \end{figure}

    \end{example}

\section{Conclusions}\label{sec:Conclusion} 
{
Generalized convection–diffusion–reaction problems arise in a wide range of scientific and engineering applications, where sharp transitions near boundaries or within the domain present fundamental challenges for accurate numerical approximation. Resolving such layer phenomena and the associated steep gradients remains an important challenge for advancing the reliability and applicability of machine-learning-based PDE solvers. In this work, we introduce ENPINN, a physics-informed neural-network framework that incorporates the energy norm into a weak-form loss function to improve the representation of sharp gradients while maintaining stable convergence during training. By combining variational test functions with gradient-based residual information, ENPINN addresses key limitations of conventional PINN formulations in layer-dominated regimes. The main features and contributions of the proposed framework are summarized below:

\begin{itemize}
\item ENPINN overcomes key limitations of both gPINN and WLPINN, leading to improved predictive accuracy and detection of sharp boundary and interior layers.
\item We provide a detailed analysis of the limitations of WLPINN and gPINN. In WLPINN, a boundary-fitted test function is incorporated into the residual loss, whereas gPINN augments the loss with derivatives of the PDE residual. This analysis clarifies why neither formulation alone is sufficient to accurately resolve steep solution gradients.

\item We address three fundamental theoretical questions underlying PINN-based methods: existence, stability, and convergence. Drawing on perturbation theory, we establish rigorous bounds for ENPINN and demonstrate its stability, providing a necessary foundation for the convergence analysis and a mathematical basis for the proposed architecture.

\item Numerical experiments demonstrate faster convergence and higher accuracy of ENPINN compared with gPINN, WLPINN, and VSPINN, while maintaining the same representational capacity. Its effectiveness is demonstrated across diverse settings, including combustion models, Burgers equations with interior layers, three-dimensional problems, and coupled multiscale systems.

\item The energy-enriched formulation of ENPINN provides a simple yet effective mechanism for resolving sharp layers in generalized convection–diffusion–reaction problems. We further extend the formulation and its theoretical analysis to coupled systems by incorporating multiple solution components into the ENPINN loss.

\end{itemize}
Despite its promising performance, ENPINN currently has limitations that needs further investigation. The present formulation is designed primarily for generalized convection – diffusion – reaction problems with layer structures characterized by sharp gradients near boundaries or within the interior of the domain. More complex settings, including time-dependent problems with moving layers, may require substantial extensions of the framework and a corresponding theoretical analysis. Moreover, incorporating derivatives of the residual into the interior loss increases the computational cost, making ENPINN computationally comparable to gPINN. This motivates the development of more structured and mathematically informed test functions that can retain the accuracy of ENPINN while reducing its computational and empirical burden.

In summary, ENPINN provides an adaptable and theoretically grounded framework for approximating layer-dominated problems with sharp gradients. By incorporating energy-norm information into a weak-form loss, the framework improves the resolution of steep gradients while providing a principled basis for stability and generalization. An important direction for future work is the development of learnable test operators that can adapt to the underlying solution structure while reducing the computational cost associated with explicit residual derivatives. Such an extension would naturally connect ENPINN with operator-learning approaches and requires a careful investigation of both approximation properties and theoretical guarantees.

}

\section*{Data Availability}
{}{Data sharing is not applicable for this work as no data is generated.}

\section*{Declarations}
{}{The authors declare that they do not have any conflict of interest. We also declare that this work is not submitted elsewhere.}

\section*{Acknowledgment}
{}{The present research is supported by the Department of Science \& Technology, Govt. of India, for the author S. Maity; by the University Grants Commission, Govt. of India, for the author A. Patawari; and by the Ministry of Education, Govt. of India, for the author P. Das.}

\section*{Funding Declarations}
{}{There is no financial supports received by any of the authors.}

\bibliographystyle{unsrt}
\bibliography{ref}

@article{patawari2025traditional,
  title={From traditional to computationally efficient scientific computing algorithms in option pricing: {C}urrent progresses with future directions},
  author={Patawari, A. and Das, P.},
  journal={Arch. Comput. Methods Eng.},
  volume={33},
  pages={1699-1738},
  year={2025},
  publisher={Springer}
}

@article{mao2020physics,
  title={Physics-informed neural networks for high-speed flows},
  author={Mao, Z. and Jagtap, A. D. and Karniadakis, G. E.},
  journal={Comput. Methods Appl. Mech. Eng.},
  volume={360},
  pages={112789},
  year={2020},
  publisher={Elsevier}
}

@article{shukla2021parallel,
  title={Parallel physics-informed neural networks via domain decomposition},
  author={Shukla, K. and Jagtap, A. D. and Karniadakis, G. E.},
  journal={J. Comput. Phys.},
  volume={447},
  pages={110683},
  year={2021},
  publisher={Elsevier}
}

@article{hu2023augmented,
  title={{A}ugmented {P}hysics-{I}nformed {N}eural {N}etworks ({APINN}s): {A} gating network-based soft domain decomposition methodology},
  author={Hu, Z. and Jagtap, A. D. and Karniadakis, G. E. and Kawaguchi, K.},
  journal={Eng. Appl. Artif. Intell.},
  volume={126},
  pages={107183},
  year={2023},
  publisher={Elsevier}
}

@article{abbasi2025challenges,
  title={Challenges and advancements in modeling shock fronts with physics-informed neural networks: {A} review and benchmarking study},
  author={Abbasi, J. and Jagtap, A. D. and Moseley, B. and Hiorth, A. and Andersen, P. {\O}.},
  journal={Neurocomputing},
  pages={131440},
  year={2025},
  publisher={Elsevier}
}

@article{xu2019frequency,
  title={Frequency principle: {F}ourier analysis sheds light on deep neural networks},
  author={John Xu, Z. Q. and Zhang, Y. and Luo, T. and Xiao, Y. and Ma, Z.},
  journal={Commun. Comput. Phys.},
  volume= {28},
  pages={1746-1767},
  year={2020},
}

@article{menon2026fekan,
  title={{FEKAN}: {F}eature-{E}nriched {K}olmogorov--{A}rnold {N}etworks},
  author={Menon, S. S. and Jagtap, A. D.},
  journal={Comput. Methods Appl. Mech. Eng.},
  volume={461},
  pages={119141},
  year={2026},
  publisher={Elsevier}
}

@article{shin2020convergence,
  title={On the convergence of physics informed neural networks for linear second-order elliptic and parabolic type PDEs},
  author={Shin, Yeonjong and Darbon, Jerome and Karniadakis, George Em},
  journal={arXiv preprint arXiv:2004.01806},
  year={2020}
}

@article{mhaskar2026approximation,
  title={An approximation theory perspective on machine learning},
  author={Mhaskar, Hrushikesh N and Tsoukanis, Efstratios and Jagtap, Ameya D},
  journal={Neural Networks},
  volume={200},
  pages={108841},
  year={2026},
  publisher={Elsevier}
}

@article{peyvan2024riemannonets,
  title={Riemannonets: Interpretable neural operators for riemann problems},
  author={Peyvan, Ahmad and Oommen, Vivek and Jagtap, Ameya D and Karniadakis, George Em},
  journal={Computer Methods in Applied Mechanics and Engineering},
  volume={426},
  pages={116996},
  year={2024},
  publisher={Elsevier}
}

@article{goswami2024learning,
  title={Learning stiff chemical kinetics using extended deep neural operators},
  author={Goswami, Somdatta and Jagtap, Ameya D and Babaee, Hessam and Susi, Bryan T and Karniadakis, George Em},
  journal={Computer Methods in Applied Mechanics and Engineering},
  volume={419},
  pages={116674},
  year={2024},
  publisher={Elsevier}
}

@article{arihant2026nn,
  title={A semi-analytical fractional order neural network framework for two-dimensional time fractional reaction-diffusion problems mathematical and Computational analysis},
  author={Patawari, A. and Das, P. and Rana, S.},
  journal={Neural Netw.},
  Volume={Accepted},
  year={2026},
  publisher={Elsevier}
}

@article{jagtap2020adaptive,
  title={Adaptive activation functions accelerate convergence in deep and physics-informed neural networks},
  author={Jagtap, A. D. and Kawaguchi, K. and Karniadakis, G. E.},
  journal={J. Comput. Phys.},
  volume={404},
  pages={109136},
  year={2020},
  publisher={Elsevier}
}

@article{jagtap2023important,
  title={How important are activation functions in regression and classification? {A} survey, performance comparison, and future directions},
  author={Jagtap, A. D. and Karniadakis, G. E.},
  journal={J. Mach. Learn. Model. Comput.},
  volume={4},
  number={1},
  year={2023},
  publisher={Begel House Inc.}
}

@article{jagtap2020locally,
  title={Locally adaptive activation functions with slope recovery for deep and physics-informed neural networks},
  author={Jagtap, A. D. and Kawaguchi, K. and Karniadakis, G. E.},
  journal={Proc. R. Soc. A},
  volume={476},
  number={2239},
  pages={20200334},
  year={2020}
}

@article{zhang2026bubbleokan,
  title={Bubble{OKAN}: {A} physics-informed interpretable neural operator for high-frequency bubble dynamics},
  author={Zhang, Y. and Menon, S. S. and Cheng, L. and Gnanaskandan, A. and Jagtap, A. D.},
  journal={Comput. Methods Appl. Mech. Eng.},
  volume={450},
  pages={118667},
  year={2026},
  publisher={Elsevier}
}

@article{menon2026scientific,
  title={On scientific foundation models: {R}igorous definitions, key applications, and a comprehensive survey},
  author={Menon, S. S. and Mondal, T. and Brahmachary, S. and Panda, A. and Joshi, S. M. and Kalyanaraman, K. and Jagtap, A. D.},
  journal={Neural Netw.},
  pages={108567},
  year={2026},
  publisher={Elsevier}
}

@article{menon2026intelligent,
  title={Intelligent fluid flows: {A} survey of deep learning methods for turbulent flows, multiphase flows, and combustion},
  author={Menon, S. S. and Lavari, M. and Kokernak, A. and Mathew, J. and Jagtap, C. A. and Jayachandran, J. and Gnanaskandan, A. and Jagtap, A. D.},
  journal={Neurocomputing},
  volume = {697},
  pages={134117},
  year={2026},
  publisher={Elsevier}
}

@article{zhao2025pikan,
  title={A {PIKAN}-based model for the prediction of the temperature fields of castings},
  author={Zhao, Q. and Wang, B. and Kang, J.},
  journal={Sci. Rep.},
  year={2025},
  publisher={Nature Publishing Group UK London}
}

@article{wang2025kolmogorov,
  title={Kolmogorov--{A}rnold-{I}nformed neural network: {A} physics-informed deep learning framework for solving forward and inverse problems based on {K}olmogorov--{A}rnold {N}etworks},
  author={Wang, Y. and Sun, J. and Bai, J. and Anitescu, C. and Eshaghi, M. S. and Zhuang, X. and Rabczuk, T. and Liu, Y.},
  journal={Comput. Methods Appl. Mech. Eng.},
  volume={433},
  pages={117518},
  year={2025},
  publisher={Elsevier}
}

@article{shukla2021physics,
  title={A physics-informed neural network for quantifying the microstructural properties of polycrystalline nickel using ultrasound data: {A} promising approach for solving inverse problems},
  author={Shukla, K. and Jagtap, A. D. and Blackshire, J. L. and Sparkman, D. and Karniadakis, G. E.},
  journal={IEEE Signal Process. Mag.},
  volume={39},
  number={1},
  pages={68--77},
  year={2021},
  publisher={IEEE}
}

@article{jagtap2022physics,
  title={Physics-informed neural networks for inverse problems in supersonic flows},
  author={Jagtap, A. D. and Mao, Z. and Adams, N. and Karniadakis, G. E.},
  journal={J. Comput. Phys.},
  volume={466},
  pages={111402},
  year={2022},
  publisher={Elsevier}
}

@article{menon2025anant,
  title={Anant-{N}et: {B}reaking the curse of dimensionality with scalable and interpretable neural surrogate for high-dimensional {PDE}s},
  author={Menon, S. S. and Jagtap, A. D.},
  journal={Comput. Methods Appl. Mech. Eng.},
  volume={447},
  pages={118403},
  year={2025},
  publisher={Elsevier}
}

@article{abbasi2025history,
  title={History-matching of imbibition flow in fractured porous media using physics-informed neural networks ({PINN}s)},
  author={Abbasi, J. and Moseley, B. and Kurotori, T. and Jagtap, A. D. and Kovscek, A. R. and Hiorth, A. and Andersen, P. {\O}.},
  journal={Comput. Methods Appl. Mech. Eng.},
  volume={437},
  pages={117784},
  year={2025},
  publisher={Elsevier}
}

@article{jagtap2022deepD,
  title={Deep learning of inverse water waves problems using multi-fidelity data: {A}pplication to {S}erre--{G}reen--{N}aghdi equations},
  author={Jagtap, A. D. and Mitsotakis, D. and Karniadakis, G. E.},
  journal={Ocean Eng.},
  volume={248},
  pages={110775},
  year={2022},
  publisher={Elsevier}
}

@article{penwarden2023unified,
  title={A unified scalable framework for causal sweeping strategies for {P}hysics-{I}nformed {N}eural {N}etworks ({PINN}s) and their temporal decompositions},
  author={Penwarden, M. and Jagtap, A. D. and Zhe, S. and Karniadakis, G. E. and Kirby, R. M.},
  journal={J. Comput. Phys.},
  volume={493},
  pages={112464},
  year={2023},
  publisher={Elsevier}
}

@article{hu2021extended,
  title={When do extended physics-informed neural networks ({XPINN}s) improve generalization?},
  author = {Hu, Z. and Jagtap, A. D. and Karniadakis, G. E. and Kawaguchi, K.},
  journal = {SIAM J. Sci. Comput.},
  volume = {44},
  number = {5},
  pages = {A3158-A3182},
  year = {2022}
}

@article{jagtap2022deep,
  title={Deep {K}ronecker neural networks: {A} general framework for neural networks with adaptive activation functions},
  author={Jagtap, A. D. and Shin, Y. and Kawaguchi, K. and Karniadakis, G. E.},
  journal={Neurocomputing},
  volume={468},
  pages={165--180},
  year={2022},
  publisher={Elsevier}
}

@article{maity2026ti,
  title={T{I}-{P}er{PINN}: a theoretically guided neural network driven approaches for singularly perturbed convection dominated differential equations on circular domains},
  author={Maity, S. and Rathore, K.},
  journal={J. Anal.},
  volume={34},
  pages={1-19},
  year={2026},
  publisher={Springer}
}

@article{yu2022gradient,
  title={Gradient-enhanced physics-informed neural networks for forward and inverse {PDE} problems},
  author={Yu, J. and Lu, L. and Meng, X. and Karniadakis, G. E.},
  journal={Comput. Methods Appl. Mech. Eng.},
  volume={393},
  pages={114823},
  year={2022},
  publisher={Elsevier}
}

@article{kumar2025uniformly,
  title={A uniformly convergent analysis for multiple scale parabolic singularly perturbed convection-diffusion coupled systems: Optimal accuracy with less computational time},
  author={Kumar, S. and Das, P.},
  journal={Appl. Numer. Math.},
  volume={207},
  pages={534--557},
  year={2025},
  publisher={Elsevier}
}

@article{linss2009numerical,
  title={Numerical Solution of Systems of Singularly Perturbed Differential Equations.},
  author={Lin{\ss}, T. and Stynes, M.},
  journal={Comput. Methods Appl. Math.},
  volume={9},
  pages={165--191},
  year={2009}
}

@article{patawari2025modified,
  title={A modified iterative {PINN} algorithm for strongly coupled system of boundary layer originated convection diffusion reaction problems in {MHD} flows with analysis},
  author={Patawari, A. and Kumar, S. and Das, P.},
  journal={Netw. Heterog. Media},
  volume={20},
  pages={1026-1060},
  year={2025}
}

@article{uat_cybenko_1989,
  title={Approximation by superpositions of a sigmoidal function},
  author={Cybenko, G.},
  journal={Math. Control Signals Syst.},
  volume={2},
  pages={303--314},
  year={1989},
  publisher={Springer}
}

@article{pinn_2018_indroce,
  title={Physics-informed neural networks: A deep learning framework for solving forward and inverse problems involving nonlinear partial differential equations},
  author={Raissi, M. and Perdikaris, P. and Karniadakis, G. E.},
  journal={J. Comput. Phys.},
  volume={378},
  pages={686--707},
  year={2019},
  publisher={Elsevier}
}

@article{cpinn_intro,
  title={Conservative physics-informed neural networks on discrete domains for conservation laws: Applications to forward and inverse problems},
  author={Jagtap, A. D. and Kharazmi, E. and Karniadakis, G. E.},
  journal={Comput. Methods Appl. Mech. Eng.},
  volume={365},
  pages={113028},
  year={2020},
  publisher={Elsevier}
}

@article{vpinn_intro,
  title={Variational physics-informed neural networks for solving partial differential equations},
  author={Kharazmi, E. and Zhang, Z. and Karniadakis, G. E.},
  journal={arXiv preprint},
  year={2019}
}

@article{hp-vpinn_intro,
  title={{hp-VPINNs}: Variational physics-informed neural networks with domain decomposition},
  author={Kharazmi, E. and Zhang, Z. and Karniadakis, G. E.},
  journal={Comput. Methods Appl. Mech. Eng.},
  volume={374},
  pages={113547},
  year={2021},
  publisher={Elsevier}
}

@article{Xpinn,
  title={Extended physics-informed neural networks ({XPINNs}): A generalized space-time domain decomposition based deep learning framework for nonlinear partial differential equations},
  author={Jagtap, A. D. and Karniadakis, G. E.},
  journal={Commun. Comput. Phys.},
  volume={28},
  year={2020},
  publisher={Brown Univ., Providence, RI (United States)}
}

@article{o2011parameter,
  title={A parameter-uniform numerical method for a singularly perturbed two parameter elliptic problem},
  author={O’Riordan, E. and Pickett, M.L.},
  journal={Adv. Comput. Math.},
  volume={35},
  pages={57--82},
  year={2011},
  publisher={Springer}
}

@article{dasmehrmann_1D_spp,
  title={Numerical solution of singularly perturbed convection-diffusion-reaction problems with two small parameters},
  author={Das, P. and Mehrmann, V.},
  journal={BIT Numer. Math.},
  volume={56},
  pages={51--76},
  year={2016},
  publisher={Springer}
}

@article{semianalytic_rectangulardomain,
  title={Semi-analytic {PINN} methods for boundary layer problems in a rectangular domain},
  author={Gie, G. and Hong, Y. and Jung, C. and Munkhjin, T.},
  journal={J. Comput. Appl. Math.},
  volume={450},
  pages={115989},
  year={2024},
  publisher={Elsevier}
}

@article{vspinn,
  title={{VS-PINN}: A fast and efficient training of physics-informed neural networks using variable-scaling methods for solving {PDEs} with stiff behavior},
  author={Ko, S. and Park, S.},
  journal={J. Comput. Phys.},
  volume={529},
  pages={113860},
  year={2025},
  publisher={Elsevier}
}

@article{papinn_onepar,
  title={Physics-informed neural networks with parameter asymptotic strategy for learning singularly perturbed convection-dominated problem},
  author={Cao, F. and Gao, F. and Guo, X. and Yuan, D.},
  journal={Comput. Math. Appl.},
  volume={150},
  pages={229--242},
  year={2023},
  publisher={Elsevier}
}

@article {JKS11,
    AUTHOR = {John, V. and Knobloch, P. and Savescu, S. B.},
     TITLE = {A posteriori optimization of parameters in stabilized methods
              for convection-diffusion problems---{P}art {I}},
   JOURNAL = {Comput. Methods Appl. Mech. Engrg.},
    VOLUME = {200},
      YEAR = {2011},
     PAGES = {2916--2929},
}

@article {FHJ24,
    AUTHOR = {Frerichs-Mihov, D. and Henning, L. and John, V.},
    TITLE = {On Loss Functionals for Physics-Informed Neural Networks for Steady-State Convection-Dominated Convection-Diffusion Problems},
    JOURNAL = {Commun. Appl. Math. Comput. },
    Volume = {8},
    pages = {287–308},
    YEAR = {2024},
}

@article{mishra2023estimates,
  title={Estimates on the generalization error of physics-informed neural networks for approximating {PDEs}},
  author={Mishra, S. and Molinaro, R.},
  journal={IMA J. Numer. Anal.},
  volume={43},
  pages={1--43},
  year={2023},
  publisher={Oxford University Press}
}

@article{de2024error,
  title={Error estimates for physics-informed neural networks approximating the {N}avier-{S}tokes equations},
  author={De Ryck, T. and Jagtap, A. D. and Mishra, S.},
  journal={IMA J. Numer. Anal.},
  volume={44},
  pages={83--119},
  year={2024},
  publisher={Oxford University Press}
}

@article{haber1966modified,
  title={A modified Monte-Carlo quadrature},
  author={Haber, S.},
  journal={Math. Comput.},
  volume={20},
  pages={361--368},
  year={1966},
  publisher={JSTOR}
}

@article{opschoor2025neural,
  title={Neural networks for singular perturbations},
  author={Opschoor, J. A. and Schwab, C. and Xenophontos, C.},
  journal={Numer. Math.},
  volume={157},
  pages={1897--1936},
  year={2025},
  publisher={Springer}
}

@article{wang2025adaptive,
  title={Adaptive deep physics-informed neural network with dual-nested activation for solving complex partial differential equations},
  author={Wang, T. and Liu, G. and Li, E. and Xu, X.},
  journal={Comput. Methods Appl. Mech. Eng.},
  volume={444},
  pages={118125},
  year={2025},
  publisher={Elsevier}
}

@book{pao2012nonlinear,
  title={Nonlinear Parabolic and Elliptic Equations},
  author={Pao, C.},
  year={2012},
  publisher={Springer Science \& Business Media}
}

\appendix

\section{Bounds of ENPINN}\label{sec: appen_Bounds of ENPINN}
\noindent\textbf{Theorem \ref{th: existence_enpinn}}: Let $\overline{\Omega} = [0,1]$. For a fixed depth $h$ and every $\mathfrak{B}>0$, there exists $\hat{p}\in \mathcal{P}$ such that $\mathcal{L}_{en}^C(\hat{p})< \mathfrak{B}$, where the corresponding neural network $u_{N,\hat{p}}$ approximates the solution of \eqref{eq: ode_in_existence}.
\begin{proof}
        We have $\mathcal{L}_{en}^C(\hat{p}) = \|(\mathcal{L}u_{N,\hat{p}} -f)\mathbb{V}\|_{H^{1}({\Omega})}$. Therefore
        \begin{equation}\label{eq: existence_loss_bound}
            \begin{aligned}
                \|(\mathcal{L}u_{N,\hat{p}} -f)\mathbb{V}\|^2_{H^{1}({\Omega})} &\leq \|(\mathcal{L}u_{N,\hat{p}} -f)\mathbb{V}\|^2_{L^{2}({\Omega})} + \|\nabla((\mathcal{L}u_{N,\hat{p}} -f)\mathbb{V})\|^2_{L^{2}({\Omega})}\\
                & \leq C \|u-u_{N,\hat{p}}\|_{H^2({\Omega})}\\
                & \leq C \|u-u_{N,\hat{p}}\|_{L^{\infty}(\overline{\Omega})} + \|\nabla(u-u_{N,\hat{p}})\|_{L^{\infty}(\overline{\Omega})} + \|\Delta(u-u_{N,\hat{p}})\|_{L^{\infty}(\overline{\Omega})}.
            \end{aligned}
        \end{equation}
        For describing the layers of \eqref{eq: ode_in_existence}, consider the characteristic equation of \ref{eq: ode_in_existence}
        \begin{equation*}
            -\varepsilon D^2 + a\mu D+b=0,
        \end{equation*}
        which has solutions $D_{L,R}= \left(\dfrac{\mu a}{2 \varepsilon}\mp \sqrt{\dfrac{\mu^2 a^2}{4\varepsilon^2}+\dfrac{b}{\varepsilon}}\right)$.
        We can rearrange the solution terms as
        \[D_{L,R}=
        \begin{cases}
            -\dfrac{\mu a}{2 \varepsilon}\left(-1\pm \sqrt{1+\dfrac{4b\varepsilon}{\mu^2a^2}}\right), \quad &\text{if } \mu^2 \geq \dfrac{4b}{a^2} \varepsilon \\
             -\sqrt{\dfrac{b}{ \varepsilon}}\left( {-\sqrt{{\dfrac{\mu^2 a^2}{4b\varepsilon}}}\pm\sqrt{1+\dfrac{\mu^2 a^2}{4b\varepsilon}}}\right), \quad &\text{if } \mu^2 \leq  \dfrac{4b}{a^2} \varepsilon.
        \end{cases} \]
        The solution $D_L$ describes the layer near $x=0$ and $D_R$ describes the layer near $x=1$. Let us consider the case $\mu^2 \leq ({4b}\varepsilon/{a^2}) $. {The decomposition near the left boundary can be described as
        $$s_L = C^L \exp{\left(-\sqrt{\dfrac{b}{ \varepsilon}}\left( {-\sqrt{{\dfrac{\mu^2 a^2}{4b\varepsilon}}}\pm\sqrt{1+\dfrac{\mu^2 a^2}{4b\varepsilon}}}\right)x\right)}.$$}
        Take
        \begin{equation}\label{eq: left_layer_approx_NN}
            u_{N,\hat{p}}^{L}= C^L \exp(1) u_{\mathfrak{J},N}^{\exp} \circ u_{1,h-1,\mathfrak{K},T,N}^{id} \circ \mathfrak{Q}, 
        \end{equation}
        {where 
        $$\mathfrak{Q}: x \mapsto \left(\sqrt{\dfrac{b}{ \varepsilon}}\left( {-\sqrt{{\dfrac{\mu^2 a^2}{4b\varepsilon}}}\pm\sqrt{1+\dfrac{\mu^2 a^2}{4b\varepsilon}}}\right)\right)x +1.$$} {For the remainder of the proof, we introduce the following domains:  
        \begin{equation*}
            \begin{aligned}
                \Omega_1 &= \left[0, 2+ \left(\sqrt{\dfrac{b}{ \varepsilon}}\left( {-\sqrt{{\dfrac{\mu^2 a^2}{4b\varepsilon}}}\pm\sqrt{1+\dfrac{\mu^2 a^2}{4b\varepsilon}}}\right)\right)\right], \\
                \Omega_2 &= \left[1, 1+ \left(\sqrt{\dfrac{b}{ \varepsilon}}\left( {-\sqrt{{\dfrac{\mu^2 a^2}{4b\varepsilon}}}\pm\sqrt{1+\dfrac{\mu^2 a^2}{4b\varepsilon}}}\right)\right)\right].
            \end{aligned}
        \end{equation*}} 
        Therefore,
        \begin{equation}\label{eq: existence_exact_err}
            \begin{aligned}
                \|s_L - u_{N,\hat{p}}^{L}\|_{L^{\infty}(\overline{\Omega})} &{\leq} C^L \exp{(1)} \|\exp{(-)}\circ {id}\, \circ \mathfrak{Q} - u_{\mathfrak{J},N}^{\exp} \circ u_{1,h-1,\mathfrak{K},T,N}^{id} \circ \mathfrak{Q}\|_{L^{\infty}(\overline{\Omega})}\\
                & \leq C^L \exp{(1)} \|(\exp{(-)}\circ {id}\, - \exp{(-)} \circ u_{1,h-1,\mathfrak{K},T,N}^{id}) \circ \mathfrak{Q}\|_{L^{\infty}(\overline{\Omega})}\\
                & \quad + C^L \exp{(1)} \|(\exp{(-)}-u_{\mathfrak{J},N}^{\exp})\circ u_{1,h-1,\mathfrak{K},T,N}^{id} \circ \mathfrak{Q} \|_{L^{\infty}(\overline{\Omega})}\\
                & \leq C^L \exp{(1)} \|\exp{(-)}\|_{L^{\infty}{(\Omega_1})} \|id - u_{1,h-1,\mathfrak{K},T,N}^{id}\|_{L^{\infty}(\Omega_2)} \\
                &\quad+  C^L \exp{(1)} \|(\exp{(-)}-u_{\mathfrak{J},N}^{\exp}) \|_{L^{\infty}(\Omega_1)}\\
                & \leq C (\mathfrak{K}+ \mathfrak{J}).
            \end{aligned}
        \end{equation}
        Next, we show that the approximation error between the derivatives of the layer component and those of the neural network can be made arbitrarily small. 
        \begin{equation}\label{eq: existence_first_deriv_err}
            \begin{aligned}
                &\|(s_L)_x - (u_{N,\hat{p}}^{L})_x\|_{L^{\infty}(\overline{\Omega})} \\
                &\leq C^L \exp{(1)} \|(\exp{(-)}\circ {id} - u_{\mathfrak{J},N}^{\exp} \circ u_{1,h-1,\mathfrak{K},T,N}^{id})_x \|_{L^{\infty}(\Omega_2)}\\
                & \leq C^L \exp{(1)} \|(-\exp{(-)}\circ {id}\, - \exp{(-)} \circ u_{1,h-1,\mathfrak{K},T,N}^{id}) . id_x\|_{L^{\infty}(\Omega_2)}\\
                & \quad + C^L \exp{(1)} \|(-\exp{(-)} \circ u_{1,h-1,\mathfrak{K},T,N}^{id} . id_x- (-\exp{(-)} \circ u_{1,h-1,\mathfrak{K},T,N}^{id})) . (u_{1,h-1,\mathfrak{K},T,N}^{id})_x\|_{L^{\infty}(\Omega_2)}\\
                & \quad + C^L \exp{(1)} \|(-\exp{(-)}-(u_{\mathfrak{J},N}^{\exp})_x) \circ u_{1,h-1,\mathfrak{K},T,N}^{id}. (u_{1,h-1,\mathfrak{K},T,N}^{id})_x\|_{L^{\infty}(\Omega_2)}\\
                & \leq C(2\mathfrak{K}+ 2 \mathfrak{J}).
            \end{aligned}
        \end{equation}
        Next, we estimate the error between $(s_L)_{xx}$ and $(u_{N,p}^{L,B})_{xx}$.
        
        \begin{equation}\label{eq: existence_second_deriv_err}
            \begin{aligned}
                &\|(s_L)_{xx} - (u_{N,\hat{p}}^{L})_{xx}\|_{L^{\infty}(\overline{\Omega})} \\
                &\leq C^L \exp{(1)} \|(\exp{(-)}\circ {id} - u_{\mathfrak{J},N}^{\exp} \circ u_{1,h-1,\mathfrak{K},T,N}^{id})_{xx} \|_{L^{\infty}(\Omega_2)}\\
                &\leq C^L\|\exp(-) \circ id.(id_{x})^2 - \exp(-) \circ u_{1,h-1,\mathfrak{K},T,N}^{id}.(id_{x})^2\|_{L^{\infty}(\Omega_2)}\\
                &\quad + C^L \|\exp(-) \circ u_{1,h-1,\mathfrak{K},T,N}^{id} .(id_{x})^2 - \exp(-) \circ u_{1,h-1,\mathfrak{K},T,N}^{id}. ((u_{1,h-1,\mathfrak{K},T,N}^{id})_{x})^2\|_{L^{\infty}(\Omega_2)}\\
                &\quad + C^L\|\exp(-) \circ u_{1,h-1,\mathfrak{K},T,N}^{id}. ((u_{1,h-1,\mathfrak{K},T,N}^{id})_{x})^2- (u_{\mathfrak{J},N}^{\exp})_{xx}\circ u_{1,h-1,\mathfrak{K},T,N}^{id}. ((u_{1,h-1,\mathfrak{K},T,N}^{id})_{x})^2\|_{L^{\infty}(\Omega_2)}\\
                &\quad +C^L\|\exp{(-)}\circ id \circ id_{xx} - \exp{(-)}\circ u_{1,h-1,\mathfrak{K},T,N}^{id} \circ id_{xx}\|_{L^{\infty}(\Omega_2)}\\
                &\quad +C^L \|-\exp{(-)}\circ u_{1,h-1,\mathfrak{K},T,N}^{id}.id_{xx}+ \exp{(-)}\circ u_{1,h-1,\mathfrak{K},T,N}^{id}. (u_{1,h-1,\mathfrak{K},T,N}^{id})_{xx}\|_{L^{\infty}(\Omega_2)}\\
                &\quad +C^L \|\exp{(-)}\circ u_{1,h-1,\mathfrak{K},T,N}^{id}. (u_{1,h-1,\mathfrak{K},T,N}^{id})_{xx}- (u_{\mathfrak{J},N}^{\exp})\circ u_{1,h-1,\mathfrak{K},T,N}^{id}. (u_{1,h-1,\mathfrak{K},T,N}^{id})_{xx}\|_{L^{\infty}(\Omega_2)}\\
                &\leq C(4\mathfrak{K}+ 2 \mathfrak{J}).
            \end{aligned}
        \end{equation}
        {Analogously, we can find out $u_{N,p}^L$ for the case $\mu^2 \geq ({4b}\varepsilon/{a^2})$.} 
        Similarly, neural networks corresponding to the remaining decomposition components of the exact solution can be constructed. Hence, the neural network that approximates the underlying solution is
        \begin{equation*}
            u_{N,\hat{p}} = u_{N,\hat{p}}^r + u_{N,\hat{p}}^L + u_{N,\hat{p}}^R.
        \end{equation*}
        Thus, we get the desired estimate
        \begin{equation*}
            \begin{aligned}
                \|u&-u_{N,\hat{p}}\|_{H^2({\Omega})}\\
                & \leq \|r-u_{N,\hat{p}}^r\|_{H^2({\Omega})} + \|s_L-u_{N,\hat{p}}^L\|_{H^2({\Omega})} + \|s_R-u_{N,\hat{p}}^R\|_{H^2({\Omega})}\\
                & \leq C \mathfrak{B},
            \end{aligned}
        \end{equation*}
        where {}{$\mathfrak{K}$, $\mathfrak{J}$, the required tolerances are chosen sufficiently small such that the resulting estimate is bounded above by $\mathfrak{B}$.}
    \end{proof}

     \noindent \textbf{Theorem \ref{th: deriv_bounds_enpinn}:}
     Let $u_{N,\hat{p}}^L$ be the approximation of the layer decomposition $s_L$ of $u$ {for $\mu^2 \leq ({4b}\varepsilon/{a^2})$}. Then
    \begin{equation*}
        |\partial ^{l_x} u_{N,\hat{p}}^L| \leq C \varepsilon ^{-\frac{l_x}{2}} \exp{\left(-\sqrt{\frac{b}{\varepsilon}}x\right)}. 
    \end{equation*}
    \begin{proof}
        From \eqref{eq: left_layer_approx_NN} of Theorem \ref{th: existence_enpinn}, we get the bound of $u_{N,\hat{p}}^L$ as
        \begin{equation*}
            \begin{aligned}
                |u_{N,\hat{p}}^L| &= |C^L \exp(1) u_{\mathfrak{J},N}^{\exp} \circ u_{1,h-1,\mathfrak{K},T,N}^{id} \circ \mathfrak{Q}|\\
                & \leq C |\mathfrak{J}+ \exp{(-)}\circ{u_{1,h-1,\mathfrak{K},T,N}^{id}\circ \mathfrak{Q}}|\\
                & \leq C |\mathfrak{J}+ \exp{(-)} \circ (\mathfrak{K} + id \circ \mathfrak{Q})|\\
                & \leq  C |\mathfrak{J}+ (\exp{(-)} \circ \mathfrak{K}). (\exp{(-)} \circ \mathfrak{Q})|.
            \end{aligned}
        \end{equation*}
        Thus, the first-order derivative bound with respect to $x$ is
        \begin{equation*}
            |(u_{N,\hat{p}}^L)_x| \leq C \varepsilon^{-\frac{1}{2}} \exp\left({-\sqrt{\frac{b}{\varepsilon}}}x\right),
        \end{equation*}
        where $C$ depends on $\mathfrak{J}, \mathfrak{K}$ and independent of inverse powers of $\varepsilon$. In a similar manner, we get the desired estimate.
    \end{proof}

\section{Error Estimation}\label{sec: appen_Error Estimation}
\noindent \textbf{Theorem \ref{th:gronwall_genearlization_err_l^2_norm}}
    Let {$u \in {H^2(\Omega) \cap C^0(\overline{\Omega})}$} denote the exact solution of \eqref{eq: main_cont_prob_lin} for {$\mu^2 \geq \zeta \varepsilon/\alpha_0$}. Let $u_{N,p^*}$ be the approximation of u obtained using the 
    loss function $\mathcal{L}_{N}^C (p)$  in \eqref{loss_function}. 
    Then the generalization error satisfies the following inequality
    \begin{equation*}\label{eq: appen_gronwall_genearlization_err_l^2_norm}
        \|\mathfrak{e}\|_{{L^2(\Omega_s)}}^2= \|u-u_{N,{{p}^*}}\|_{{L^2(\Omega_s)}}^2 \leq e^{\kappa_{L_2}^l}\times \mathcal{G}_{L_2}^l,
    \end{equation*}
    where $\kappa_{L_2}={(2\alpha + 1)T}$, $\mathcal{G}_{L_2} = \|\mathcal{N}_{ini,{{p}^*}}\|_{L^2(\Omega_s)}^2 + T^{1/2} \|\mathcal{N}_{spb,{{p}^*}}\|_{L^2(\Omega)}  \mathcal{K}_{\partial \Omega_s}^{\nabla \mathfrak{e}} + 2\alpha \|\mathcal{N}_{spb,{{p}^*}}\|^2_{L^2(\Omega)} + \|\mathcal{N}_{int,{{p}^*}}\|_{L^2(\Omega)}^2$. The positive constant $\mathcal{K}_{\partial\Omega_s}^{\nabla \mathfrak{e}}$ is independent of $\varepsilon$ and depends on the continuous solution $u$, the NN approximation $u_{N,{{p}^*}}$, $i.e.,$ 
    $$\mathcal{K}_{\partial\Omega_s}^{\nabla \mathfrak{e}}\equiv \mathcal{K}_{\partial\Omega_s}^{\nabla \mathfrak{e}}(u,u_{N,{{p}^*}}).$$
    \begin{proof}
        Using \eqref{eq: main_cont_prob_lin} and residuals defined in Subsection \ref{subsec:Modified Residuals}, we obtain the following problem for $\mathfrak{e}$:
        \begin{equation}\label{difference_con_prob}
        \begin{cases}
            \mathfrak{e}_t=\varepsilon \Delta \mathfrak{e}-\mu {a}(x,y). \nabla \mathfrak{e}-b(x,y)\mathfrak{e} + \mathcal{N}_{int,{{p}^*}}(x,y,t), \quad (x,y,t) \in \Omega\equiv \Omega_s \times \Omega_t,\\
            \mathfrak{e}(x,y,t)=\mathcal{N}_{spb,{{p}^*}}(x,y,t), \quad (x,y)\in \partial \Omega_s=\overline{\Omega}_s\backslash \Omega_s, ~ t\in [0,T], \\
            \mathfrak{e}(x,y,0)=\mathcal{N}_{ini,{{{p}^*}}}(x,y), \quad (x,y)\in \Omega_s.
        \end{cases}
        \end{equation}
        Here, we denote $\hat{n}$ as the outward unit normal vector to $\partial \Omega_s$.\\
        Multiplying the first equation in \eqref{difference_con_prob} by $\mathfrak{e}$ and integrate over $\Omega_s$, we get  
    \begin{align*}
        \frac{1}{2}\frac{d}{dt}\displaystyle\int_{\Omega_s} |\mathfrak{e}(x,y,t)|^2 dxdy=&\varepsilon\int_{\Omega_s}\mathfrak{e}\Delta \mathfrak{e} \, dxdy-\mu\int_{\Omega_s}\mathfrak{e}\left(a(x,y). \nabla \mathfrak{e}\right) \, dxdy\\
        &-\int_{\Omega_s} b(x,y) \mathfrak{e}^2 \, dxdy + \int_{\Omega_s} \mathfrak{e}\mathcal{N}_{int,{{p}^*}}(x,y,t) \, dxdy.
    \end{align*}
    Using integration by parts, Young's inequality, and the bounds from \eqref{bound_a_b}, we obtain
    \begin{align}
            \begin{split}
                \frac{1}{2}&\frac{d}{dt}\displaystyle\int_{\Omega_s} |\mathfrak{e}(x,y,t)|^2 dxdy \\
                & \leq \, \varepsilon\int_{\partial\Omega_s}\mathfrak{e}\left(\nabla \mathfrak{e}. \hat{n}\right) \, ds - \varepsilon \int_{\Omega_s} |\nabla \mathfrak{e}|^2 \, dxdy -\dfrac{\mu}{2}\int_{\Omega_s}{a.\nabla (\mathfrak{e}^2)\, dxdy} + \frac{1}{2}\int_{\Omega_s} |\mathfrak{e}|^2 \, dxdy
            \end{split} \label{eq: change_mu_less_to_mu_gr_l2_err1}\\
            \begin{split}
                & \quad + \frac{1}{2}\int_{\Omega_s} |\mathcal{N}_{int,{}{{p}^*}}(x,y,t)|^2 \, dxdy.\\
                & \leq \, \varepsilon\int_{\partial\Omega_s}\mathfrak{e}\left(\nabla \mathfrak{e}. \hat{n}\right) \, ds +{\mu\alpha}\left(\int_{\Omega_s}{\mathfrak{e}^2\, dxdy} + \int_{\partial\Omega_s} \mathfrak{e}^2 \, ds \right) +\frac{1}{2}\int_{\Omega_s} |\mathfrak{e}|^2 \, dxdy \\
                &\quad + \frac{1}{2}\int_{\Omega_s} |\mathcal{N}_{int,{}{{p}^*}}(x,y,t)|^2 \, dxdy.
            \end{split} \label{eq: change_mu_less_to_mu_gr_l2_err2}
    \end{align}
    Applying Cauchy--Schwarz and Young's inequality, we get
    \begin{align*}
        & \leq  \left(\int_{\partial\Omega_s}|\mathcal{N}_{spb,{}{{p}^*}}|^2 \, ds\right)^{\frac{1}{2}}\left(\varepsilon \int_{\partial\Omega_s}(\mathfrak{e}_x^{^2}+\mathfrak{e}_y^{^2}) \, ds\right)^{\frac{1}{2}} +\left(\frac{2\alpha + 1}{2}\right)\int_{\Omega_s}|\mathfrak{e}|^2 \, dxdy \\
        &\quad ~~ +{\mu \alpha}\int_{\partial\Omega_s} |\mathcal{N}_{spb,{}{{p}^*}}|^2 \, dxdy + \frac{1}{2}\int_{\Omega_s}{|\mathcal{N}_{int,{}{{p}^*}}(x,y,t)|^2} \, dxdy.
    \end{align*}
    {Using the derivative bounds from \eqref{eq: exact_sol_derivative_bound_mu_less_eps}-\eqref{eq: exact_sol_derivative_bound_mu_greater_eps} along with Theorem \ref{th: deriv_bounds_enpinn} and Remark \ref{remark: deriv_bounds_enpinn}, we obtain}
    \begin{equation*}
    \begin{aligned}
        \frac{1}{2}\frac{d}{dt}\displaystyle\int_{\Omega_s} |\mathfrak{e}(x,y,t)|^2 dxdy \leq  &\left(\int_{\partial\Omega_s}|\mathcal{N}_{spb,{}{{p}^*}}|^2 \, ds\right)^{\frac{1}{2}}\mathcal{K}_{\partial \Omega_s}^{\nabla \mathfrak{e},l}(u,\mathfrak{e}) +\left(\frac{2\alpha + 1}{2}\right)\int_{\Omega_s}|\mathfrak{e}|^2 \, dxdy\\
        & +{\mu \alpha}\int_{\partial\Omega_s} |\mathcal{N}_{spb,{}{{p}^*}}|^2 \, dxdy +
         \frac{1}{2}\int_{\Omega_s} {|\mathcal{N}_{int,{}{{p}^*}}(x,y,t)|^2} \, dxdy.
    \end{aligned}
    \end{equation*}
        Integrating both sides over $[0,\overline{T}],$ $\overline{T}\leq T$, we get
        \begin{align*}
            \int_{\Omega_s} |\mathfrak{e}(x,y,\overline{T})|^2 dxdy  \leq &\int_{\Omega_s}|\mathcal{N}_{ini,{{}{{p}^*}}}(x,y)|^2\, dxdy+ \left(\int_0^T\int_{\partial\Omega_s}|\mathcal{N}_{spb,{}{{p}^*}}|^2 \, ds dt\right)^{\frac{1}{2}}T^{\frac{1}{2}}\mathcal{K}_{\partial \Omega_s}^{\nabla \mathfrak{e},l}(u,\mathfrak{e}) \\
            &+(2\alpha + 1)\int_0^{\overline{T}}\int_{\Omega_s}|\mathfrak{e}(x,y,t)|^2 \, dxdy dt+ 2{\alpha}\int_0^T\int_{\partial\Omega_s} |\mathcal{N}_{spb,{{p}^*}}|^2 \, dxdy \\
            &+ \int_0^T\int_{\Omega_s} |\mathcal{N}_{int,{}{{p}^*}}(x,y,t)|^2 \, dxdy dt.
        \end{align*}
        For convenience, we write $\mathcal{K}_{\partial \Omega_s}^{\nabla \mathfrak{e},l}(u,\mathfrak{e})$ as $ \mathcal{K}_{\partial \Omega_s}^{\nabla \mathfrak{e},l}$ in the remaining proof.
        Applying Gronwall's inequality yields,
        \begin{equation}\label{eq: final_eq_th_before_conv_L2}
            \begin{aligned}
                \int_{\Omega_s}  &|\mathfrak{e}(x,y,\overline{T})|^2 dxdy \\ 
                & \leq  e^{(2\alpha + 1)T}\times \left(\int_{\Omega_s}|\mathcal{N}_{ini,{{p}^*}}(x,y)|^2\, dxdy +\left({T}\int_0^T\int_{\partial\Omega_s}|\mathcal{N}_{spb,p^*}|^2 \, ds dt\right)^{\frac{1}{2}}\mathcal{K}_{\partial \Omega_s}^{\nabla\mathfrak{e},l}\right)\\
                & \quad + e^{(2\alpha + 1)T}\times \left({2 \alpha}\int_0^T\int_{\partial\Omega_s} |\mathcal{N}_{spb,{{p}^*}}|^2 \, dxdy + \int_0^T\int_{\Omega_s} |\mathcal{N}_{int,p^*}(x,y,t)|^2 \, dxdy dt\right).
            \end{aligned}
        \end{equation}
    \end{proof}

\noindent\textbf{Theorem \ref{th:gronwall_genearlization_err_energy_norm}:}
     Let $u\in {H^3(\Omega) \cap C^0(\overline{\Omega})}$ satisfies \eqref{eq: main_cont_prob_lin} {for $\mu^2 \geq \zeta \varepsilon/\alpha_0$} and $u_{N,{{p}^*}}$ be ENPINN approximation of u obtained through ENPINN. Then the energy norm-based error will satisfy the following inequality:
    \begin{equation*}\label{appen_gronwall_genearlization_err_energy_norm}
        \|\mathfrak{e}\|_{L^2(\Omega)}^2+ \varepsilon\|\nabla \mathfrak{e}\|_{L^2(\Omega)}^2
        \leq \mathcal{G}_{L_2} e^{\kappa_{L_2}}+\mathcal{G}_{en} e^{\kappa_{en}},
    \end{equation*}
     where $T_1 \leq T, \text{and }T \in (0,1]$, $$\kappa_{en}= \left(\frac{ 6  \alpha+3\beta +2}{T_1}\right)T, ~~ \mathcal{G}_{en}= \|\mathcal{N}_{int_x,{{p}^*}}^{\varepsilon}(x,y,t)\|^2_{L^2(\Omega)}+ \|\mathcal{N}_{int_y,{{p}^*}}^{\varepsilon}(x,y,t)\|^2_{L^2(\Omega)}+ 2 \beta\mathcal{E}_{L_2}^{G}$$ and the quantities $\mathcal{G}_{L_2}, e^{\kappa_{L_2}}$ are defined in Theorem \ref{th:gronwall_genearlization_err_l^2_norm}.
    \begin{proof}
    {The generalization error bound for $\|\mathfrak{e}\|_{L^2(\Omega)}$ is mentioned in Theorem \ref{th:generalization_err_l^2_norm}. To compute the corresponding generalization error for energy norm, we just need to identify the term associated with $\varepsilon\|\nabla \mathfrak{e}\|_{L^2(\Omega)}.$}
    To proceed, we rewrite the residual error equation \eqref{difference_con_prob} as
        $$\mathfrak{e}_t=\varepsilon (\mathfrak{e}_{xx}+\mathfrak{e}_{yy})-\mu (a_1\mathfrak{e}_x+a_2\mathfrak{e}_y)-b\mathfrak{e}+\mathcal{N}_{int,{}{{p}^*}}(x,y,t).$$
        Taking partial derivative with respect to $x$ and $y$, we obtain 
        \begin{align*}
            \mathfrak{e}_{\mathbb{n}t}&=\varepsilon \partial_{\mathbb{n}}\Delta \mathfrak{e}-\mu \left[(\partial_{\mathbb{n}}a_{1})\mathfrak{e}_x+a_{1}\mathfrak{e}_{x{\mathbb{n}}}+(\partial_ {\mathbb{n}}a_{2})\mathfrak{e}_y+a_2\mathfrak{e}_{y{\mathbb{n}}}\right]-(\partial_{\mathbb{n}}b)\mathfrak{e} - b\mathfrak{e}_{\mathbb{n}} +\mathcal{N}_{int_{\mathbb{n}}, {}{{p}^*}}(x,y,t),\\
            & {\mbox{where }} \mathbb{n}=x,y.
        \end{align*}
        Multiply $\mathfrak{e}_{xt}$ with $\mathfrak{e}_x$ and $\mathfrak{e}_{yt}$ by $\mathfrak{e}_y$, and then summing both, yields
        \begin{align*}
            \mathfrak{e}&_x\mathfrak{e}_{xt}+\mathfrak{e}_y\mathfrak{e}_{yt}\\
            &=\varepsilon (\mathfrak{e}_x\mathfrak{e}_{xxx}+\mathfrak{e}_x\mathfrak{e}_{yyx}+\mathfrak{e}_y\mathfrak{e}_{xxy}+\mathfrak{e}_y\mathfrak{e}_{yyy})\\
            &\quad- \mu \left(a_{1_x}(\mathfrak{e}_x)^2+a_{1}\mathfrak{e}_x\mathfrak{e}_{xx}+a_{2_x}\mathfrak{e}_x\mathfrak{e}_y+a_2\mathfrak{e}_x\mathfrak{e}_{yx}+a_{1_y}\mathfrak{e}_x\mathfrak{e}_y+a_{1}\mathfrak{e}_{xy}\mathfrak{e}_y+a_{2_y}(\mathfrak{e}_y)^2+a_2\mathfrak{e}_{yy}\mathfrak{e}_y\right)\\
            &\quad -\left(b_x\mathfrak{e}_x\mathfrak{e}+b(\mathfrak{e}_x)^2+b_y\mathfrak{e}\mathfrak{e}_y+b(\mathfrak{e}_y)^2\right)+\mathfrak{e}_x\mathcal{N}_{int_x,{{p}^*}}(x,y,t)+\mathfrak{e}_y\mathcal{N}_{int_y,{{p}^*}}(x,y,t).
        \end{align*}
        Now we have
        \begin{equation*}
            \mu \int_{\Omega_s} (a_{1}\mathfrak{e}_x\mathfrak{e}_{xx} + a_{2}\mathfrak{e}_x\mathfrak{e}_{yx} + a_{1}\mathfrak{e}_y\mathfrak{e}_{xy} +
                a_{2}\mathfrak{e}_y\mathfrak{e}_{yy})\, dx dy = \dfrac{\mu}{2} \int_{\Omega_s} a. \nabla(|\nabla e|^2) \, dx dy
        \end{equation*}
        Using the divergence theorem, we get 
        \begin{equation}\label{eq: en_prob_mu_in_convec}
            \begin{aligned}
                \mu \int_{\Omega_s} (a_{1}\mathfrak{e}_x\mathfrak{e}_{xx} + a_{2}\mathfrak{e}_x\mathfrak{e}_{yx} + a_{1}\mathfrak{e}_y\mathfrak{e}_{xy} +
                a_{2}\mathfrak{e}_y\mathfrak{e}_{yy})\, dx dy &= \dfrac{\mu}{2} \int_{\partial\Omega_s} a. \hat{n}(|\nabla e|^2) \, dx dy\\
                & \quad - \dfrac{\mu}{2} \int_{\Omega_s} (\nabla.a)(|\nabla e|^2) \, dx dy.
            \end{aligned} 
        \end{equation}
        We now integrate over the spatial domain $\Omega_s$ and apply the coefficients bounds in \eqref{bound_a_b} to obtain
        \begin{align*}
            \int_{\Omega_s}\left(\mathfrak{e}_x\mathfrak{e}_{xt}+\mathfrak{e}_y\mathfrak{e}_{yt}\right)\, dxdy\leq & \underbrace{\varepsilon \int_{\Omega_s}\nabla \mathfrak{e}.(\nabla(\Delta \mathfrak{e}))\, dxdy}_{I_1}\\
            &+\underbrace{\mu \alpha\left[ \int_{\Omega_s}\left((\mathfrak{e}_x)^2+(\mathfrak{e}_y)^2\right) \, dxdy + 2 \int_{\Omega_s} \mathfrak{e}_x\mathfrak{e}_y \, dxdy\right]}_{I_{21}}\\
            &-\underbrace{\mu\int_{\Omega_s}(a_{1}\mathfrak{e}_x\mathfrak{e}_{xx} + a_{2}\mathfrak{e}_x\mathfrak{e}_{yx} + a_{1}\mathfrak{e}_y\mathfrak{e}_{xy} +
                a_{2}\mathfrak{e}_y\mathfrak{e}_{yy}) \, dxdy}_{I_{22}}\\
            &+\underbrace{\beta\left[\int_{\Omega_s}\left((\mathfrak{e}_x)^2+(\mathfrak{e}_y)^2)\right)\, dxdy + \int_{\Omega_s}\left(\mathfrak{e}_x\mathfrak{e}+\mathfrak{e}\mathfrak{e}_y\right)\, dxdy\right]}_{I_3}\\
            &+\underbrace{\int_{\Omega_s}\left(\mathfrak{e}_x\mathcal{N}_{int_x,{{p}^*}}(x,y,t)+\mathfrak{e}_y\mathcal{N}_{int_y,{{p}^*}}(x,y,t)\right)\, dxdy}_{I_4}.
        \end{align*}
        From the previous equation, we get 
        \begin{equation*}
            \frac{1}{2}\frac{d}{dt}\left(|\mathfrak{e}|^2_{H^1(\Omega_s)}\right)=\frac{1}{2}\frac{d}{dt}\int_{\Omega_s}\left(\mathfrak{e}_x^{^2}+\mathfrak{e}_y^{^2}\right) \, dxdy {}{\leq} I_1+{I_{21}- I_{22}} +I_3+I_4,
        \end{equation*}
        with the terms $I_i,$ for $i=1,3,4,21,22$ representing the individual integrals described above. We now analyze each term separately.\\
        \textbf{Estimate for $I_1$:}
        Applying integration by parts on $I_1$, we get
        \begin{align*}
            {\varepsilon}I_1=&\varepsilon^2 \int_{\partial\Omega_s}\Delta \mathfrak{e}\left(\nabla \mathfrak{e}.\hat{n}\right)\, ds-\varepsilon ^2\int_{\Omega_s}|\Delta \mathfrak{e}|^2 \, dxdy\\
            \leq&\varepsilon^2 \int_{\partial\Omega_s}(\mathfrak{e}_{xx}+\mathfrak{e}_{yy})(\mathfrak{e}_xn_1+\mathfrak{e}_yn_2) \, ds,\\
            \leq & \left(\int_{\partial\Omega_s}\varepsilon^2|\mathfrak{e}_x\mathfrak{e}_{xx}| ds + \int_{\partial\Omega_s}\varepsilon^2|\mathfrak{e}_y\mathfrak{e}_{xx}|ds + \int_{\partial\Omega_s}\varepsilon^2|\mathfrak{e}_x\mathfrak{e}_{yy}|\, ds + \int_{\partial\Omega_s}\varepsilon^2|\mathfrak{e}_y\mathfrak{e}_{yy}|\, ds\right).
        \end{align*}
        {Using the bounds in \eqref{eq: exact_sol_derivative_bound_mu_less_eps}-\eqref{eq: exact_sol_derivative_bound_mu_greater_eps}, and Theorem \ref{th: deriv_bounds_enpinn} with Remark \ref{remark: deriv_bounds_enpinn}, we derive }
        \begin{equation}\label{I_1_value_gen_err_thm}
            I_1 \leq \mathcal{K}_{\partial \Omega_s}^{I_1}(u,u_{N,{{p}^*}}).
        \end{equation}
        \textbf{Estimation for $ {\varepsilon I_2=(I_{21}- I_{22})}$:}
        Applying {\eqref{eq: en_prob_mu_in_convec}} and Young's inequality in {$I_{2}$}, we get
        \begin{equation*}
            {\varepsilon} I_2 \leq  \varepsilon \alpha\left[ 3\int_{\Omega_s}\left((\mathfrak{e}_x)^2+(\mathfrak{e}_y)^2\right) \, dxdy  \right].
        \end{equation*}
        Thus, we get
        \begin{equation}\label{I_2_value_gen_err_thm}
            {\varepsilon} I_2 \leq 3 \varepsilon\alpha \int_{\Omega_s}|\nabla \mathfrak{e}|^2 \, dxdy.
        \end{equation}
        Note that the constant $\mathcal{K}_{\partial\Omega_s}^{I_1}(u,u_{N,{{p}^*}})$ are non negative.\\
        \textbf{Estimation for $ \varepsilon I_3$:}
        Using Cauchy-Schwarz and Young inequalities
        \begin{align*}
            \varepsilon I_3=& \beta \varepsilon\left[\int_{\Omega_s}|\nabla \mathfrak{e}|^2\, dxdy + \int_{\Omega_s}\left(\mathfrak{e}_x\mathfrak{e}+\mathfrak{e}\mathfrak{e}_y\right)\ dxdy\right]\\
            \leq & \beta \varepsilon\int_{\Omega_s}|\nabla \mathfrak{e}|^2\, dxdy + {\beta } \int_{\Omega_s} |\mathfrak{e}|^2 \, dxdy+ \frac{\beta \varepsilon}{2} \int_{\Omega_s} \left(\mathfrak{e}_x^2+ \mathfrak{e}_y^2\right) \, dxdy.
        \end{align*}
        Thus, the estimation follows 
        \begin{equation}\label{I_3_value_gen_err_thm}
            \varepsilon I_3 \leq \frac{3\beta \varepsilon}{2} \int_{\Omega_s}|\nabla \mathfrak{e}|^2\, dxdy+{\beta} \int_{\Omega_s} |\mathfrak{e}|^2 \, dxdy.
        \end{equation}
        \textbf{Estimation for $\varepsilon I_4$:}
        \begin{align*}
            \varepsilon I_4=& \varepsilon \int_{\Omega_s}\left(\mathfrak{e}_x\mathcal{N}_{int_x,{}{{p}^*}}(x,y,t)+\mathfrak{e}_y\mathcal{N}_{int_y,{}{{p}^*}}(x,y,t)\right)\, dxdy\\
            \leq& \frac{1}{2} \int_{\Omega_s} \left|\sqrt{\varepsilon}~ \mathcal{N}_{int_x,{}{{p}^*}}(x,y,t)\right|^2 \, dxdy+ \frac{1}{2} \int_{\Omega_s} \left|\sqrt{\varepsilon}~\mathcal{N}_{int_y,{}{{p}^*}}(x,y,t)\right|^2 \, dxdy+ \frac{\varepsilon}{2} \int_{\Omega_s} (\mathfrak{e}_x^2+ \mathfrak{e}_y^2) \, dxdy.
        \end{align*}
        We will use the notation $\sqrt{\varepsilon}~\mathcal{N}_{int_{\mathbb{n}}}= \mathcal{N}_{int_{\mathbb{n}}}^{\varepsilon}, \mathbb{n}=x,y$ and $ \mathcal{K}_{\partial \Omega_s}^{I_1}(u,u_{N,{}{{p}^*}}) \equiv \mathcal{K}_{\partial \Omega_s}^{I_1},$  as our convenience, for the remaining of the proof.
        Therefore, applying standard inequalities follows
        \begin{equation}\label{I_4_value_gen_err_thm}
            \varepsilon I_4 \leq \frac{1}{2} \int_{\Omega_s} \left(\left| \mathcal{N}_{int_x,{}{{p}^*}}^{\varepsilon}(x,y,t)\right|^2 + \left|\mathcal{N}_{int_y,{}{{p}^*}}^{\varepsilon}(x,y,t)\right|^2\right) \, dxdy+ \frac{\varepsilon}{2} \int_{\Omega_s}|\nabla \mathfrak{e}|^2 \, dxdy.
        \end{equation}
        Using all estimates (\ref{I_1_value_gen_err_thm})-(\ref{I_4_value_gen_err_thm}) we obtain
        \begin{equation}\label{half_err_H1_norm}
            \begin{aligned}
            \frac{\varepsilon}{2}\frac{d}{dt}\int_{\Omega_s}|\nabla \mathfrak{e}|^2 \, dxdy \leq & \mathcal{K}_{\partial \Omega_s}^{I_1}+ 3 \alpha \varepsilon \int_{\Omega_s}|\nabla \mathfrak{e}|^2 \, dxdy+ \frac{3\beta \varepsilon}{2} \int_{\Omega_s}|\nabla \mathfrak{e}|^2\, dxdy+{\beta } \int_{\Omega_s} |\mathfrak{e}|^2 \, dxdy\\
            &+ \frac{1}{2} \int_{\Omega_s} \left(|\mathcal{N}_{int_x,{}{{p}^*}}^{\varepsilon}(x,y,t)|^2 + |\mathcal{N}_{int_y,{}{{p}^*}}^{\varepsilon}(x,y,t)|^2\right) \, dxdy+ \frac{\varepsilon}{2} \int_{\Omega_s}|\nabla \mathfrak{e}|^2 \, dxdy.
        \end{aligned}
        \end{equation}
        Now, we get
        \begin{align*}
            \frac{\varepsilon}{2}\frac{d}{dt}\int_{\Omega_s}|\nabla \mathfrak{e}|^2 \, dxdy  
            \leq & \beta \int_{\Omega_s}|\mathfrak{e}|^2 \, dxdy
            +\left(\frac{6 \alpha+3\beta +1}{2}\right)\varepsilon\int_{\Omega_s}|\nabla \mathfrak{e}|^2 \, dxdy
            + \mathcal{K}_{\partial \Omega_s}^{I_1}\\
            &+ \frac{1}{2} \int_{\Omega_s} \left(|\mathcal{N}_{int_x,{}{{p}^*}}^{\varepsilon}(x,y,t)|^2 + |\mathcal{N}_{int_y,{}{{p}^*}}^{\varepsilon}(x,y,t)|^2\right) \, dxdy.
        \end{align*}
        Choose $T_1\leq \overline{T}\leq T$, $T\in(0,1)$ such that
        \begin{align*}
             {\varepsilon}&\frac{d}{dt}\int_{\Omega_s}|\nabla \mathfrak{e}|^2 \, dxdy+ \frac{\mathcal{K}_{\partial \Omega_s}^{I_1}}{\overline{T}}\\
             \leq & \left(\frac{6 \alpha+3\beta +2}{T_1}\right)\left[\varepsilon\int_{\Omega_s}|\nabla \mathfrak{e}|^2 \, dxdy+ \mathcal{K}_{\partial \Omega_s}^{I_1}\right]+ 2\beta \int_{\Omega_s}{|\mathfrak{e}|^2} \, dxdy\\
             &+  \int_{\Omega_s} \left(|\mathcal{N}_{int_x,{}{{p}^*}}^{\varepsilon}(x,y,t)|^2 + |\mathcal{N}_{int_y,{}{{p}^*}}^{\varepsilon}(x,y,t)|^2\right) \, dxdy.
        \end{align*}
        Integrating both sides with respect to $t\in (0,\overline{T}]$, we get
        \begin{align*}
            {\varepsilon}\int_{\Omega_s}&|\nabla \mathfrak{e}(x,y,\overline{T})|^2 \, dxdy+ \int_0^{\overline{T}}\frac{\mathcal{K}_{\partial \Omega_s}^{I_1}}{\overline{T}}\, dt\\
            \leq & \varepsilon\int_{\Omega_s}|\nabla \mathfrak{e}(x,y,0)|^2 \, dxdy
            + \int_0^{T} \int_{\Omega_s} \left(|\mathcal{N}_{int_x,{}{{p}^*}}^{\varepsilon}(x,y,t)|^2 + |\mathcal{N}_{int_y,{}{{p}^*}}^{\varepsilon}(x,y,t)|^2\right) \, dxdy dt\\
             &+\left(\frac{6  \alpha +3\beta +2}{T_1}\right)\left[\varepsilon\int_0^{\overline{T}}\int_{\Omega_s}|\nabla \mathfrak{e}|^2 \, dxdydt+ \int_0^{\overline{T}}\mathcal{K}_{\partial \Omega_s}^{I_1}\, dt\right]+2 \beta \int_0^{\overline{T}}\int_{\Omega_s}{|\mathfrak{e}|^2} \, dxdy dt.
        \end{align*}
        This gives
        \begin{align*}
            {\varepsilon}\int_{\Omega_s}&|\nabla \mathfrak{e}(x,y,\overline{T})|^2 \, dxdy+ {\mathcal{K}_{\partial \Omega_s}^{I_1}}+\varepsilon\int_{\Omega_s}|\nabla \mathfrak{e}(x,y,0)|^2 \, dxdy\\
            \leq &\left(\frac{6 \alpha+3\beta +2}{T_1}\right) \times 
            \int_0^{\overline{T}} \left[ \varepsilon \int_{\Omega_s}|\nabla \mathfrak{e}|^2 \, dxdy+ \mathcal{K}_{\partial \Omega_s}^{I_1}\, + \varepsilon \int_{\Omega_s}|\nabla \mathfrak{e}(x,y,0)|^2 \, dxdy\,\right]dt\\
            &+  \int_0^{T} \int_{\Omega_s} \left(|\mathcal{N}_{int_x,{}{{p}^*}}^{\varepsilon}(x,y,t)|^2 + |\mathcal{N}_{int_y,{}{{p}^*}}^{\varepsilon}(x,y,t)|^2\right) \, dxdy dt+ 2 \beta\mathcal{E}_{L_2}^{G}
        \end{align*}
        By Gronwall's inequality, we get
        \begin{equation*}
            {\varepsilon}\int_{\Omega_s}|\nabla \mathfrak{e}(x,y,\overline{T})|^2 \, dxdy+ {\mathcal{K}_{\partial \Omega_s}^{I_1}}+ \varepsilon\int_{\Omega_s}|\nabla \mathfrak{e}(x,y,0)|^2 \, dxdy
            \leq e^{\left(\frac{6  \alpha+3\beta +2}{T_1}\right)T}\times \mathcal{G}_{en},
        \end{equation*}
        where 
        \begin{equation*}
            \mathcal{G}_{en}= \int_0^{T} \int_{\Omega_s} \left(|\mathcal{N}_{int_x,{}{{p}^*}}^{\varepsilon}(x,y,t)|^2 + |\mathcal{N}_{int_y,{}{{p}^*}}^{\varepsilon}(x,y,t)|^2\right) \, dxdy dt+ 2 \beta\mathcal{E}_{L_2}^{G}.
        \end{equation*}
        Since ${\mathcal{K}_{\partial \Omega_s}^{I_1}}\geq 0$, we have
        \begin{align*}
            {\varepsilon}\int_{\Omega_s}|\nabla \mathfrak{e}(x,y,\overline{T})|^2 \, dxdy
            \leq & {\varepsilon}\int_{\Omega_s}|\nabla \mathfrak{e}(x,y,\overline{T})|^2 \, dxdy+ {\mathcal{K}_{\partial \Omega_s}^{I_1}}+ \varepsilon\int_{\Omega_s}|\nabla \mathfrak{e}(x,y,0)|^2 \, dxdy\\
            \leq & \mathcal{G}_{en} \times e^{\left(\frac{6  \alpha+3\beta +2}{T_1}\right)T}.
        \end{align*}
        This gives
        \begin{equation}\label{eq: final_energy_inequality}
            {\varepsilon}\int_{\Omega_s}|\nabla \mathfrak{e}(x,y,\overline{T})|^2 \, dxdy
            \leq  \mathcal{G}_{en} \times e^{\left(\frac{ 6  \alpha+3\beta +2}{T_1}\right)T}.
        \end{equation}
        {}{Hence the required estimate follows from Theorem \ref{th:gronwall_genearlization_err_l^2_norm} and the inequality \eqref{eq: final_energy_inequality}.}
    \end{proof}

\section{Stability for System of Equations with ENPINN Loss}\label{sec: appen_Stability for System}

\noindent\textbf{Theorem \ref{th: sys_gronwall_genearlization_err_l^2_norm}:}
    Let {$u_s \in {H^2(\Omega) \cap C^0(\overline{\Omega})}$} denote the exact solution of \eqref{eq: system_equation}. Let $u_{s_N,p^*}$ be the approximation of $u_s$.
    Then the generalization error satisfies the following inequality
    \begin{equation*}\label{appen_gronwall_genearlization_err_l^2_norm}
        \|\mathfrak{e}_1\|_{{L^2(\Omega)}}^2 + \|\mathfrak{e}_2\|_{{L^2(\Omega)}}^2= \|u_1-u_{1_N,{{p}^*}}\|_{{L^2(\Omega)}}^2 + \|u_2-u_{2_N,{{p}^*}}\|_{{L^2(\Omega)}}^2 \leq e^{\kappa_{L_2}^s}\times \mathcal{G}_{L_2}^s,
    \end{equation*}
    where $\kappa_{L_2}^s={(2\beta_3 + 1)T}$, $\mathcal{G}_{L_2}^s = \|\mathcal{N}_{ini,{p}^*}^1\|_{L^2(\Omega_s)}^2 + \|\mathcal{N}_{ini,{p}^*}^2\|_{L^2(\Omega_s)}^2 + T^{1/2} \|\mathcal{N}_{spb,{p}^*}^1\|_{L^2(\Omega)}  \mathcal{K}_{\partial \Omega_s}^{\nabla \mathfrak{e}_1} + T^{1/2} \|\mathcal{N}_{spb,{p}^*}^2\|_{L^2(\Omega)}  \mathcal{K}_{\partial \Omega_s}^{\nabla \mathfrak{e}_2} + \|\mathcal{N}_{int,{p}^*}^1\|_{L^2(\Omega)}^2 + \|\mathcal{N}_{int,{p}^*}^2\|_{L^2(\Omega)}^2$. The positive constant $\mathcal{K}_{\partial\Omega_s}^{\nabla \mathfrak{e}_i}$ is independent of $\varepsilon$ and depends on the continuous solution $u_i$, the NN approximation $u_{i_N,{{p}^*}}$, $i.e.,$ 
    $$\mathcal{K}_{\partial\Omega_s}^{\nabla \mathfrak{e}_i}\equiv \mathcal{K}_{\partial\Omega_s}^{\nabla \mathfrak{e}_i}(u_i,u_{i_N,{{p}^*}}), \quad i=1,2.$$

    \begin{proof}
        Using \eqref{eq: system_equation}, we obtain the following error equations of $\mathfrak{e}_1$ and $\mathfrak{e}_2$:
        \begin{equation}\label{sys_difference_con_prob}
            \begin{cases}
                \mathfrak{e}_{1_t}=\varepsilon \Delta \mathfrak{e}_1+ {a_{11}}{ \mathfrak{e}_{1,x}} + {a_{21}}{ \mathfrak{e}_{1,y}} - b_{11} \mathfrak{e}_1 - b_{12} \mathfrak{e}_2 + \mathcal{N}_{res, p^*}^1, \quad (x,y,t) \in \Omega,\\
                \mathfrak{e}_{2_t}=\mu \Delta \mathfrak{e}_2+ {a_{12}}{ \mathfrak{e}_{2,x}} + {a_{22}}{ \mathfrak{e}_{2,y}} - b_{21} \mathfrak{e}_1 - b_{22} \mathfrak{e}_2 + \mathcal{N}_{res, p^*}^2, \quad (x,y,t) \in \Omega,\\
                \mathfrak{e}_1(x,y,t) =\mathcal{N}_{spb, p^*}^1, \quad  \mathfrak{e}_2(x,y,t) =\mathcal{N}_{spb, p^*}^2, \quad (x,y)\in \partial \Omega_s, ~ t\in [0,T], \\
                \mathfrak{e}_1(x,y,0)=\mathcal{N}_{ini, p^*}^1, \quad \mathfrak{e}_2(x,y,0)=\mathcal{N}_{ini, p^*}^2, \quad (x,y)\in \overline{\Omega}_s.
            \end{cases}
        \end{equation}
        Multiplying the first equation in \eqref{sys_difference_con_prob} by $\mathfrak{e}_1$ and integrate over $\Omega_s$, we get  
    \begin{align*}
        \frac{1}{2}\frac{d}{dt}\displaystyle\int_{\Omega_s} |\mathfrak{e}_1(x,y,t)|^2 dxdy = &\, \varepsilon\int_{\Omega_s}\mathfrak{e}_1\Delta \mathfrak{e}_1 \, dxdy-a_{11}\int_{\Omega_s}\mathfrak{e}_1\mathfrak{e}_{1_x} \, dxdy\\
        &-a_{21}\int_{\Omega_s}\mathfrak{e}_1\mathfrak{e}_{1_y} \, dxdy - \beta_0\int_{\Omega_s} \mathfrak{e}_1^2 \, dxdy \\
        &- \beta_2\int_{\Omega_s} \mathfrak{e}_1 \mathfrak{e}_2 \, dxdy + \int_{\Omega_s} \mathfrak{e}_1\mathcal{N}_{int,{{p}^*}}^1 \, dxdy.
    \end{align*}
    Using integration by parts, Young's inequality, and the bounds from \cite{kumar2025uniformly}, we obtain
    \begin{align*}
        \frac{1}{2}&\frac{d}{dt}\displaystyle\int_{\Omega_s} |\mathfrak{e}_1(x,y,t)|^2 dxdy \\
         \leq & \, \varepsilon\int_{\partial\Omega_s}\mathfrak{e}_1\left(\nabla \mathfrak{e}_1. \hat{n}\right) \, ds - \varepsilon \int_{\Omega_s} |\nabla \mathfrak{e}_1|^2 \, dxdy -\dfrac{1}{2}\int_{\Omega_s}{(a_{11}, a_{21}).\nabla (\mathfrak{e}^2)\, dxdy} + \frac{\beta_3}{2}\int_{\Omega_s} |\mathfrak{e}_1|^2 \, dxdy\\
        & + \frac{\beta_3}{2}\int_{\Omega_s} |\mathfrak{e}_2|^2 \, dxdy + \frac{1}{2}\int_{\Omega_s} |\mathfrak{e}_1|^2 \, dxdy + \frac{1}{2}\int_{\Omega_s} |\mathcal{N}_{int,{{p}^*}}^1(x,y,t)|^2 \, dxdy.
    \end{align*}
    Integrating both sides over $[0,\overline{T}],$ $\overline{T}\leq T$, we get
    \begin{equation}\label{eq: sys_eq_err1}
        \begin{aligned}
            \int_{\Omega_s}& |\mathfrak{e}_1(x,y,\overline{T})|^2 dxdy  \\
            \leq &\int_{\Omega_s}|\mathcal{N}_{ini,{p}^*}^1(x,y)|^2\, dxdy + \left(\int_0^T\int_{\partial\Omega_s}|\mathcal{N}_{spb,{p}^*}^1(x,y,t)|^2 \, ds dt\right)^{\frac{1}{2}}T^{\frac{1}{2}}\mathcal{K}_{\partial \Omega_s}^{\nabla \mathfrak{e}_1}(u,\mathfrak{e}_1) \\
            &+(\beta_3 + 1)\int_0^{\overline{T}}\int_{\Omega_s}|\mathfrak{e}_1(x,y,t)|^2 \, dxdy dt + \beta_3 \int_0^{\overline{T}} \int_{\Omega_s}|\mathfrak{e}_2(x,y,t)|^2 \, dxdy dt\\
            &+ \int_0^T\int_{\Omega_s} |\mathcal{N}_{int,{p}^*}^1(x,y,t)|^2 \, dxdy dt.
        \end{aligned}
    \end{equation}
    Similarly, from the 2nd equation of \eqref{sys_difference_con_prob}, and from \eqref{eq: sys_eq_err1}, we get
    \begin{equation}\label{eq: sys_eq_err2}
        \begin{aligned}
            \int_{\Omega_s} &|\mathfrak{e}_2(x,y,\overline{T})|^2 dxdy  \\
            \leq &\int_{\Omega_s}|\mathcal{N}_{ini,{p}^*}^2(x,y)|^2\, dxdy + \left(\int_0^T\int_{\partial\Omega_s}|\mathcal{N}_{spb,{p}^*}^2(x,y,t)|^2 \, ds dt\right)^{\frac{1}{2}}T^{\frac{1}{2}}\mathcal{K}_{\partial \Omega_s}^{\nabla \mathfrak{e}_2}(u,\mathfrak{e}_2) \\
            &+(\beta_3 + 1)\int_0^{\overline{T}}\int_{\Omega_s}|\mathfrak{e}_2(x,y,t)|^2 \, dxdy dt + \beta_3 \int_0^{\overline{T}} \int_{\Omega_s}|\mathfrak{e}_2(x,y,t)|^2 \, dxdy dt\\
            &+ \int_0^T\int_{\Omega_s} |\mathcal{N}_{int,{p}^*}^2(x,y,t)|^2 \, dxdy dt.
        \end{aligned}
    \end{equation}
    Adding \eqref{eq: sys_eq_err1} and \eqref{eq: sys_eq_err2}, we get
    \begin{equation*}
        \begin{aligned}
            \int_{\Omega_s} &|\mathfrak{e}_1(x,y,\overline{T})|^2 dxdy + \int_{\Omega_s} |\mathfrak{e}_2(x,y,\overline{T})|^2 dxdy \\
            \leq & \int_{\Omega_s}|\mathcal{N}_{ini,{p}^*}^1(x,y)|^2\, dxdy + \int_{\Omega_s}|\mathcal{N}_{ini,{p}^*}^2(x,y)|^2\, dxdy \\
            &+ \left(\int_0^T\int_{\partial\Omega_s}|\mathcal{N}_{spb,{p}^*}^1(x,y,t)|^2 \, ds dt\right)^{\frac{1}{2}}T^{\frac{1}{2}}\mathcal{K}_{\partial \Omega_s}^{\nabla \mathfrak{e}_1}(u,\mathfrak{e}_1) \\
            &+ \left(\int_0^T\int_{\partial\Omega_s}|\mathcal{N}_{spb,{p}^*}^2(x,y,t)|^2 \, ds dt\right)^{\frac{1}{2}}T^{\frac{1}{2}}\mathcal{K}_{\partial \Omega_s}^{\nabla \mathfrak{e}_2}(u,\mathfrak{e}_2) \\
            &+(2\beta_3 + 1)\int_0^{\overline{T}}\int_{\Omega_s}\left(|\mathfrak{e}_2(x,y,t)|^2 + |\mathfrak{e}_1(x,y,t)|^2\right) \, dxdy dt \\
            &+ \int_0^T\int_{\Omega_s} |\mathcal{N}_{int,{p}^*}^1(x,y,t)|^2 \, dxdy dt  + \int_0^T\int_{\Omega_s} |\mathcal{N}_{int,{p}^*}^2(x,y,t)|^2 \, dxdy dt.
        \end{aligned}
    \end{equation*}
    Applying Gronwall's inequality yields,
        \begin{equation*}
            \begin{aligned}
                \int_{\Omega_s} &|\mathfrak{e}_1(x,y,\overline{T})|^2 dxdy + \int_{\Omega_s} |\mathfrak{e}_2(x,y,\overline{T})|^2 dxdy  \\ 
                \leq & e^{(2\beta_3 + 1)T}\times \left(\int_{\Omega_s}|\mathcal{N}_{ini,{p}^*}^1(x,y)|^2\, dxdy +\int_{\Omega_s}|\mathcal{N}_{ini,{p}^*}^2(x,y)|^2\, dxdy \right)\\
                & + e^{(2\beta_3 + 1)T} \times \left(\left({T}\int_0^T\int_{\partial\Omega_s}|\mathcal{N}_{spb,p^*}^1(x,y,t)|^2 \, ds dt\right)^{\frac{1}{2}}\mathcal{K}_{\partial \Omega_s}^{\nabla\mathfrak{e}_1}\right) \\
                & + e^{(2\beta_3 + 1)T} \times \left(\left({T}\int_0^T\int_{\partial\Omega_s}|\mathcal{N}_{spb,p^*}^2(x,y,t)|^2 \, ds dt\right)^{\frac{1}{2}}\mathcal{K}_{\partial \Omega_s}^{\nabla\mathfrak{e}_2}\right)\\
                &+ e^{(2\beta_3 + 1)T} \times \left(\int_0^T\int_{\Omega_s} |\mathcal{N}_{int,p^*}^1(x,y,t)|^2 \, dxdy dt  + \int_0^T\int_{\Omega_s} |\mathcal{N}_{int,p^*}^2(x,y,t)|^2 \, dxdy dt \right)
            \end{aligned}
        \end{equation*}
        Now integrating over $[0, {T}]$ on both sides gives the required estimate.
    \end{proof}

\noindent\textbf{Theorem \ref{th: sys_gronwall_genearlization_err_energy_norm}:}
     Let {$u_s\in {H^3(\Omega) \cap C^0(\overline{\Omega})}$} satisfies \eqref{eq: system_equation}  and $u_{s_N,{{p}^*}}$ be the ENPINN approximation of $u_s$. Then the energy norm-based error will satisfy the following inequality:
    \begin{equation*}\label{appen_sys_gronwall_genearlization_err_energy_norm}
        \|\mathfrak{e}_1\|_{E}^2+  \|\mathfrak{e}_2\|_{E}^2
        \leq \mathcal{G}_{L_2}^s e^{\kappa_{L_2}^s}+\mathcal{G}_{en}^s e^{\kappa_{en}^s},
    \end{equation*}
     where $T_1 \leq T, \text{and }T \in (0,1]$,  $\kappa_{en}^s= \left(\frac{2\beta_3 +1}{T_1}\right)T$, and
     \begin{equation*}
         \begin{aligned}
            \mathcal{G}_{en}^s
             =&\, \|\mathcal{N}_{int_x,{p}^*}^{\varepsilon,1}(x,y,t)\|^2_{L^2(\Omega)}
             + \|\mathcal{N}_{int_y,{p}^*}^{\varepsilon,1}(x,y,t)\|^2_{L^2(\Omega)}
             + \|\mathcal{N}_{int_x,{p}^*}^{\varepsilon,2}(x,y,t)\|^2_{L^2(\Omega)}\\
             &+ \|\mathcal{N}_{int_y,{p}^*}^{\varepsilon,2}(x,y,t)\|^2_{L^2(\Omega)}
         \end{aligned}
     \end{equation*}
     and the quantities $\mathcal{G}_{L_2}^s, e^{\kappa_{L_2}^s}$ are defined in Theorem \ref{th: sys_gronwall_genearlization_err_l^2_norm}.
    \begin{proof}
        To compute the corresponding generalization error for energy norm, we just need to identify the term associated with $\varepsilon\|\nabla \mathfrak{e}_i\|_{L^2(\Omega)}$, for $i =1,2$.
        To proceed, we rewrite the first residual error equation \eqref{sys_difference_con_prob} as
        $$\mathfrak{e}_{1_t}=\varepsilon (\mathfrak{e}_{1_{xx}}+\mathfrak{e}_{1_{yy}})- a_{11}\mathfrak{e}_{1_x} - a_{21}\mathfrak{e}_{1_y}-b_{11}\mathfrak{e}_1 - b_{12}\mathfrak{e}_2 +\mathcal{N}_{int,{p}^*}^1.$$
        Multiply $\mathfrak{e}_{1_{xt}}$ with $\mathfrak{e}_{1_x}$ and $\mathfrak{e}_{1_{yt}}$ by $\mathfrak{e}_{1_y}$, and then summing both, yields
        \begin{align*}
            \mathfrak{e}_{1_x}&\mathfrak{e}_{1_{xt}}+\mathfrak{e}_{1_y}\mathfrak{e}_{1_{yt}}\\
            = & \, \varepsilon (\mathfrak{e}_{1_x}\mathfrak{e}_{1_{xxx}}+\mathfrak{e}_{1_x}\mathfrak{e}_{1_{yyx}}+\mathfrak{e}_{1_y}\mathfrak{e}_{1_{xxy}}+\mathfrak{e}_{1_y}\mathfrak{e}_{1_{yyy}})
            - a_{11}(\mathfrak{e}_{1_x} \mathfrak{e}_{1_{xx}} + \mathfrak{e}_{1_y} \mathfrak{e}_{1_{xy}}) - a_{21}(\mathfrak{e}_{1_x} \mathfrak{e}_{1_{yx}} + \mathfrak{e}_{1_y} \mathfrak{e}_{1_{yy}})\\
            & - b_{11} (\mathfrak{e}_{1_x}^2 + \mathfrak{e}_{1_y}^2) 
            - b_{12} (\mathfrak{e}_{1_x} \mathfrak{e}_{2_x} + \mathfrak{e}_{1_y} \mathfrak{e}_{2_y}) 
            +\mathfrak{e}_{1_x}\mathcal{N}_{int_x,{p}^*}^1 +\mathfrak{e}_{1_y}\mathcal{N}_{int_y,{p}^*}^2.
        \end{align*}
        Now using a similar proof as we have performed for the scalar equations in Theorem \ref{th:gronwall_genearlization_err_energy_norm}, we have
        \begin{equation*}
            \begin{aligned}
                \frac{\varepsilon}{2}\frac{d}{dt}\int_{\Omega_s}|\nabla \mathfrak{e}_1|^2 \, dxdy \leq & \mathcal{K}_{\partial \Omega_s}^{I_1^1}
                + \left(\frac{\beta_3 + 1}{2}\right) \int_{\Omega_s}|\nabla \mathfrak{e}_1|^2\, dxdy
                + \frac{\beta_3}{2} \int_{\Omega_s}|\nabla \mathfrak{e}_2|^2\, dxdy\\
                &+ \frac{1}{2} \int_{\Omega_s} \left(|\mathcal{N}_{int_x,{p}^*}^{\varepsilon,1}|^2 + |\mathcal{N}_{int_y,{p}^*}^{\varepsilon,1}|^2\right) \, dxdy.
            \end{aligned}
        \end{equation*}
        Consequently, we get the inequality for $\mathfrak{e}_2$. Adding these two inequalities gives the estimates below:
        \begin{equation*}
            \begin{aligned}
                {\varepsilon}\frac{d}{dt}&\int_{\Omega_s}|\nabla \mathfrak{e}_1|^2 \, dxdy + {\varepsilon}\frac{d}{dt}\int_{\Omega_s}|\nabla \mathfrak{e}_2|^2 \, dxdy \\
                \leq & \mathcal{K}_{\partial \Omega_s}^{I_1^1} + \mathcal{K}_{\partial \Omega_s}^{I_1^2}
                + \left({2\beta_3 + 1}\right) \int_{\Omega_s}(|\nabla \mathfrak{e}_1|^2 +
                |\nabla \mathfrak{e}_2|^2) \, dxdy\\
                &+ \frac{1}{2} \int_{\Omega_s} \left(|\mathcal{N}_{int_x,{p}^*}^{\varepsilon,1}|^2 + |\mathcal{N}_{int_y,{p}^*}^{\varepsilon,1}|^2\right) \, dxdy + \frac{1}{2} \int_{\Omega_s} \left(|\mathcal{N}_{int_x,{p}^*}^{\varepsilon,2}|^2 + |\mathcal{N}_{int_y,{p}^*}^{\varepsilon,2}|^2\right) \, dxdy.
            \end{aligned}
        \end{equation*}
        {Hence the desired estimate is obtained by integrating over $[0,\overline{T}],$ with $\overline{T}\leq T$ in the same manner as in Theorem \ref{th: sys_gronwall_genearlization_err_l^2_norm}, followed by an application of Gronwall's inequality.}
    \end{proof}

\end{document}